\documentclass[12pt]{amsart}
\usepackage{amsmath,amsfonts,amssymb,amsthm}
\usepackage[abbrev,nobysame]{amsrefs}
\usepackage{mathrsfs}
\usepackage[all,cmtip]{xy}
\usepackage{tikz-cd}
 \usepackage{relsize}
\usepackage{hyperref}
\usepackage{accents}
\usepackage{xcolor}
\usepackage{soul}\usepackage{calligra}
\usepackage{enumitem}
\usepackage{bbm}

\setlist[itemize]{align=parleft,left=0pt..1.5em}

 \allowdisplaybreaks

\DeclareMathAlphabet{\mathcalligra}{T1}{calligra}{m}{n}

\DeclareMathOperator{\Res}{Res}

\DeclareMathOperator{\Hom}{Hom}

\DeclareMathOperator{\Sing}{Sing}
\DeclareMathOperator{\Rat}{Rat}

\DeclareMathOperator{\End}{End}

\DeclareMathOperator{\diag}{diag}

\DeclareMathOperator{\obj}{obj}

\begin{document}

%--------------<Theorem Style Head>--------------
\newtheorem{thm}{Theorem}[section]
\newtheorem{prop}[thm]{Proposition}
\newtheorem{coro}[thm]{Corollary}
\newtheorem{conj}[thm]{Conjecture}
\newtheorem{example}[thm]{Example}
\newtheorem{lem}[thm]{Lemma}
\newtheorem{rem}[thm]{Remark}
\newtheorem{hy}[thm]{Hypothesis}
\newtheorem*{acks}{Acknowledgements}
\theoremstyle{definition}
\newtheorem{de}[thm]{Definition}
\newtheorem{ex}[thm]{Example}

\newtheorem{convention}[thm]{Convention}

\newtheorem{bfproof}[thm]{{\bf Proof}}
%\xymatrixcolsep{5pc}
%--------------<\Theorem Style Head>-------------

%--------------<Common Sets>---------------------
\newcommand{\C}{{\mathbb{C}}}
\newcommand{\Z}{{\mathbb{Z}}}
\newcommand{\N}{{\mathbb{N}}}
\newcommand{\Q}{{\mathbb{Q}}}
\newcommand{\te}[1]{\mbox{#1}}
\newcommand{\set}[2]{{
    \left.\left\{
        {#1}
    \,\right|\,
        {#2}
    \right\}
}}
\newcommand{\sett}[2]{{
    \left\{
        {#1}
    \,\left|\,
        {#2}
    \right\}\right.
}}

\newcommand{\choice}[2]{{
\left[
\begin{array}{c}
{#1}\\{#2}
\end{array}
\right]
}}
\def \<{{\langle}}
\def \>{{\rangle}}

\def\({\left(}

\def\){\right)}

\def \:{\mathopen{\overset{\circ}{
    \mathsmaller{\mathsmaller{\circ}}}
    }}
\def \;{\mathclose{\overset{\circ}{\mathsmaller{\mathsmaller{\circ}}}}}

\newcommand{\overit}[2]{{
    \mathop{{#1}}\limits^{{#2}}
}}
\newcommand{\belowit}[2]{{
    \mathop{{#1}}\limits_{{#2}}
}}

\newcommand{\wt}[1]{\widetilde{#1}}

\newcommand{\wh}[1]{\widehat{#1}}

\newcommand{\wck}[1]{\reallywidecheck{#1}}

\newlength{\dhatheight}
\newcommand{\dwidehat}[1]{%
    \settoheight{\dhatheight}{\ensuremath{\widehat{#1}}}%
    \addtolength{\dhatheight}{-0.45ex}%
    \widehat{\vphantom{\rule{1pt}{\dhatheight}}%
    \smash{\widehat{#1}}}}
\newcommand{\dhat}[1]{%
    \settoheight{\dhatheight}{\ensuremath{\hat{#1}}}%
    \addtolength{\dhatheight}{-0.35ex}%
    \hat{\vphantom{\rule{1pt}{\dhatheight}}%
    \smash{\hat{#1}}}}

\newcommand{\dwh}[1]{\dwidehat{#1}}

\newcommand{\dis}{\displaystyle}

\newcommand{\pd}[1]{\frac{\partial}{\partial {#1}}}

\newcommand{\pdiff}[2]{\frac{\partial^{#2}}{\partial #1^{#2}}}

%--------------<\Common Sets>--------------------

%--------------<Global>--------------------------
\newcommand{\g}{{\mathfrak g}}
\newcommand{\ff}{{\mathfrak f}}
\newcommand{\f}{\ff}
\newcommand{\gc}{{\bar{\g'}}}
\newcommand{\h}{{\mathfrak h}}
\newcommand{\cent}{{\mathfrak c}}
\newcommand{\notc}{{\not c}}
\newcommand{\Loop}{{\mathcal L}}
\newcommand{\G}{{\mathcal G}}
\newcommand{\D}{\mathcal D}
\newcommand{\T}{\mathcal T}
\newcommand{\Free}{\mathcal F}
\newcommand{\Cfk}{\mathcal C}
\newcommand{\nil}{\mathfrak n}
\newcommand{\al}{\alpha}
\newcommand{\be}{\beta}
\newcommand{\beck}{\be^\vee}
\newcommand{\ssl}{{\mathfrak{sl}}}
\newcommand{\id}{\te{id}}
\newcommand{\rtu}{{\xi}}
\newcommand{\period}{{N}}
\newcommand{\half}{{\frac{1}{2}}}
\newcommand{\reciprocal}[1]{{\frac{1}{#1}}}
\newcommand{\inverse}{^{-1}}
\newcommand{\inv}{\inverse}
\newcommand{\SumInZm}[2]{\sum\limits_{{#1}\in\Z_{#2}}}
\newcommand{\uce}{{\mathfrak{uce}}}
\newcommand{\Rcat}{\mathcal R}
\newcommand{\cS}{{\mathcal{S}}}

%--------------<\Global>-------------------------

%--------------<Local>---------------------------
\newcommand{\E}{{\mathcal{E}}}
\newcommand{\F}{{\mathcal{F}}}

\newcommand{\Etopo}{{\mathcal{E}_{\te{topo}}}}

\newcommand{\Ye}{{\mathcal{Y}_\E}}

\newcommand{\rh}{{{\bf h}}}
\newcommand{\rp}{{{\bf p}}}
\newcommand{\rrho}{{{\pmb \varrho}}}
\newcommand{\ral}{{{\pmb \al}}}

%% ----------------<functors>--------------------
\newcommand{\comp}{{\mathfrak{comp}}}
\newcommand{\ctimes}{{\widehat{\boxtimes}}}
\newcommand{\ptimes}{{\widehat{\otimes}}}
\newcommand{\ptimeslt}{{
%   \leftidx{ _{\te{tri}}}{\ctimes}{}
{}_{\te{t}}\ptimes
}}
\newcommand{\ptimesrt}{{\ot_{\te{t}} }}
\newcommand{\ttp}[1]{{
    {}_{{#1}}\ptimes
}}
\newcommand{\bigptimes}{{\widehat{\bigotimes}}}
\newcommand{\bigptimeslt}{{
%   \leftidx{ _{\te{tri}}}{\ctimes}{}
{}_{\te{t}}\bigptimes
}}
\newcommand{\bigptimesrt}{{\bigptimes_{\te{t}} }}
\newcommand{\bigttp}[1]{{
    {}_{{#1}}\bigptimes
}}

\newcommand{\ot}{\otimes}
\newcommand{\Ot}{\bigotimes}
\newcommand{\bt}{\boxtimes}

\newcommand{\affva}[1]{V_{\wh\g}\(#1,0\)}
\newcommand{\saffva}[1]{L_{\wh\g}\(#1,0\)}
\newcommand{\saffmod}[1]{L_{\wh\g}\(#1\)}

\newcommand{\otcopies}[2]{\belowit{\underbrace{{#1}\ot \cdots \ot {#1}}}{{#2}\te{-times}}}

\newcommand{\wtotcopies}[3]{\belowit{\underbrace{{#1}\wh\ot_{#2} \cdots \wh\ot_{#2} {#1}}}{{#3}\te{-times}}}

%% ----------------<\functors>-------------------

%% ----------------<algebras>--------------------
\newcommand{\tar}{{\mathcal{DY}}_0\(\mathfrak{gl}_{\ell+1}\)}
\newcommand{\U}{{\mathcal{U}}}
\newcommand{\htar}{\mathcal{DY}_\hbar\(A\)}
\newcommand{\hhtar}{\widetilde{\mathcal{DY}}_\hbar\(A\)}
\newcommand{\htarz}{\mathcal{DY}_0\(\mathfrak{gl}_{\ell+1}\)}
\newcommand{\hhtarz}{\widetilde{\mathcal{DY}}_0\(A\)}
\newcommand{\qhei}{\U_\hbar\left(\hat{\h}\right)}
\newcommand{\n}{{\mathfrak{n}}}
\newcommand{\vac}{{{\mathbbm 1}}}
\newcommand{\vtar}{{{
    \mathcal{V}_{\hbar,\tau}\left(\ell,0\right)
}}}

\newcommand{\qtar}{
    \U_q\(\wh\g_\mu\)}
\newcommand{\rk}{{\bf k}}
% ----------------<\algebras>-------------------

%--------------<\Local>--------------------------

\makeatletter
\@addtoreset{equation}{section}
\def\theequation{\thesection.\arabic{equation}}
\makeatother \makeatletter

\title{Twisted tensor products of quantum affine vertex algebras and coproducts}

\author{Fei Kong$^1$}
\email{kongmath@hunnu.edu.cn}
\address{Key Laboratory of Computing and Stochastic Mathematics (Ministry of Education), School of Mathematics and Statistics, Hunan Normal University, Changsha, China 410081}
\thanks{$^1$Partially supported by the NSF of China  (No.12371027).}

%\begin{titlepage}
%\title{Quantization of parafermion vertex algebras}
%\author{Fei Kong}
%\newcommand\institute{%
%Key Laboratory of Computing and Stochastic Mathematics (Ministry of Education), School of Mathematics and Statistics, Hunan Normal University, Changsha, China 410081}
%\makeatletter
%\centering
%{\Huge\bfseries\sffamily\@title} \bigskip\par
%{\Large\bfseries\@author} \bigskip\par
%\email{kongmath@hunnu.edu.cn}
%\makeatother
%\vfill
%\large\institute
%\end{titlepage}

%------
% Insert an abstract.
%------
\begin{abstract}
Let $\g$ be a symmetrizable Kac-Moody Lie algebra, and let $V_{\hat\g,\hbar}^\ell$, $L_{\hat\g,\hbar}^\ell$ be the quantum affine vertex algebras constructed in \cite{K-Quantum-aff-va}.
For any complex numbers $\ell$ and $\ell'$, we present an $\hbar$-adic quantum vertex algebra homomorphism $\Delta$ from $V_{\hat\g,\hbar}^{\ell+\ell'}$ to the twisted tensor product $\hbar$-adic quantum vertex algebra $V_{\hat\g,\hbar}^\ell\wh\ot V_{\hat\g,\hbar}^{\ell'}$.
In addition, if both $\ell$ and $\ell'$ are positive integers, we show that $\Delta$ induces an $\hbar$-adic quantum vertex algebra homomorphism from $L_{\hat\g,\hbar}^{\ell+\ell'}$ to the twisted tensor product $\hbar$-adic quantum vertex algebra $L_{\hat\g,\hbar}^\ell\wh\ot L_{\hat\g,\hbar}^{\ell'}$. Moreover, we prove the coassociativity of $\Delta$.
\end{abstract}

%------
% Optional: Dedication.
%------
%\dedication{This memoir is dedicated to XXX.}

%------
% Insert a list of keywords.
% -- Separate keywords with comma.
% -- Capitalize only the first keyword in the list.
% -- No final full stop.
%------

%------
% Insert MSC 2020 codes according to www.ams.org/msc/msc2020.html.
% -- There must be exactly *one* primary code.
% -- The number of secondary codes is not specified.
% -- No final full stop.
%------
%------
% Insert acknowledgments.
%------
%\begin{ack}
%Part of this paper was finished during my visit at Xiamen University and Tianyuan Mathematical Center in Southeast China, in
%August 2023. I am very grateful to Professor Shaobin Tan, Fulin Chen, Qing Wang for their hospitality.
%\end{ack}
%
%%------
%% Insert information regarding funding.
%%------
%\begin{funding}
%NSF of China (No.12371027).
%\end{funding}
\maketitle

%\tableofcontents

%------
% Insert the body of the book here.
%------
\section{Introduction}

Let $\g$ be a finite dimensional simple Lie algebra over $\C$, and let $\hat\g=\g\ot\C[t,t\inv]\oplus\C c$ be the affine Kac-Moody Lie algebra.
It is well known that both the universal affine vertex algebra $V_{\hat\g}^\ell$ and the (graded) simple affine vertex algebra $L_{\hat\g}^\ell$ are highest weight modules of $\hat\g$ generated by the vacuum vectors $\vac$.
Moreover, the $\hat\g$-module homomorphisms $\Delta:X_{\hat\g}^{\ell+\ell'}\to X_{\hat\g}^\ell\ot X_{\hat\g}^{\ell'}$ ($X=V,L$, $\ell,\ell'\in\C$) induced by $\vac\to \vac\ot\vac$ are also vertex algebra homomorphisms.

In \cite{EK-qva}, Etingof and Kazhdan developed a theory of \emph{quantum vertex operator algebras} in the sense of formal deformations of vertex algebras. In particular, they constructed the \emph{quantum affine vertex algebras} as formal deformations of universal vertex algebras of type $A$, by using the $R$-matrix type realizations given in \cite{RS-RTT}.
Based on the $R$-matrix type realizations of Yangian doubles \cite{JYL-R-mat-DY},
Butorac, Jing and Ko\v{z}i\'{c} (\cite{BJK-qva-BCD}) extended Etingof-Kazhdan's construction
to type $B$, $C$ and $D$ rational $R$-matrices.
The modules of these quantum vertex operator algebras are in one-to-one
correspondence with restricted modules for the corresponding Yangian doubles (see \cite{K-qva-phi-mod-BCD}).
Recently, Ko\v{z}i\'{c} utilized the $R$-matrix presentation of quantum affine algebras
(see \cite{DF-qaff-RTT-Dr,JLM-qaff-RTT-Dr-BD,JLM-qaff-RTT-Dr-C}), to constructed the quantum vertex operator algebras associated with trigonometric
$R$-matrices of types $A$, $B$, $C$ and $D$ (\cite{Kozic-qva-tri-A, K-qva-phi-mod-BCD}),
and established a one-to-one correspondence between $\phi$-coordinated modules and
restricted modules for quantum affine algebras.

In our previous work \cite{JKLT-Defom-va} coauthored with N. Jing, H. Li and S. Tan, we developed a method to construct quantum vertex operator algebras by using vertex bialgebras. By utilizing this,
we constructed affine quantum vertex operator algebras $V_{\hat\g,\hbar}^\ell$ ($\ell\in\C$) and $L_{\hat\g,\hbar}^\ell$ ($\ell\in\Z_+$) for all symmetric Kac-Moody Lie algebras $\g$ \cite{K-Quantum-aff-va}, based on the Drinfeld type relations (\cite{Dr-new,J-KM,Naka-quiver}).
We also established a one-to-one correspondence between $\phi$-coordinated modules of $V_{\hat\g,\hbar}^\ell$ and restricted modules for the quantum affinization algebra $\U_\hbar(\hat\g)$ of level $\ell$, and showed that every restricted integrable $\U_\hbar(\hat\g)$-modules of level positive integer $\ell$ are all $\phi$-coordinated modules of $L_{\hat\g,\hbar}^\ell$.
Moreover, when $\g$ is of finite type and $\ell\in\Z_+$, we proved that $L_{\hat\g,\hbar}^\ell/\hbar L_{\hat\g,\hbar}^\ell$ is isomorphic to the simple affine vertex algebra $L_{\hat\g}^\ell$.

Our goal in this paper is to give the quantum analogue of the vertex algebra homomorphism $\Delta$.
Different from the vertex algebra case, the tensor product quantum vertex operator algebra $X_{\hat\g,\hbar}^\ell\wh\ot X_{\hat\g,\hbar}^{\ell'}$ should be a twisted tensor product (\cite{LS-twisted-tensor,S-iter-twisted-tensor}) rather than the usual tensor product.
So in Section \ref{sec:construct-n-qvas}, we extend the method given in \cite{JKLT-Defom-va}
and derive a construction method of twisted tensor products by using vertex bialgebras.
Using this construction, we construct the twisted tensor products of quantum affine vertex algebras in Section \ref{sec:construct-n-qvas-qaffva}.

The paper is organized as follows.
Section \ref{sec:VAs} provides an introduction of vertex algebras and affine vertex algebras.
Section \ref{sec:qvas} presents the basics about $\hbar$-adic quantum vertex algebras.
Section \ref{sec:twisted-tensor-prod} recall the theory of twisted tensor products.
Section \ref{sec:construct-n-qvas} gives a construction method for twisted tensor products by using vertex bialgebras.
Section \ref{sec:qaff-va} presents the construction of quantum affine vertex algebras as described in \cite{K-Quantum-aff-va},
their twisted tensor products are presented in Section \ref{sec:construct-n-qvas-qaffva}.
Finally, in Section \ref{sec:coprod}, we construct the coproducts of the quantum affine vertex algebras.

Throughout this paper, we denote by $\Z_+$ and $\N$ the set of positive and nonnegative integers, respectively.
For a vector space $W$ and $g(z)\in W[[z,z\inv]]$, we denote by $g(z)^+$ (resp. $g(z)^-$) the regular (singular) part of $g(z)$.

\section{Vertex algebras}\label{sec:VAs}

A \emph{vertex algebra} (VA) is a vector space $V$ together with a \emph{vacuum vector} $\vac\in V$ and a vertex operator map
\begin{align}
    Y(\cdot,z):&V\longrightarrow \E(V):=\Hom(V,V((z)));\quad
    v\mapsto Y(v,z)=\sum_{n\in\Z}v_nz^{-n-1},
\end{align}
such that
\begin{align}\label{eq:vacuum-property}
    Y(\vac,z)v=v,\quad Y(v,z)\vac\in V[[z]],\quad \lim_{z\to 0}Y(v,z)\vac=v,
    \quad\te{for }v\in V,
\end{align}
and that
\begin{align}\label{eq:Jacobi}
    &z_0\inv\delta\(\frac{z_1-z_2}{z_0}\)Y(u,z_1)Y(v,z_2)
    -z_0\inv\delta\(\frac{z_2-z_1}{-z_0}\)Y(v,z_2)Y(u,z_1)\\
    &\quad=z_1\inv\delta\(\frac{z_2+z_0}{z_1}\)Y(Y(u,z_0)v,z_2)
    \quad\quad\te{for }u,v\in V.\nonumber
\end{align}

Let $\mathcal B$ be a set, and let $N:\mathcal B\times \mathcal B\to \N$ be a symmetric function called \emph{locality function} (\cite{R-free-conformal-free-va}).
Define $\mathcal V(\mathcal B,N)$ to be the category consisting of VAs $V$, containing $\mathcal B$ as a subset and
\begin{align}
    (z_1-z_2)^{N(a,b)}[Y(a,z_1),Y(b,z_2)]=0\quad \te{for }a,b\in\mathcal B.
\end{align}
The following result was given in \cite{R-free-conformal-free-va}.
\begin{prop}
There exists a unique VA $V(\mathcal B,N)$ such that for any $V\in\mathcal V(\mathcal B,N)$, there exists a unique VA homomorphism $f:V(\mathcal B,N)\to V$ such that
\begin{align}
    f(a)=a\quad\te{for }a\in \mathcal B.
\end{align}
\end{prop}

Let $A=(a_{ij})_{i,j\in I}$ be a symmetrizable generalized Cartan matrix.
Then there are unique relatively prime positive integers $r_i$ ($i\in I$) such that $DA$ is symmetric with $D=\diag\{r_i\}_{i\in I}$.
Recall the following Lie algebra given in \cite[Definition 2.4]{K-Quantum-aff-va}.

\begin{de}
The Lie algebra $\hat\g$ is isomorphic to the Lie algebra generated by
\begin{align*}
    \set{h_i(m),\,x_i^\pm(m)}{i\in I}
\end{align*}
and a central element $c$, subject to the relations written in terms of generating functions in $z$:
\begin{align*}
    h_i(z)=\sum_{m\in\Z}h_i(m)z^{-m-1},\quad x_i^\pm(z)=\sum_{m\in\Z}x_i^\pm(m)z^{-m-1},\quad i\in I.
\end{align*}
The relations are $(i,j\in I)$:
\begin{align*}
    \te{(L1)}\quad &[h_i(z_1),h_j(z_2)]=r_i a_{ij}rc\pd{z_2}z_1\inv\delta\(\frac{z_2}{z_1}\),\\
    \te{(L2)}\quad &[h_i(z_1),x_j^\pm(z_2)]=\pm r_i a_{ij}x_j^\pm(z_2)z_1\inv\delta\(\frac{z_2}{z_1}\),\\
    \te{(L3)}\quad &\left[x_i^+(z_1),x_j^-(z_2)\right]=\frac{\delta_{ij}}{r_i}\(h_i(z_2)z_1\inv\delta\(\frac{z_2}{z_1}\)+rc\pd{z_2}z_1\inv\delta\(\frac{z_2}{z_1}\)  \),\\
    \te{(L4)}\quad&(z_1-z_2)^{n_{ij}}\left[x_i^\pm(z_1),x_j^\pm(z_2)\right]=0,\\
    \te{(S)}\quad &\left[ x_i^\pm(z_1),\left[ x_i^\pm(z_2),\dots,\left[ x_i^\pm(z_{m_{ij}}),x_j^\pm(z_0) \right]\cdots \right] \right]=0,\quad \te{if }a_{ij}\le 0,
\end{align*}
where $n_{ij}=1-\delta_{ij}$ for $i,j\in I$ and $m_{ij}=1-a_{ij}$ for $i,j\in I$ with $a_{ij}\le 0$.
\end{de}

Introduce a set $\mathcal B=\set{h_i,\,x_i^\pm}{i\in I}$, and defined a function $N:\mathcal B\times\mathcal B\to\N$ by
\begin{align*}
    N(h_i,h_j)=2,\quad N(h_i,x_j^\pm)=1,\quad N(x_i^\pm,x_j^\pm)=n_{ij},\quad N(x_i^+,x_j^-)=\delta_{ij}2.
\end{align*}

\begin{de}\label{de:affVAs}
For $\ell\in \C$, we let $F_{\hat\g}^\ell$ be the quotient VA of $V(\mathcal B,N)$ modulo the ideal generated by
\begin{align*}
    (h_i)_0(h_j),\quad (h_i)_1(h_j)-r_i{a_{ij}}r\ell\vac,\quad (h_i)_0(x_j^\pm)\mp r_i a_{ij} x_j^\pm \quad \te{for }i,j\in I.
\end{align*}
Furthermore, let $V_{\hat\g}^\ell$ be the quotient VA of $F_{\hat\g}^\ell$ modulo the ideal generated by
\begin{align*}
    &(x_i^+)_0(x_j^-)-\frac{\delta_{ij}}{r_i}h_i,\quad (x_i^+)_1(x_j^-)-\frac{\delta_{ij}}{r_i}r\ell \vac\quad\te{for }i,j\in I,\\
    &\(x_i^\pm\)_0^{m_{ij}}(x_j^\pm)\quad\te{for }i,j\in I\,\,\te{with }a_{ij}\le 0.
\end{align*}
If $\ell\in\Z_+$, we define $L_{\hat\g}^\ell$ to be the quotient VA $V_{\hat\g}^\ell$ modulo the ideal generated by
\begin{align*}
    \(x_i^\pm\)_{-1}^{r\ell/r_i}(x_i^\pm)\quad\te{for }i\in I.
\end{align*}
\end{de}

%\begin{rem}
%\emph{Set $\hat\g_+=\g\ot \C[t]\oplus\C c$. For $\ell\in \C$, we let $\C_\ell:=\C$ be a $\hat\g_+$-module, with $\g\ot\C[t].\C_\ell=0$ and $c=\ell$. Then there are $\hat\g$-module structures on both $V_{\hat\g}^\ell$ and $L_{\hat\g}^\ell$,
%such that
%\begin{align*}
%    h_i(z)=Y(h_i,z),\quad x_i^\pm(z)=Y(x_i^\pm,z),\quad i\in I.
%\end{align*}
%In addition, as $\hat\g$-modules, we have that
%\begin{align*}
%    V_{\hat\g}^\ell\cong \U(\hat\g)\ot_{\U(\hat\g_+)}\C_\ell.
%\end{align*}
%And $L_{\hat\g}^\ell$ is the unique simple quotient $\hat\g$-module of $V_{\hat\g}^\ell$, when $\ell\in\Z_+$.
%}
%\end{rem}

\section{$\hbar$-adic quantum vertex algebras}\label{sec:qvas}

In this paper, we let $\hbar$ be a formal variable, and let $\C[[\hbar]]$ be the ring of formal power series in $\hbar$.
A $\C[[\hbar]]$-module $V$
is said to be \emph{torsion-free} if $\hbar v\ne 0$ for every $0\ne v\in V$.
And a $\C[[\hbar]]$-module $V$ is \emph{topologically free} if $V=V_0[[\hbar]]$
for some vector space $V_0$ over $\C$.
It is known that a $\C[[\hbar]]$-module is topologically free if and only if it is torsion-free, and Hausdorff complete under the $\hbar$-adic topology (see \cite{Kassel-topologically-free}).
For a topologically free $\C[[\hbar]]$-module $V$ and a submodule $U\subset V$, we denote by $\bar U$ the closure of $U$ and set
\begin{align}
  [U]=\set{u\in V}{\hbar^n u\in U\,\,\te{for some }n\in\Z_+}.
\end{align}
Then $\overline{[U]}$ is the minimal closed submodule that is invariant under the operation $[\cdot]$.
For another topologically free $\C[[\hbar]]$-module $U=U_0[[\hbar]]$, we recall the complete tensor
\begin{align*}
    U\wh\ot V=(U_0\ot V_0)[[\hbar]].
\end{align*}

We view a vector space as a $\C[[\hbar]]$-module by letting $\hbar=0$.
Fix a $\C[[\hbar]]$-module $W$. For $k\in\Z_+$, and some formal variables $z_1,\dots,z_k$, we define
\begin{align}
    \E^{(k)}(W;z_1,\dots,z_k)=\Hom_{\C[[\hbar]]}\(W,W((z_1,\dots,z_k))\)
\end{align}
We will denote $\E^{(k)}(W;z_1,\dots,z_k)$ by $\E^{(k)}(W)$ if there is no ambiguity,
and will denote $\E^{(1)}(W)$ by $\E(W)$.
An ordered sequence $(a_1(z),\dots,a_k(z))$ in $\E(W)$ is said to be \emph{compatible} (\cite{Li-nonlocal})
if there exists an $m\in\Z_+$, such that
\begin{align*}
    \(\prod_{1\le i<j\le k}(z_i-z_j)^m\)a_1(z_1)\cdots a_k(z_k)\in\E^{(k)}(W).
\end{align*}

Now, we assume that $W=W_0[[\hbar]]$ is topologically free.
Define
\begin{align}
    \E_\hbar^{(k)}(W;z_1,\dots,z_k)=\Hom_{\C[[\hbar]]}\(W,W_0((z_1,\dots,z_k))[[\hbar]]\).
\end{align}
Similarly, we denote $\E_\hbar^{(k)}(W;z_1,\dots,z_k)$ by $\E_\hbar^{(k)}(W)$ if there is no ambiguity,
and denote $\E_\hbar^{(1)}(W)$ by $\E_\hbar(W)$ for short.
We note that $\E_\hbar^{(k)}(W)=\E^{(k)}(W_0)[[\hbar]]$ is topologically free.
For $n,k\in\Z_+$, the quotient map from $W$ to $W/\hbar^nW$ induces the following $\C[[\hbar]]$-module map
\begin{align*}
    \wt\pi_n^{(k)}:\End_{\C[[\hbar]]}(W)[[z_1^{\pm 1},\dots,z_k^{\pm 1}]]
    \to \End_{\C[[\hbar]]}(W/\hbar^nW)[[z_1^{\pm 1},\dots,z_k^{\pm 1}]].
\end{align*}
For $A(z_1,z_2),B(z_1,z_2)\in\Hom_{\C[[\hbar]]}(W,W_0((z_1))((z_2))[[\hbar]])$, we write $A(z_1,z_2)\sim B(z_2,z_1)$
if for each $n\in\Z_+$ there exists $k\in\N$, such that
\begin{align*}
    (z_1-z_2)^k\wt\pi_n^{(2)}(A(z_1,z_2))=(z_1-z_2)^k\wt\pi_n^{(2)}(B(z_2,z_1)).
\end{align*}

For each $k\in\Z_+$, the inverse system
\begin{align*}
    \xymatrix{
    0&W/\hbar W\ar[l]&W/\hbar^2W\ar[l]&W/\hbar^3W\ar[l]&\cdots\ar[l]
    }
\end{align*}
induces the following inverse system
\begin{align}\label{eq:E-h-inv-sys}
    \xymatrix{
    0&\E^{(k)}(W/\hbar W)\ar[l]&\E^{(k)}(W/\hbar^2W)\ar[l]&\cdots\ar[l]
    }
\end{align}
Then $\E_\hbar^{(k)}(W)$ is isomorphic to the inverse limit of \eqref{eq:E-h-inv-sys}.
The map $\wt \pi_n^{(k)}$ induces a $\C[[\hbar]]$-module $\pi_n^{(k)}:\E_\hbar^{(k)}(W)\to \E^{(k)}(W/\hbar^nW)$.
It is easy to verify that $\ker \pi_n^{(k)}=\hbar^n\E_\hbar^{(k)}(W)$.
We will denote $\pi_n^{(1)}$ by $\pi_n$ for short.
An ordered sequence $(a_1(z),\dots,a_r(z))$ in $\E_\hbar(W)$ is called \emph{$\hbar$-adically compatible} if for every $n\in\Z_+$,
the sequence $$(\pi_n(a_1(z)),\dots,\pi_n(a_r(z)))$$ in $\E(W/\hbar^nW)$ is compatible.
A subset $U$ of $\E_\hbar(W)$ is called \emph{$\hbar$-adically compatible} if every finite sequence in $U$ is $\hbar$-adically compatible.
Let $(a(z),b(z))$ in $\E_\hbar(W)$ be $\hbar$-adically compatible. That is, for any $n\in\Z_+$, we have
\begin{align*}
    (z_1-z_2)^{k_n}\pi_n(a(z_1))\pi_n(b(z_2))\in\E^{(2)}(W/\hbar^nW)\quad\te{for some }k_n\in\N.
\end{align*}
We recall the following vertex operator map (\cite{Li-h-adic}):
\begin{align*}
    &Y_\E(a(z),z_0)b(z)=\sum_{n\in\Z}a(z)_nb(z)z_0^{-n-1}\\
    =&\varinjlim_{n>0}z_0^{-k_n}\left.\((z_1-z)^{k_n}\pi_n(a(z_1))\pi_n(b(z))\)\right|_{z_1=z+z_0}.
\end{align*}

An \emph{$\hbar$-adic nonlocal VA} (\cite{Li-h-adic}) is a topologically free $\C[[\hbar]]$-module $V$ equipped with a vacuum vector $\vac$ such that the vacuum property \eqref{eq:vacuum-property} hold, and a vertex operator map $Y(\cdot,z):V\to \E_\hbar(V)$
such that
\begin{align*}
    \set{Y(u,z)}{u\in V}\subset\E_\hbar(V)\quad\te{is $\hbar$-adically compatible},
\end{align*}
and that
\begin{align*}
    Y_\E(Y(u,z),z_0)Y(v,z)=Y(Y(u,z_0)v,z)\quad\te{for }u,v\in V.
\end{align*}
We denote by $\partial$ the canonical derivation of $V$:
\begin{align}
    u\to\partial u=\lim_{z\to 0}\frac{d}{dz}Y(u,z)\vac.
\end{align}

An \emph{$\hbar$-adic quantum VA} is an $\hbar$-adic nonlocal VA $V$ equipped with a $\C[[\hbar]]$-module map (called a \emph{quantum Yang-Baxter operator})
\begin{align}
  S(z):V\wh\ot V\to V\wh\ot V\wh\ot \C((z))[[\hbar]],
\end{align}
which satisfies the \emph{shift condition}:
\begin{align}\label{eq:qyb-shift}
  [\partial\ot 1,S(z)]=-\frac{d}{dz}S(z),
\end{align}
the \emph{quantum Yang-Baxter equation}:
\begin{align}\label{eq:qyb}
  S^{12}(z_1)S^{13}(z_1+z_2)S^{23}(z_2)=S^{23}(z_2)S^{13}(z_1+z_2)S^{12}(z_1),
\end{align}
and the \emph{unitarity condition}:
\begin{align}\label{eq:qyb-unitary}
  S^{21}(z)S(-z)=1,
\end{align}
satisfying the following conditions:

  (1) The \emph{vacuum property}:
  \begin{align}\label{eq:qyb-vac-id}
    S(z)(\vac\ot v)=\vac\ot v,\quad \te{for }v\in V.
  \end{align}

 (2) The \emph{$S$-locality}:
  For any $u,v\in V$, one has
  \begin{align}\label{eq:qyb-locality}
  Y(u,z_1)Y(v,z_2)\sim Y(z_2)(1\ot Y(z_1))S(z_2-z_1)(v\ot u).
  \end{align}

  (3) The \emph{hexagon identity}:
  \begin{align}\label{eq:qyb-hex1}
    S(z_1)Y^{12}(z_2)=Y^{12}(z_2)S^{23}(z_1)S^{13}(z_1+z_2).
  \end{align}

\begin{rem}\label{rem:Jacobi-S}
\emph{
For $u,v\in V$, we have that
\begin{align*}
    &Y(u,z_1)Y(v,z_2)-Y(z_2)(1\ot Y(z_1))S(z_2-z_1)(v\ot u)\\
    &\quad=Y(Y(u,z_1-z_2)^-v-Y(u,-z_2+z_1)^-v,z_2).
\end{align*}
}
\end{rem}

\begin{rem}{\em
Let $(V,S(z))$ be a quantum VA.
Then
\begin{align}
  &S(z)(v\ot\vac)=v\ot\vac\quad \te{for }v\in V,\label{eq:qyb-vac-id-alt}\\
  &[1\ot\partial,S(z)]=\frac{d}{dz}S(z),\label{eq:qyb-shift-alt}\\
  &S(z_1)(1\ot Y(z_2))=(1\ot Y(z_2))S^{12}(z_1-z_2)S^{13}(z_1)\label{eq:qyb-hex2},\\
  &S(z)f(\partial\ot 1)=f\(\partial\ot 1+\pd{z}\)S(z)\quad \te{for }f(z)\in\C[z][[\hbar]],\label{eq:qyb-shift-total1}\\
  &S(z)f(1\ot \partial)=f\(1\ot \partial-\pd{z}\)S(z)\quad \te{for }f(z)\in\C[z][[\hbar]].\label{eq:qyb-shift-total2}
\end{align}
}
\end{rem}

\section{Twisted tensor products}\label{sec:twisted-tensor-prod}
In this section, we first present the direct $\hbar$-adic analogue of the \emph{twisted tensor products} given in \cite{LS-twisted-tensor, S-iter-twisted-tensor}.

\begin{de}
Let $V_1$ and $V_2$ be two $\hbar$-adic nonlocal VAs.
A \emph{twisting operator} for the ordered pair $(V_1,V_2)$ is a $\C[[\hbar]]$-module map
\begin{align*}
  S(z):V_1\wh\ot V_2\to V_1\wh\ot V_2\wh\ot\C((z))[[\hbar]],
\end{align*}
satisfying the following conditions:
\begin{align*}
  &S(z)(\vac\ot v)=\vac\ot v,\quad S(z)(u\ot \vac)=u\ot \vac\quad\te{for }u\in V_1,\,\,v\in V_2,\quad\te{and}\\
  &S(z_1)Y_1^{12}(z_2)=Y_1^{12}(z_2) S^{23}(z_1)S^{13}(z_1+z_2),\quad S(z_1)Y_2^{23}(z_2)=Y_2^{23}(z_2)S^{12}(z_1-z_2)S^{13}(z_1),
\end{align*}
where $Y_1$ and $Y_2$ are the vertex operators of $V_1$ and $V_2$, respectively.
\end{de}

\begin{rem}
    \emph{
    Given a twisting operator $S(z)$ for the ordered pair $(V_1,V_2)$, $S(z)\sigma$ is a twisting operator defined in \cite[Definition 2.2]{LS-twisted-tensor}, where $\sigma:V_1\wh\ot V_2\to V_2\wh\ot V_1$ is the flip map.
    Conversely, let $R(z)$ be a twisting operator for the ordered pair $(V_1,V_2)$ in sense of \cite[Definition 2.2]{LS-twisted-tensor}.
    Then $R(z)\sigma$ is a twisting operator for $(V_1,V_2)$.
    }
\end{rem}

$S(x)$ is said to be \emph{invertible}, if it is invertible
as a $\C((z))[[\hbar]]$-linear transformation of $V_1\wh\ot V_2\wh\ot\C((z))[[\hbar]]$.
Then $S^{21}(-z)\inv$ is an invertible twisting operator for the ordered pair $(V_2,V_1)$
(\cite[Lemma 2.3]{LS-twisted-tensor}).

%The following result is the $\hbar$-adic analogue of \cite[Theorem 2.4]{LS-twisted-tensor}.
\begin{thm}\label{thm:twist-tensor}
Define $Y_S(z):(V_1\wh\ot V_2)\ot(V_1\wh\ot V_2)\to (V_1\wh\ot V_2)[[z,z\inv]]$ by
\begin{align*}
  Y_S(z)=(Y_1(z)\ot Y_2(z))S^{23}(-z)\sigma,
\end{align*}
where $Y_1$ and $Y_2$ are the vertex operators of $V_1$ and $V_2$,
respectively.
Then $(V_1\wh\ot V_2,Y_S,\vac\ot\vac)$ carries the structure of a nonlocal VA.
Moreover, we denote this nonlocal VA by $V_1\wh\ot_S V_2$.
\end{thm}

%The following is the $\hbar$-adic analogue of \cite[Theorem 2.5]{S-iter-twisted-tensor}.
\begin{thm}\label{thm:3-twisted-tensor}
Let $V_1,V_2$ and $V_3$ be nonlocal VAs, $S_1(z),S_2(z)$ and $S_3(z)$ twisting operators for ordered pairs $(V_1,V_2)$, $(V_2,V_3)$ and $(V_1,V_3)$, respectively.
Suppose that $S_1(z),S_2(z)$ and $S_3(z)$ satisfy the following hexagon equation
\begin{align}\label{eq:twisting-op-qybeq}
  S_1^{12}(z_1)S_3^{13}(z_1+z_2)S_2^{23}(z_2)=S_2^{23}(z_2)S_3^{13}(z_1+z_2)S_1^{12}(z_1).
\end{align}
Then we have that
\begin{align}
  T_1(z)=S_2^{23}(z)S_3^{13}(z)\quad\te{and}\quad
  T_2(z)=S_1^{12}(z)S_3^{13}(z)
\end{align}
are twisting operators for ordered pairs $(V_1\wh\ot_{S_1} V_2, V_3)$
and $(V_1, V_2\wh\ot_{S_2}V_3)$, respectively.
Moreover, the two twisted tensor products $(V_1\wh\ot_{S_1}V_2)\wh\ot_{T_1}V_3$
and $V_1\wh\ot_{T_2}(V_2\wh\ot_{S_2}V_3)$ are equivalent.
\end{thm}

%The following is the $\hbar$-adic analogue of \cite[Lemma 4.1]{S-iter-twisted-tensor}
%and \cite[Theorem 4.2]{S-iter-twisted-tensor} that
\begin{thm}\label{thm:coherence}
Let $V_1,\dots,V_n$ be $\hbar$-adic nonlocal VAs, and let
$$S_{ij}(z):V_i\wh\ot V_j\to V_i\wh\ot V_j\wh\ot\C((z))[[\hbar]]$$ be twisting operators for every
$1\le i<j\le n$, such that for any $i<j<k$ the twisting operators $S_{ij}(z)$,
$S_{jk}(z)$ and $S_{ik}(z)$ satisfy the hexagon equation \eqref{eq:twisting-op-qybeq}.
Let
\begin{align*}
S_{i,\{n-1,n\}}(z)=S_{i,n-1}^{12}(z)S_{in}^{13}(z)
\end{align*}
be twisting operators
for ordered pairs $(V_i, V_{n-1}\wh\ot_{S_{n-1,n}}V_n)$, ($i=1,2,\dots,n-2$).
Then for every $1<i<j\le n-2$, the twisting operators $S_{ij}(z)$,
$S_{j,\{n-1,n\}}(z)$ and $S_{i,\{n-1,n\}}(z)$ satisfy the hexagon equation \eqref{eq:twisting-op-qybeq}.
Moreover, for disjoint subsets $K_1,K_2,K_3\subset\{1,2,\dots,n\}$,
such that $k_1<k_2<k_3$ for all $k_i\in K_i$,
one has the (inductively defined) twisting operator
\begin{align*}
  &S_{K_1\cup K_2,K_3}(z)=S_{K_2,K_3}^{23}(z)S_{K_1,K_3}^{13}(z),\quad S_{K_1,K_2\cup K_3}(z)=S_{K_1,K_2}^{12}(z)S_{K_1,K_3}^{13}(z),
\end{align*}
and the (inductively defined) $\hbar$-adic nonlocal VA
\begin{align*}
  V_{K_1\cup K_2\cup K_3}:=
  V_{K_1}\wh\ot_{S_{K_1,K_2\cup K_3}}
    \(V_{K_2}\wh\ot_{S_{K_2,K_3}}V_{K_3}\)
  =\(V_{K_1}\wh\ot_{S_{K_1,K_2}}V_{K_2}\)\wh\ot_{S_{K_1\cup K_2,K_3}}V_{K_3}.
\end{align*}
Furthermore, we also denote the $\hbar$-adic nonlocal VA $V_{\{1,2,\dots,n\}}$
by
\begin{align*}
  V_1\wh\ot_{S_{12}}V_2\wh\ot_{S_{23}}\cdots\wh\ot_{S_{n-1,n}}V_n.
\end{align*}
\end{thm}

We generalize the notion of $\hbar$-adic quantum VAs as follows.

\begin{de}\label{de:qyb-mult}
Let $(V_i,Y_i,\vac)$ ($1\le i\le n$) be $\hbar$-adic nonlocal VAs.
We call a family of $\C[[\hbar]]$-module maps $S_{ij}(z):V_i\wh\ot V_j\to V_i\wh\ot V_j\wh\ot \C((z))[[\hbar]]$ ($1\le i,j\le n$)
\emph{quantum Yang-Baxter operators} of the $n$-tuple $(V_1,\dots, V_n)$, if
\begin{align}
  &S_{ij}(z)(v\ot \vac)=v\ot \vac\quad
   S_{ij}(z)(\vac\ot u)=\vac\ot u\quad\te{for }u\in V_j,\,\,v\in V_i,\label{eq:multqyb-vac}\\
  &[\partial\ot 1,S_{ij}(z)]=-\frac{d}{dz}S_{ij}(z),
  \quad [1\ot\partial, S_{ij}(z)]=\frac{d}{dz}S_{ij}(z),\quad S_{ji}^{21}(z)S_{ij}(-z)=1,\label{eq:multqyb-der-shift}\\
  &Y_i(u,z_1)Y_i(v,z_2)\sim Y_i(z_2)(1\ot Y_i(z_1))S_{ii}(z_2-z_1)(v\ot u),\quad\te{for }u,v\in V_i,\label{eq:qyb-locality}\\
  &S_{ij}(z_1)Y_i^{12}(z_2)=Y_i^{12}(z_2)S_{ij}^{23}(z_1)S_{ij}^{13}(z_1+z_2),\label{eq:multqyb-hex1}\\
  &S_{ij}(z_1)Y_j^{23}(z_2)=Y_j^{23}(z_2)S_{ij}^{12}(z_1-z_2)S_{ij}^{13}(z_1),\label{eq:multqyb-hex2}\\
  &S_{ij}^{12}(z_1)S_{ik}^{13}(z_1+z_2)S_{jk}^{23}(z_2)
  =S_{jk}^{23}(z_2)S_{ik}^{13}(z_1+z_2)S_{ij}^{12}(z_1)\quad\te{for } 1\le i,j,k\le n.\label{eq:multqyb-qybeq}
\end{align}
We call the pair $((V_i)_{i=1}^n,(S_{ij}(z))_{i,j=1}^n)$ an \emph{$\hbar$-adic $n$-quantum VA}.
Moreover, let $$((V'_i)_{i=1}^n,(S'_{ij}(z))_{i,j=1}^n)$$ another $\hbar$-adic $n$-quantum VA, and let $f_i:V_i\to V'_i$ ($1\le i\le n$) be $\hbar$-adic nonlocal VA homomorphisms.
We call $(f_1,\dots,f_n)$ an \emph{$\hbar$-adic $n$-quantum VA homomorphism} if
\begin{align}
  (f_i\ot f_j)\circ S_{ij}(z)=S'_{ij}(z)\circ(f_i\ot f_j)\quad\te{for }1\le i,j\le n.
\end{align}
\end{de}

\begin{rem}
{\em
The notion of $\hbar$-adic $1$-quantum VAs is identical to the notion of $\hbar$-adic quantum VAs.
}
\end{rem}

\begin{rem}
{\em
Let $(V_i,Y_i,\vac)$ ($1\le i\le n$) be $\hbar$-adic nonlocal VAs, and let $(S_{ij}(z))_{i,j=1}^n$ be quantum Yang-Baxter operators of the $n$-tuple $(V_1,\dots,V_n)$.
Then
%for any $\sigma\in S_n$, $(S_{\sigma(i),\sigma(j)})_{i,j=1}^n$ are quantum Yang-Baxter operators of the $n$-tuple $(V_{\sigma(1)},\dots,V_{\sigma(n)})$.
%Moreover,
for any subset $\{i_1,i_2,\dots,i_k\}$ of $\{1,2,\dots,n\}$, $(S_{i_a,i_b}(z))_{a,b=1}^k$ are quantum Yang-Baxter operators of the $k$-tuple $(V_{i_1},\dots,V_{i_k})$.
}
\end{rem}

\begin{lem}\label{lem:S-quotient-alg}
For each $1\le i\le n$, let $U_i\subset V_i$ be a generating subset of $V_i$,
and let $M_i\subset V_i$ be a closed ideal of $V_i$ such that $[M_i]=M_i$ and
\begin{align*}
  &S_{ij}(z)(M_i\ot U_j),\,S_{ij}(x)(U_i\ot M_j)
  \subset M_i\wh\ot V_j\wh\ot \C((z))[[\hbar]]+V_i\wh\ot M_j\wh\ot \C((z))[[\hbar]]
\end{align*}
for $1\le i,j\le n$.
Then $S_{ij}(z)$ induces a $\C[[\hbar]]$-module map on $$V_i/M_i\wh\ot V_j/M_j\wh\ot \C((z))[[\hbar]],$$
which is still denoted by $S_{ij}(z)$.
And $((V_i/M_i)_{i=1}^n,(S_{ij}(z))_{i,j=1}^n)$ is an $\hbar$-adic $n$-quantum VA.
\end{lem}

\begin{lem}\label{lem:vertex-op-K}
For $K=\{i_1<i_2<\cdots<i_k\}\subset\{1,2,\dots,n\}$, we denote by $Y_K$ the vertex operator of $V_K$. Then we have that
\begin{align*}
  &Y_K(z)=Y_{i_1}^{12}(z) Y_{i_2}^{34}(z)\cdots Y_{i_k}^{2k-1,2k}(z)\\
  \times&\prod_{\substack{a-b=-1\\ 1\le a,b\le k}}S_{i_a,i_b}^{2a,2b-1}(-z)
  \prod_{\substack{a-b=-2\\ 1\le a,b\le k}}S_{i_a,i_b}^{2a,2b-1}(-z)\cdots
  \prod_{\substack{a-b=1-k\\ 1\le a,b\le k}}S_{i_a,i_b}^{2a,2b-1}(-z)
  \sigma^{23}\sigma^{345}\cdots  \sigma^{k,k+1,\dots,2k-1},
\end{align*}
where $\sigma^{a,a+1,\dots,b}=\sigma^{a,a+1}\sigma^{a+1,a+2}\cdots\sigma^{b-1,b}$.
\end{lem}

\begin{proof}
By utilizing \eqref{eq:multqyb-hex1} and \eqref{eq:multqyb-hex2}, one can straightforwardly prove the lemma using induction on $k$.
\end{proof}

For two subsets $K=\{i_1<i_2<\cdots<i_k\}$ and $L=\{j_1<j_2<\cdots<j_l\}$ of $\{1,2,\dots,n\}$, we set
\begin{align}
  S_{K,L}(z)=\prod_{\substack{a-b=k-1\\ 1\le a\le k\\ 1\le b\le l}}S_{i_a,j_b}^{a,k+b}(z)
    \prod_{\substack{a-b=k-2\\ 1\le a\le k\\ 1\le b\le l}}S_{i_a,j_b}^{a,k+b}(z)\cdots
    \prod_{\substack{a-b=1-l\\ 1\le a\le k\\ 1\le b\le l}}S_{i_a,j_b}^{a,k+b}(z).
\end{align}
The following result can be straightforwardly verified.

\begin{lem}\label{lem:3-qva-to-2-qva}
Let $(V_1,V_2,V_3)$ be an $\hbar$-adic $3$-quantum VA with quantum Yang-Baxter operators $(S_{ij}(z))_{i,j=1}^3$.
Set $K_1=\{1,2\}$ and $K_2=\{3\}$.
Then $(V_{K_1},V_{K_2})$ is an $\hbar$-adic $2$-quantum VA with quantum Yang-Baxter operators $(S_{K_a,K_b}(z))_{a,b=1}^2$.
\end{lem}

\begin{prop}\label{prop:n-qva-to-k-qva}
Let $((V_i)_{i=1}^n,(S_{ij}(z))_{i,j=1}^n)$ be an $\hbar$-adic $n$-quantum VA, and let $(K_1,\dots,K_k)$ be a partition of the set $\{1,2,\dots,n\}$.
Then $$((V_{K_i})_{i=1}^k,(S_{K_i,K_j}(z))_{i,j=1}^k)$$ is an $\hbar$-adic $k$-quantum VA.
\end{prop}

\begin{proof}
This proposition is immediate from induction on the size of each subset $K_i$ and Lemma \ref{lem:3-qva-to-2-qva}.
\end{proof}

Finally, we list some formulas for later use.
Let $((V_i)_{i=1}^n,(S_{ij}(z))_{i,j=1}^n$ be an $\hbar$-adic $n$-quantum VA.
The lemma below follows immediately from \eqref{eq:multqyb-der-shift}.

\begin{lem}\label{lem:multqyb-shift-total}
For any $f(z)\in\C[z][[\hbar]]$,
we have that
\begin{align*}
  &S_{ij}(z)f(\partial\ot 1)=f\(\partial\ot 1+\pd{z}\)S_{ij}(z),\quad
  S_{ij}(z)f(1\ot \partial)=f\(1\ot \partial-\pd{z}\)S_{ij}(z).
\end{align*}
\end{lem}

\begin{lem}\label{lem:S-special-tech-gen2}
Let $v\in V_i,u\in V_j$, $f(z)\in\C((z))[[\hbar]]$.
\begin{itemize}
  \item[(1)] If $S_{ij}(z)(v\ot u)=v\ot u+\vac\ot\vac\ot f(z)$,
then
\begin{align*}
  &S_{ij}(z)((v_{-1})^n\vac\ot u)
  =v_{-1}^n\vac\ot u+nv_{-1}^{n-1}\vac\ot \vac\ot f(z),\\
  &S_{ij}(z)(v\ot (u_{-1})^n\vac)
  =v\ot u_{-1}^n\vac+n\vac\ot u_{-1}^{n-1}\vac\ot f(z),\\
  &S_{ij}(z)\(\exp ( v_{-1})\vac\ot u\)
  =\exp ( v_{-1})\vac\ot u+\exp (v_{-1})\vac\ot \vac\ot  f(z),\\
  &S_{ij}(z)\(v\ot \exp( u_{-1})\vac\)
  =v\ot \exp( u_{-1})\vac+\vac\ot \exp( u_{-1})\vac\ot  f(z).
\end{align*}

\item[(2)] If $S_{ij}(z)(v\ot u)=v\ot u+\vac\ot u\ot f(z)$, then
\begin{align*}
  &S_{ij}(z)(v_{-1}^n\vac\ot u)=\sum_{i=0}^n\binom{n}{i}v_{-1}^i\vac\ot u\ot f(z)^{n-i},\\
  &S_{ij}(z)\(\exp ( v_{-1})\vac\ot u\)=\exp ( v_{-1})\vac\ot u\ot\exp\( f(z)\).
\end{align*}

\item[(3)] If $S_{ij}(z)(v\ot u)=v\ot u+v\ot \vac\ot f(z)$, then
\begin{align*}
  &S_{ij}(z)(v\ot u_{-1}^n\vac)=\sum_{i=0}^n\binom{n}{i}v\ot u_{-1}^i\vac\ot f(z)^{n-i},\\
  &S_{ij}(z)\(v\ot \exp( u_{-1})\vac\)=v\ot \exp\( u_{-1}\)\vac\ot \exp( f(z).
\end{align*}
\end{itemize}
\end{lem}

\begin{lem}\label{lem:S-special-tech-gen3}
Let $v\in V_i$, $u,u'\in V_j$, $f_1(z),f_2(z)\in\C((z))[[\hbar]]$.
\begin{itemize}
  \item[(1)] If
  $S_{ij}(z)(v\ot u)=v\ot u+\vac\ot \vac\ot f_1(z)$,
    $S_{ij}(z)(v\ot u')=v\ot u'+\vac\ot u'\ot f_2(z)$, then
  \begin{align*}
    &S_{ij}(z)(v\ot u_{-1}^nu')=v\ot u_{-1}^nu'
    +n\vac\ot u_{-1}^{n-1}u'\ot f_1(z)
    +\vac\ot u_{-1}^nu'\ot f_2(z),\\
    &S_{ij}(z)(v\ot \exp\(u_{-1}\)u')=v\ot \exp\(u_{-1}\)u'
    +\vac\ot \exp\(u_{-1}\)u'\ot (f_1(z)+f_2(z)).
  \end{align*}

  \item[(2)] If
  $S_{ij}(z)(v\ot u)=v\ot u+v\ot \vac\ot f_1(z)$,
    $S_{ij}(z)(v\ot u')=v\ot u'\ot f_2(z)$,
  then
    \begin{align*}
      &S_{ij}(z)(v\ot u_{-1}^n u')
      =\sum_{i=0}^n\binom{n}{i}v\ot u_{-1}^iu'\ot f_1(z)^{n-i}f_2(z),\\
      &S_{ij}(z)(v\ot \exp\(u_{-1}\)u')
      =v\ot\exp\(u_{-1}\)u'\ot \exp\(f_1(z)\)f_2(z).
    \end{align*}
\end{itemize}
\end{lem}

\begin{lem}\label{lem:S-special-tech-gen4}
Let $v,v'\in V_i$, $u\in V_j$, $f_1(z),f_2(z)\in\C((z))[[\hbar]]$.

\begin{itemize}
  \item[(1)] If
  $S_{ij}(z)(v\ot u)=v\ot u+\vac\ot\vac\ot f_1(z)$,
    $S_{ij}(z)(v'\ot u)=v'\ot u+v'\ot \vac\ot f_2(z)$,
  then
  \begin{align*}
    &S_{ij}(z)(v_{-1}^nv'\ot u)=v_{-1}^nv'\ot u
    +nv_{-1}^{n-1}v'\ot \vac\ot f_1(z)+v_{-1}^nv'\ot\vac\ot f_2(z),\\
    &S_{ij}(z)(\exp\(v_{-1}\)v'\ot u)=\exp\(v_{-1}\)v'\ot u
    +\exp\(v_{-1}\)v'\ot \vac\ot (f_1(z)+f_2(z)).
  \end{align*}

    \item[(2)]
    If $S_{ij}(z)(v\ot u)=v\ot u+v\ot \vac\ot f_1(z)$,
     $S_{ij}(z)(v'\ot u)=v'\ot u\ot f_2(z)$, then
      \begin{align*}
        &S_{ij}(z)(v_{-1}^n v'\ot u)
        =\sum_{i=0}^n\binom{n}{i}v_{-1}^iv'\ot u\ot f_1(z)^{n-i}f_2(z),\\
        &S_{ij}(z)(\exp\(v_{-1}\)v'\ot u)
        =\exp\(v_{-1}\)v'\ot u\ot \exp\(f_1(z)\)f_2(z).
      \end{align*}
\end{itemize}
\end{lem}

\section{Construction of $\hbar$-adic $n$-quantum vertex algebras}\label{sec:construct-n-qvas}
This section extends the construction of $\hbar$-adic quantum VAs introduced in \cite{JKLT-Defom-va} to the construction of $\hbar$-adic $n$-quantum VAs.

We first recall the smash product introduced in \cite{Li-smash}.
An \emph{$\hbar$-adic (nonlocal) vertex bialgebra} \cite{Li-smash}
is an $\hbar$-adic (nonlocal) VA $H$
equipped with a classical coalgebra structure $(\Delta,\varepsilon)$
such that (the coproduct) $\Delta:H\to H\wh\ot H$
and (the counit) $\varepsilon:H\to\C[[\hbar]]$ are homomorphism of $\hbar$-adic nonlocal VAs.
An \emph{$H$-module (nonlocal) VA} \cite{Li-smash} if an $\hbar$-adic nonlocal VA $V$ equipped with a module structure $\tau$ on $V$ for $H$ viewed as an $\hbar$-adic nonlocal VA such that
\begin{align}
  &\tau(h,z)v\in V\ot \C((z)),\qquad
  \tau(h,z)\vac_V=\varepsilon(h)\vac_V,\label{eq:mod-va-for-vertex-bialg1-2}\\
  &\tau(h,z_1)Y(u,z_2)v=\sum Y(\tau(h_{(1)},z_1-z_2)u,z_2)\tau(h_{(2)},z_1)v
  \label{eq:mod-va-for-vertex-bialg3}
\end{align}
for $h\in H$, $u,v\in V$, where $\vac_V$ denotes the vacuum vector of $V$
and $\Delta(h)=\sum h_{(1)}\ot h_{(2)}$ is the coproduct in the Sweedler notation.
The following result was given in \cite{Li-smash}.
\begin{thm}\label{thm:smash-prod}
Let $H$ be an $\hbar$-adic nonlocal vertex bialgebra and let $(V,\eta)$ be an $H$-module nonlocal VA.
Set $V\sharp H=V\wh\ot H$. For $u,v\in V$, $h,k\in H$, define
\begin{align}
    Y^\sharp(u\ot h,z)(v\ot k)=\sum Y(u,z)\tau(h_{(1)},z)v\ot Y(h_{(2)},z)k.
\end{align}
Then $V\sharp H$ is an $\hbar$-adic nonlocal VA.

Moreover, let $W$ be a $V$-module and an $H$-module such that
\begin{align}
    &Y_W(h,z)w\in W\wh\ot \C((z))[[\hbar]],\\
    &Y_W(h,z_1)Y_W(v,z_2)w=\sum Y_W(\tau(h_{(1)},z_1-z_2)v,z_2)Y_W(h_{(2)},z_1)w
\end{align}
for $h\in S$, $v\in V$, $w\in W$, where $S$ is a generating subset of $H$ as an $\hbar$-adic nonlocal VA.
Then $W$ is a module for $V\sharp H$ with
\begin{align}
    Y_W^\sharp(v\ot h,z)w=Y_W(v,z)Y_W(h,z)w\quad \te{for }h\in H,v\in V,w\in W.
\end{align}
\end{thm}

A \emph{(right) $H$-comodule nonlocal VA} (\cite[Definition 2.22]{JKLT-Defom-va}) is a nonlocal VA $V$ equipped with a homomorphism
$\rho:V\to V\wh\ot H$ of $\hbar$-adic nonlocal VAs such that
\begin{align}
  (\rho\ot 1)\rho=(1\ot \Delta)\rho,\quad (1\ot \varepsilon)\rho=\te{Id}_V.
\end{align}
$\rho$ and $\tau$ are \emph{compatible} (\cite[Definition 2.24]{JKLT-Defom-va}) if
\begin{align}
  \rho(\tau(h,z)v)=(\tau(h,z)\ot 1)\rho(v)\quad \te{for }h\in H,\,v\in V.
\end{align}

\begin{de}
Let $V$ be an $\hbar$-adic nonlocal VA.
A \emph{deforming triple} is a triple $(H,\rho,\tau)$, where $H$ is a cocommutative $\hbar$-adic nonlocal vertex bialgebra, $(V,\rho)$ is an $H$-nonlocal VA and $(V,\tau)$ is an $H$-module nonlocal VA, such that $\rho$ and $\tau$ are compatible.
\end{de}

The following represents the direct $\hbar$-adic analogue of \cite[Theorem 2.25 and Proposition 2.26]{JKLT-Defom-va}.

\begin{thm}\label{thm:deform-va}
Let $V$ be an $\hbar$-adic nonlocal VA, and let $(H,\rho,\tau)$ be a deforming triple.
Define
\begin{align}
\mathfrak D_\tau^\rho (Y)(a,z)=\sum Y(a_{(1)},z)\tau(a_{(2)},z)\quad\te{for }a\in V,
\end{align}
where $\rho(a)=\sum a_{(1)}\ot a_{(2)}\in V\otimes H$.
Then $(V,\mathfrak D_\tau^\rho (Y),\vac)$ forms an $\hbar$-adic nonlocal VA,
which is denoted by $\mathfrak D_\tau^\rho (V)$.
Moreover, $\rho$ is an $\hbar$-adic nonlocal VA homomorphism from
$\mathfrak D_\tau^\rho (V)$ to $V\sharp H$.
Furthermore, $(\mathfrak D_\tau^\rho(V),\rho)$ is also a right $H$-comodule nonlocal VA.
\end{thm}

\begin{rem}\label{rem:trivial-deform}
{\em
Let $H$ be a cocommutative $\hbar$-adic nonlocal vertex bialgebra, and let $(V,\rho)$ be an $H$-comodule nonlocal VA.
We view the counit $\varepsilon$ as an $H$-module nonlocal VA structure on $V$ as follows
\begin{align}
  \varepsilon(h,z)v=\varepsilon(h)v\quad \te{for }h\in H,\,v\in V.
\end{align}
Then it is easy to verify that $\varepsilon$ and $\rho$ are compatible,
and $\mathfrak D_\varepsilon^\rho(V)=V$.
}
\end{rem}

Now, we fix a cocommutative $\hbar$-adic nonlocal vertex bialgebra $H$, and an $H$-comodule nonlocal VA $(V,\rho)$.
Note that $$\Hom\(H,\Hom(V,V\wh\ot\C((x))[[\hbar]])\)$$ is a unital associative algebra,
where the multiplication is defined by
\begin{align*}
  (f\ast g)(h,z)=\sum f(h_{(1)},z)g(h_{(2)},z),
\end{align*}
for $f,g\in\Hom\(H,\Hom(V,V\wh\ot\C((z))[[\hbar]])\)$
and the unit $\varepsilon$.
For $$f,g\in\Hom\(H,\Hom(V,V\wh\ot\C((z))[[\hbar]])\),$$ we say that $f$ and $g$ are commute if
\begin{align*}
  [f(h,z_1),g(k,z_2)]=0\quad \te{for }h,k\in H.
\end{align*}
Let $\mathfrak L_H^\rho(V)$ be the set of $H$-module nonlocal VA structures on $V$ that are compatible with $\rho$.
The following is a direct $\hbar$-adic analogue of \cite[Proposition 3.3 and Proposition 3.4]{JKLT-Defom-va}.

\begin{prop}\label{prop:L-H-rho-V-compostition}
Let $\tau$ and $\tau'$ are commuting elements in $\mathfrak L_H^\rho(V)$.
Then $\tau\ast\tau'\in\mathfrak L_H^\rho(V)$
and $\tau\ast\tau'=\tau'\ast\tau$.
Moreover, $\tau\in \mathfrak L_H^\rho\(\mathfrak D_{\tau'}^\rho(V)\)$ and
\begin{align}
  \mathfrak D_\tau^\rho\(\mathfrak D_{\tau'}^\rho(V)\)=\mathfrak D_{\tau\ast\tau'}^\rho(V).
\end{align}
\end{prop}

Recall that an element $\tau\in\mathfrak L_H^\rho(V)$ is said to be invertible if there exists $\tau\inv \in \mathfrak L_H^\rho(V)$, such that $\tau$ and $\tau\inv$ are commute and $\tau\ast\tau\inv=\varepsilon$.
We have the following immediate $\hbar$-adic analogue of \cite[Theorem 3.6]{JKLT-Defom-va}.

\begin{thm}\label{thm:qva}
Let $V_0$ be a VA, and $V=V_0[[\hbar]]$ be an $\hbar$-adic VA.
Let $\tau$ be an invertible element in $\mathfrak L_H^\rho(V)$.
Then $\mathfrak D_\tau^\rho(V)$ is an $\hbar$-adic quantum VA with quantum Yang-Baxter operator $S(z)$ defined by
\begin{align*}
  S(z)(v\ot u)=\sum \tau(u_{(2)},-z)v_{(1)}\ot \tau\inv(v_{(2)},z)u_{(1)}\quad \te{for }u,v\in V,
\end{align*}
where $\rho(u)=\sum u_{(1)}\ot u_{(2)}$ and $\rho(v)=\sum v_{(1)}\ot v_{(2)}$.
\end{thm}

Let $V_1,\dots,V_n$ be $\hbar$-adic VAs and let $H$ be a commutative and cocommutative $\hbar$-adic vertex bialgebra.
Assign an $H$-comodule VA structure $\rho_i$ to $V_i$ for each $1\le i\le n$.
Choose an invertible element $\lambda_{ij}\in\mathfrak L_H^{\rho_i}(V_i)$ for $1\le i,j\le n$.
Suppose further that
\begin{align*}
  [\lambda_{ij}(h,z_1),\lambda_{ik}(h',z_2)]=0\quad\te{for }1\le i,j,k\le n,\,h,h'\in H.
\end{align*}
Define
\begin{align}\label{eq:S-twisted-def}
  S_{ij}(z):&V_i\wh\ot V_j\to V_i\wh\ot V_j\wh\ot \C((z))[[\hbar]];
  \quad v\ot u\mapsto \sum \lambda_{ij}(u_{(2)},-z)v_{(1)}\ot \lambda_{ji}\inv(v_{(2)},z)u_{(1)},\nonumber
\end{align}
where $\rho_i(v)=\sum v_{(1)}\ot v_{(2)}$ and $\rho_j(u)=\sum u_{(1)}\ot u_{(2)}$ are the Sweedler notations.
The following result can be easily verified.

\begin{lem}\label{lem:Sij-Sji}
For $1\le i,j\le n$, we have that
\begin{align*}
  &\sigma S_{ij}(-z)\inv\sigma=S_{ji}(z).
\end{align*}
\end{lem}

%\begin{proof}
%For $v\in V_i$ and $u\in V_j$, we have that
%\begin{align*}
%  &S_{ij}(z)\sigma S_{ji}(-z)\sigma (v\ot u)
%  =S_{ij}(z)\sigma S_{ji}(-z)(u\ot v)\\
%  =&S_{ij}(z)\sigma \sum \lambda_{ji}(v_{(2)},z)u_{(1)}\ot \lambda_{ij}\inv(u_{(2)},-z)v_{(1)}\\
%  =&\sum S_{ij}(z)\lambda_{ij}\inv(u_{(2)},-z)v_{(1)}\ot \lambda_{ji}(v_{(2)},z)u_{(1)}\\
%  =&\sum \lambda_{ij}\(\(\lambda_{ji}(v_{(2)},z)u_{(1)}\)_{(2)},-z\)\(\lambda_{ij}\inv(u_{(2)},-z)v_{(1)}\)_{(1)}\\
%  &\qquad\ot \lambda_{ji}\inv\(\(\lambda_{ij}\inv(u_{(2)},-z)v_{(1)}\)_{(2)},z\)\(\lambda_{ji}(v_{(2)},z)u_{(1)}\)_{(1)}\\
%  =&\sum \lambda_{ij}\(u_{(1),(2)},-z\)\lambda_{ij}\inv(u_{(2)},-z)v_{(1),(1)}\\
%  &\qquad\ot \lambda_{ji}\inv\(v_{(1),(2)},z\)\lambda_{ji}(v_{(2)},z)u_{(1),(1)}\\
%  =&\sum \lambda_{ij}\(u_{(2),(1)},-z\)\lambda_{ij}\inv(u_{(2),(2)},-z)v_{(1)}\\
%  &\qquad\ot \lambda_{ji}\inv\(v_{(2),(1)},z\)\lambda_{ji}(v_{(2),(2)},z)u_{(1)}\\
%  =&\sum \varepsilon(u_{(2)})v_{(1)}\ot \varepsilon(v_{(2)})u_{(1)}=v\ot u.
%\end{align*}
%Similarly, we have that
%\begin{align*}
%  \sigma S_{ji}(-z)\sigma S_{ij}(z)(v\ot u)=v\ot u.
%\end{align*}
%Therefore,
%\begin{align*}
%  \sigma S_{ji}(-z)\sigma=S_{ij}(z)\inv.
%\end{align*}
%Consequently, $\sigma S_{ij}(-z)\inv\sigma=S_{ji}(z)$.
%\end{proof}

Similar to the proof of \cite[Theorem 3.6]{JKLT-Defom-va}, we have the following result.
\begin{prop}\label{prop:qva-twisted-tensor-deform}
$((\mathfrak D_{\lambda_{ii}}^{\rho_i}(V_i))_{i=1}^n,(S_{ij}(z))_{i,j=1}^n)$ is an $\hbar$-adic $n$-quantum VA.
%For each $1\le i\le n$, $\mathfrak D_{\lambda_{ii}}^{\rho_i}(V_i)$ is an $\hbar$-adic quantum VA with quantum Yang-Baxter operator
%$S_{ii}(z)$.
%And for each $1\le i,j\le n$, $S_{ij}(z)$ is a quantum Yang-Baxter operator for the ordered pair $(\mathfrak D_{\lambda_{ii}}^{\rho_i}(V_i),\mathfrak D_{\lambda_{ii}}^{\rho_i}(V_i))$.
%Moreover, we have that
%\begin{align}\label{eq:qyb-hex-sp}
%  S_{ij}^{12}(z_1)S_{ik}^{13}(z_1+z_2)S_{jk}^{23}(z_2)
%  =S_{jk}^{23}(z_2)S_{ik}^{13}(z_1+z_2)S_{ij}^{12}(z_1)\quad\te{for } 1\le i,j,k\le n.
%\end{align}
\end{prop}

Combining Lemma \ref{lem:Sij-Sji} and Propositions \ref{prop:n-qva-to-k-qva}, \ref{prop:qva-twisted-tensor-deform}, we get the following corollary.
\begin{coro}
The twisting tensor product
$$\mathfrak D_{\lambda_{11}}^{\rho_1}(V_1)\wh\ot_{S_{12}}\cdots \wh\ot_{S_{n-1,n}}\mathfrak D_{\lambda_{nn}}^{\rho_n}(V_n)$$ is an $\hbar$-adic quantum VA with the quantum Yang-Baxter operator
\begin{align}
  S(z)=\prod_{i-j=n-1}S_{ij}^{i,n+j}(z)
    \prod_{i-j=n-2}S_{ij}^{i,n+j}(z)
    \cdots
    \prod_{i-j=1-n}S_{ij}^{i,n+j}(z).
\end{align}
\end{coro}

\section{Quantum affine vertex algebras}\label{sec:qaff-va}

Let us refer to the notations provided in Section \ref{sec:VAs}.
In this subsection, we recall the construction of quantum affine VAs introduced in \cite{K-Quantum-aff-va}.
Let $\mathfrak T$ be the set of tuples
\begin{align*}
\tau=\big(\tau_{ij}(z),\tau_{ij}^{1,\pm}(z),\tau_{ij}^{2,\pm}(z),\tau_{ij}^{\epsilon_1,\epsilon_2}(z)\big)_{i,j\in I}^{\epsilon_1,\epsilon_2=\pm},
\end{align*}
where $\tau_{ij}(z),\,\tau_{ij}^{1,\pm}(z),\tau_{ij}^{2,\pm}(z),\tau_{ij}^{\epsilon_1,\epsilon_2}(z)\in \C((z))[[\hbar]]$, such that
\begin{align*}
  &\lim_{\hbar\to 0}\tau_{ij}(z)=\lim_{\hbar\to 0}\tau_{ji}(-z),\quad \lim_{\hbar\to 0}\tau_{ij}^{1,\pm}(z)=-\lim_{\hbar\to 0}\tau_{ji}^{2,\pm}(-z),\\
  &\lim_{\hbar\to 0}\tau_{ij}^{\epsilon_1,\epsilon_2}(z)=\lim_{\hbar\to 0}\tau_{ji}^{\epsilon_2,\epsilon_1}(-z)\in\C[[z]]^\times.
\end{align*}
Then $\mathfrak T$ is a commutative group, where the multiplication $\tau\ast\sigma$ of $\sigma,\tau\in\mathcal T$ is defined by
\begin{align*}
  &(\tau\ast\sigma)_{ij}(z)=\tau_{ij}(z)+\sigma_{ij}(z),\quad (\tau\ast\sigma)_{ij}^{k,\pm}(z)=\tau_{ij}^{k,\pm}(z)+\sigma_{ij}^{k,\pm}(z),\\
  &(\tau\ast\sigma)_{ij}^{\epsilon_1,\epsilon_2}(z)=\tau_{ij}^{\epsilon_1,\epsilon_2}(z)\sigma_{ij}^{\epsilon_1,\epsilon_2}(z),
  \quad\te{for }i,j\in I,\,k=1,2,\,\epsilon_1,\epsilon_2=\pm,
\end{align*}
%for $i,j\in I$, $k=1,2$, $\epsilon_1,\epsilon_2=\pm$,
the inverse of $\tau$ is defined by
\begin{align*}
  \tau\inv=(-\tau_{ij}(z),-\tau_{ij}^{1,\pm}(z),-\tau_{ij}^{2,\pm}(z),\tau_{ij}^{\epsilon_1,\epsilon_2}(z)\inv)_{i,j\in I}^{\epsilon_1,\epsilon_2=\pm},
\end{align*}
and the identity $\varepsilon$ is defined by
\begin{align*}
  \varepsilon_{ij}(z)=0=\varepsilon_{ij}^{1,\pm}(z)=\varepsilon_{ij}^{2,\pm}(z),\quad \varepsilon_{ij}^{\epsilon_1,\epsilon_2}(z)=1\quad\te{for }i,j\in I, \epsilon_1,\epsilon_2=\pm.
\end{align*}

Define $\mathcal M_\tau(\g)$ to be the category consisting of topologically free $\C[[\hbar]]$-modules $W$, equipped with fields $h_{i,\hbar}(z),x_{i,\hbar}^\pm(z)\in\E_\hbar(W)$ ($i\in I$) satisfying the following conditions
\begin{align}
  &[h_{i,\hbar}(z_1),h_{j,\hbar}(z_2)]\label{eq:fqva-rel1}
  =r_i a_{ij} r\ell\pd{z_2}z_1\inv\delta\(\frac{z_2}{z_1}\)+\tau_{ij}(z_1-z_2)-\tau_{ji}(z_2-z_1),
  \\
  &[h_{i,\hbar}(z_1),x_{j,\hbar}^\pm(z_2)]\label{eq:fqva-rel2}
  =\pm x_{j,\hbar}^\pm(z_2)\(r_ia_{ij} z_1\inv\delta\(\frac{z_2}{z_1}\)+ \tau_{ij}^{1,\pm}(z_1-z_2)+\tau_{ji}^{2,\pm}(z_2-z_1)\),\\
  &\tau_{ij}^{\pm,\pm}(z_1-z_2)(z_1-z_2)^{n_{ij}}x_{i,\hbar}^\pm(z_1)x_{j,\hbar}^\pm(z_2)\label{eq:fqva-rel3}
  =\tau_{ji}^{\pm,\pm}(z_2-z_1)(z_1-z_2)^{n_{ij}}x_{j,\hbar}^\pm(z_2)x_{i,\hbar}^\pm(z_1),\\
  &\tau_{ij}^{\pm,\mp}(z_1-z_2)(z_1-z_2)^{2\delta_{ij}}x_{i,\hbar}^\pm(z_1)x_{j,\hbar}^\mp(z_2)\label{eq:fqva-rel4}
  =\tau_{ji}^{\mp,\pm}(z_2-z_1)(z_1-z_2)^{2\delta_{ij}}x_{j,\hbar}^\mp(z_2)x_{i,\hbar}^\pm(z_1).
\end{align}

\begin{prop}\label{prop:universal-M-tau}
There is a unique $\hbar$-adic nonlocal VA $(F_\tau(\g,\ell),Y_\tau,\vac)$, such that
\begin{itemize}
  \item[(a)] $F_\tau(\g,\ell)$ is generated by $\set{h_i,\,x_i^\pm}{i\in I}$,
  and $$(F_\tau(\g,\ell),\{Y_\tau(h_i,z)\}_{i\in I},\{Y_\tau(x_i^\pm,z)\}_{i\in I})$$ is an object of $\mathcal M_\tau(\g)$.

  \item[(b)] For any $\hbar$-adic nonlocal VA $(V,Y,\vac)$ containing $\bar h_i, \bar x_i^\pm$ ($i\in I$),
    such that $$(V,\{Y(\bar h_i,z)\}_{i\in I},\{Y(\bar x_i^\pm,z)\}_{i\in I})$$
    is an object of $\mathcal M_\tau(\g)$,
    then there exists a unique $\hbar$-adic nonlocal VA homomorphism $\varphi:F_\tau(\g,\ell)\to V$,
    such that $\varphi(h_i)=\bar h_i$ and $\varphi(x_i^\pm)=\bar x_i^\pm$ for $i\in I$.
\end{itemize}
\end{prop}

\begin{rem}\label{rem:F-varepsilon=F}
\emph{
Recall the VA $F_{\hat\g}^\ell$ in Definition \ref{de:affVAs}.
View $F_{\hat\g}^\ell[[\hbar]]$ as the natural $\hbar$-adic VA.
Then $F_{\hat\g}^\ell[[\hbar]]\cong F_{\varepsilon}(\g,\ell)$.
}
\end{rem}

Let $H_0$ be the symmetric algebra of the following vector space, and let $H=H_0[[\hbar]]$:
\begin{align*}
  \bigoplus_{i\in I}\(\C[\partial] \(\C \al^\vee_i\oplus\C e_i^+\oplus\C e_i^-\)\).
\end{align*}
Then $(H,\Delta,\varepsilon)$ carries a commutative and cocommutative bialgebra structure, such that ($i\in I$, $n\in\N$):
\begin{align*}
  &\Delta(\partial^n \al^\vee_i)=\partial^n \al^\vee_i\ot 1+1\ot \partial^n \al^\vee_i,\quad
  \varepsilon(\partial^n \al^\vee_i)=0,
  \\
  &\Delta(\partial^n e_i^\pm)=\sum_{k=0}^n\binom{n}{k} \partial^k e_i^\pm\ot \partial^{n-k} e_i^\pm,
  \quad\varepsilon(\partial^n e_i^\pm)=\delta_{n,0}.
\end{align*}
Let $\partial$ be the derivation on $H$ such that
\begin{align*}
  \partial (\partial^n \al^\vee_i)=\partial^{n+1} \al^\vee_i,\quad \partial(\partial^n e_i^\pm)=\partial^{n+1}e_i^\pm,\quad i\in I,\,0\le n\in\Z.
\end{align*}
Define a vertex operator map $Y(\cdot,z)$ on $H$ by $Y(u,z)v=(e^{z\partial}u)v$.
Then $(H,Y,1,\Delta,\varepsilon)$ forms a cocommutative and commutative $\hbar$-adic vertex bialgebra.

\begin{prop}\label{prop:deform-datum}
For $\sigma,\tau\in \mathfrak T$, there exists a deforming triple $(H,\rho,\sigma)$,
where $(F_\tau(\g,\ell),\rho)$ is the $H$-comodule nonlocal VA uniquely determined by
\begin{align}
  \rho(h_i)=h_i\ot 1+\vac\ot \al^\vee_i,\quad \rho(x_i^\pm)=x_i^\pm\ot e_i^\pm,
\end{align}
and $(F_\tau(\g,\ell),\sigma)$ is the $H$-module nonlocal VA uniquely determined by
\begin{align*}
  &\sigma(\al^\vee_i,z)h_j=\vac\ot \sigma_{ij}(z),\quad
  \sigma(\al^\vee_i,z)x_j^\pm=\pm x_j^\pm\ot \sigma_{ij}^{1,\pm}(z),\\
  &\sigma(e_i^\pm,z)h_j=h_j\ot 1\mp\vac\ot \sigma_{ij}^{2,\pm}(z),\quad
  \sigma(e_i^\pm,z)x_j^\epsilon=x_j^\epsilon\ot \sigma_{ij}^{\pm,\epsilon}(z)\inv.
\end{align*}
Moreover, for $\tau,\sigma_1,\sigma_2\in\mathfrak T$, we have that
\begin{align}\label{eq:F-deform-modVA-struct-com}
  [\sigma_1(h_1,z_1),\sigma_2(h_2,z_2)]=0\,\,\te{on }F_\tau(\g,\ell)\quad\te{for }h_1,h_2\in H.
\end{align}
\end{prop}

\begin{prop}\label{prop:classical-limit}
Let $\tau,\sigma\in\mathfrak T$.
Then
\begin{align*}
  \mathfrak D_\sigma^\rho(F_\tau(\g,\ell))=F_{\tau\ast\sigma}(\g,\ell).
\end{align*}
Moreover, $F_\tau(\g,\ell)/\hbar F_\tau(\g,\ell)\cong F_\g^\ell$.
\end{prop}

\begin{thm}\label{thm:S-tau}
For any $\tau\in\mathfrak T$, $F_\tau(\g,\ell)$ is an $\hbar$-adic quantum VA with the quantum Yang-Baxter operator $S_\tau(z)$ defined by
\begin{align*}
  S_\tau(z)(v\ot u)=\sum \tau(u_{(2)},-z)v_{(1)}\ot \tau\inv(v_{(2)},z)u_{(1)}\quad \te{for }u,v\in F_\tau(\g,\ell).
\end{align*}
Moreover, for any $i,j\in I$ and $\epsilon_1,\epsilon_2=\pm$, we have that
\begin{align}
  &S_\tau(z)(h_j\ot h_i)=h_j\ot h_i+\vac\ot\vac\ot \(\tau_{ij}(-z)-\tau_{ji}(z)\),\label{eq:S-tau-1}\\
  &S_\tau(z)(x_j^\pm\ot h_i)=x_j^\pm\ot h_i\pm x_j^\pm\ot \vac \ot \(\tau_{ij}^{1,\pm}(-z)+\tau_{ji}^{2,\pm}(z)\),\label{eq:S-tau-2}\\
  &S_\tau(z)(h_j\ot x_i^\pm)=h_j\ot x_i^\pm\mp\vac\ot x_i^\pm
  \ot\(\tau_{ij}^{2,\pm}(-z)+\tau_{ji}^{1,\pm}(z)\),\label{eq:S-tau-3}\\
  &S_\tau(z)(x_j^{\epsilon_1}\ot x_i^{\epsilon_2})=x_j^{\epsilon_1}\ot x_i^{\epsilon_2}\ot\tau_{ji}^{\epsilon_1,\epsilon_2}(z)  \tau_{ij}^{\epsilon_2,\epsilon_1}(-z)\inv.\label{eq:S-tau-4}
\end{align}
\end{thm}

For $0\ne P(z)\in\C(z)[[\hbar]]$ and $g(q)=\sum_{k=1}^na_kq^{m_k}\in\Z[q]$, we let
\begin{align}
  P(z)^{g(q)}=\prod_{k=1}^n \(q^{m_k\pd{z}}P(z)^{a_k}\) \in \C(z)[[\hbar]].
\end{align}
It is straightforward to check that
\begin{align*}
  P(z)^{g_1(q)}P(z)^{g_2(q)}=P(z)^{g_1(q)+g_2(q)},\quad \(P(z)^{g_1(q)}\)^{g_2(q)}=P(z)^{g_1(q)g_2(q)}
\end{align*}
for any $g_1(q),g_2(q)\in\Z[q]$, and
\begin{align*}
  \(P(z_1)^{g(q)}\)|_{z_1=-z}=\(Q(z)\)^{g(q\inv)},\quad\te{with}\,\,Q(z)=P(-z).
\end{align*}
For an $\ell\in\C$, we define the element $\wh\ell\in\mathfrak T$ as follows:
\begin{align}
  &\wh\ell_{ij}(z)=-[r_ia_{ij}]_{q^{\pd{z}}}[r\ell]_{q^{\pd{z}}}q^{r\ell\pd{z}}\pdiff{z}{2}\log f(z)
    -r_ia_{ij}r\ell z^{-2},\label{eq:tau-1}\\
  &\wh\ell_{ij}^{1,\pm}(z)=\wh\ell_{ji}^{2,\pm}(z)=[r_ia_{ij}]_{q^{\pd{z}}}q^{r\ell\pd{z}}\pd{z}\log f(z)
    -r_ia_{ij}z\inv,\label{eq:tau-2}\\
  %&\wh\ell_{ij}^{2,\pm}(z)=[a_{ji}]_{q^{r_j\pd{z}}}q^{r\ell\pd{z}}\pd{z}\log f(z)
%    -a_{ji}z\inv,\label{eq:tau-2-2}\\
  &\wh\ell_{ij}^{\pm,\pm}(z)=\begin{cases}
  f(z)^{q^{-r_ia_{ii}}-1},&\mbox{if }a_{ij}> 0,\\
  z\inv f(z)^{q^{-r_ia_{ij}}},&\mbox{if }a_{ij}\le 0,
  \end{cases}
  \label{eq:tau-3}\\
  &\wh\ell_{ij}^{+,-}(z)=z^{-\delta_{ij}}(z+2r\ell\hbar)^{\delta_{ij}},\label{eq:tau-4}\\
  &\wh\ell_{ij}^{-,+}(z)=z^{-\delta_{ij}}(z-2r\ell\hbar)^{\delta_{ij}}
    f(z)^{q^{r_ia_{ij}}-q^{-r_ia_{ij}}},\label{eq:tau-5}
\end{align}
where %$q_i=q^{r_i}$ and $g_{ij,\hbar}(z)=(1-q_i^{a_{ij}}e^{-z})/(q_i^{a_{ij}}-e^{-z})\in\C((z))[[\hbar]]$.
\begin{align}
  f(z)=e^{z/2}-e^{-z/2}\in\C[[z]].
\end{align}
In the rest of this paper, we denote $F_{\wh\ell}(\g,\ell)$ by
\begin{align}\label{eq:def-F-qva}
  F_{\hat\g,\hbar}^{\ell}.
\end{align}

\begin{lem}\label{lem:M-wh-ell}
The category $\mathcal M_{\wh\ell}(\g)$ consisting of topologically free $\C[[\hbar]]$-modules $W$ equipped with fields
$h_{i,\hbar}(z),x_{i,\hbar}^\pm(z)\in\E_\hbar(W)$ ($i\in I$) satisfying the following relations
\begin{align}
  &[h_{i,\hbar}(z_1),h_{j,\hbar}(z_2)]
  =[r_ia_{ij}]_{q^{\pd{z_2}}}[r\ell]_{q^{\pd{z_2}}}\label{eq:local-h-1}\\
  \times&
  \(\iota_{z_1,z_2}q^{-r\ell\pd{z_2}}-\iota_{z_2,z_1}q^{r\ell\pd{z_2}}\)
  \pd{z_1}\pd{z_2}\log f(z_1-z_2),\nonumber\\
  &[h_{i,\hbar}(z_1),x_{j,\hbar}^\pm(z_2)]
  =\pm x_{j,\hbar}^\pm(z_2)[r_ia_{ij}]_{q^{r_i\pd{z_2}}}\label{eq:local-h-2}\\
  \times&
  \(\iota_{z_1,z_2}q^{-r\ell\pd{z_2}}-\iota_{z_2,z_1}q^{r\ell\pd{z_2}}\)
  \pd{z_1}\log f(z_1-z_2),\nonumber\\
  &\iota_{z_1,z_2}f(z_1-z_2)^{-\delta_{ij}+q^{-r_ia_{ij}}}
  x_{i,\hbar}^\pm(z_1)x_{j,\hbar}^\pm(z_2)\label{eq:local-h-3}\\
    =&\iota_{z_2,z_1} f(-z_2+z_1)^{-\delta_{ij}+q^{r_ia_{ij}}}
    x_{j,\hbar}^\pm(z_2)x_{i,\hbar}^\pm(z_1),\nonumber\\
  &\iota_{z_1,z_2}f(z_1-z_2)^{\delta_{ij}+\delta_{ij}q^{2r\ell}}
  x_{i,\hbar}^+(z_1)x_{j,\hbar}^-(z_2)\label{eq:local-h-4}\\
    =&\iota_{z_2,z_1}f(z_1-z_2)^{\delta_{ij}+\delta_{ij}q^{2r\ell}+q^{-r_ia_{ij}}-q^{r_ia_{ij}}}
    x_{j,\hbar}^-(z_2)x_{i,\hbar}^+(z_1).\nonumber
\end{align}
Moreover, the topologically free $\C[[\hbar]]$-module $F_{\hat\g,\hbar}^\ell$ equipped fields $Y_{\wh\ell}(h_i,z)$ and $Y_{\wh\ell}(x_i^\pm,z)$ ($i\in I$) is an object of $\mathcal M_{\wh\ell}(\g)$.
\end{lem}

Define
\begin{align}
   f_0(z)=\frac{f(z)}{z}=\frac{e^{z/2}-e^{-z/2}}{z}=\sum_{n\ge 0}\frac{z^{2n}}{4^n(2n+1)!}\in 1+z^2\C[[z^2]].
\end{align}

\begin{de}\label{de:V-tau}
For $\ell\in\C$, we let $R_{1}^\ell$ be the minimal closed ideal of
$F_{\hat\g,\hbar}^{\ell}$ such that $[R_{1}^\ell]=R_{1}^\ell$ and contains the following elements
\begin{align}
  &\(x_i^+\)_0x_i^--\frac{1}{q^{r_i}-q^{-r_i}}\(\vac-E_\ell(h_i)\),\quad
  \(x_i^+\)_1x_i^-+\frac{2r\ell\hbar}{q^{r_i}-q^{-r_i}}E_\ell(h_i)\quad\te{for } i\in I,\label{eq:x+1x}\\
  &\(x_i^\pm\)_0^{m_{ij}}x_j^\pm\quad\te{for } i,j\in I\,\,\te{with}\,\,a_{ij}\le 0,\label{eq:serre}
\end{align}
where
\begin{align}\label{eq:def-E-h}
  &E_\ell(h_i)=\(\frac{f_0(2r_i\hbar+2r\ell\hbar)}{f_0(2r_i\hbar-2r\ell\hbar)}\)^\half
  \exp\(\(-q^{-r\ell\partial}2\hbar f_0(2\partial\hbar) h_i\)_{-1}\)\vac.
\end{align}
Define
\begin{align}
  V_{\hat\g,\hbar}^{\ell}=F_{\hat\g,\hbar}^{\ell}/R_{1}^\ell.
\end{align}
\end{de}

\begin{de}\label{de:L-tau}
Let $\ell\in \Z_+$, and
let $R_{2}^\ell$ be the minimal closed ideal of $V_{\hat\g,\hbar}^{\ell}$ such that $[R_{2}^\ell]=R_{2}^\ell$ and contains the following elements
\begin{align}
  & \(x_i^\pm\)_{-1}^{r\ell/r_i}x_i^\pm\quad\te{for } i\in I.\label{eq:integrable}
\end{align}
Define
\begin{align}
  L_{\hat\g,\hbar}^{\ell}=V_{\hat\g,\hbar}^{\ell}/R_{2}^\ell.
\end{align}
\end{de}

\begin{rem}
{\em
For the notation $L_{\hat\g,\hbar}^{\ell}$. we always assume that $\ell\in\Z_+$.
}
\end{rem}

%\begin{rem}
%\emph{
%Let $X=V,L$.
%The definition of $F_{\hat\g,\hbar}^\ell$ (resp. $X_{\hat\g,\hbar}^\ell$) is slightly different from the $\hbar$-adic quantum VA $F_{\wh\ell}(A,\ell)$ (resp. $X_{\hat\g,\hbar}(\ell,0)$) given in \cite{K-Quantum-aff-va}.
%Recall from \cite{K-Quantum-aff-va}*{Section 6} that $F_{\wh\ell}(A,\ell)$ (resp. $X_{\hat\g,\hbar}(\ell,0)$) is generated by $h_{i,\hbar}$, $x_{i,\hbar}^\pm$ for $i\in I$.
%Then there is a unique $\hbar$-adic quantum VA isomorphism from $F_{\hat\g,\hbar}^\ell$ (resp. $X_{\hat\g,\hbar}^\ell$) to $F_{\wh\ell}(A,\ell)$ (resp. $X_{\hat\g,\hbar}(\ell,0)$) such that
%\begin{align*}
%    h_i\mapsto [r_i]_{q^\partial}\inv h_{i,\hbar},\quad x_i^\pm\mapsto x_{i,\hbar}^\pm\quad \te{for }i\in I.
%\end{align*}
%}
%\end{rem}

Proposition \ref{prop:universal-M-tau} and Definitions \ref{de:V-tau}, \ref{de:L-tau}
provides the following result.
\begin{prop}\label{prop:universal-qaff}
Let $(V,Y,\vac)$ be an $\hbar$-adic nonlocal VA containing a subset $$\set{\bar h_i,\bar x_i^\pm}{i\in I},$$
such that
\begin{align*}
    (V,\{Y(\bar h_i,z)\}_{i\in I},\{Y(\bar x_i^\pm,z)\}_{i\in I})\in\obj \mathcal M_{\wh\ell}(\g).
\end{align*}
Suppose that
\begin{align*}
    &(\bar x_i^+)_0\bar x_i^-=\frac{1}{q^{r_i}-q^{-r_i}}\(\vac-E_\ell(\bar h_i)\),\quad
    (\bar x_i^+)_1\bar x_i^-=-\frac{2r\ell\hbar}{q^{r_i}-q^{-r_i}}E_\ell(\bar h_i)\quad\te{for }i\in I,\\
    &(\bar x_i^\pm)_0^{m_{ij}}\bar x_j^\pm=0\quad\te{for }i,j\in I\,\,\te{with}\,\,a_{ij}\le 0.
\end{align*}
Then the unique $\hbar$-adic nonlocal VA homomorphism $\varphi:F_{\hat\g,\hbar}^\ell\to V$ provided in Proposition \ref{prop:universal-M-tau} factor through $V_{\hat\g,\hbar}^\ell$.
Suppose further that $\ell\in\Z_+$ and
\begin{align*}
    (\bar x_i^\pm)_{-1}^{r\ell/r_i}\bar x_i^\pm=0\quad\te{for }i\in I.
\end{align*}
Then $\varphi$ also factor through $L_{\hat\g,\hbar}^\ell$.
\end{prop}

It is straightforward to see that:
\begin{lem}\label{lem:normal-ordering-general}
Let $W$ be a topologically free $\C[[\hbar]]$-module, and let $\zeta_i(z)\in\E_\hbar(W)$ for $i\in I$.
Suppose that
\begin{align}\label{eq:q-local}
  &\iota_{z_1,z_2}f(z_1-z_2)^{-\delta_{ij}+q^{-r_ia_{ij}}}
  \zeta_i(z_1)\zeta_j(z_2)
  =\iota_{z_2,z_1}f(z_1-z_2)^{-\delta_{ij}+q^{r_ia_{ij}}}
    \zeta_j(z_2)\zeta_i(z_1).
\end{align}
For any $n\in\Z_+$ and any $i_1,\dots,i_n\in I$, we define
\begin{align}\label{eq:def-normal-ordering}
  &\:\zeta_{i_1}(z_1)\cdots \zeta_{i_n}(z_n)\;
  =\(\prod_{1\le s<t\le n} f(z_s-z_t)^{-\delta_{i_s,i_t}+q^{-r_{i_s}a_{i_s,i_t}}}
  \)
  \zeta_{i_1}(z_1)\cdots \zeta_{i_n}(z_n).
\end{align}
Then for any $\sigma\in S_n$, we have that
\begin{align*}
  &\:\zeta_{i_{\sigma(1)}}(z_{\sigma(1)})\cdots \zeta_{i_{\sigma(n)}}(z_{\sigma(n)})\;
  =\(\prod_{s<t:\sigma(s)>\sigma(t)}(-1)^{\delta_{i_s,i_t}-1}\)
  \:\zeta_{i_1}(z_1)\cdots \zeta_{i_n}(z_n)\;.
\end{align*}
Moreover, we have that
\begin{align*}
  \:\zeta_{i_1}(z_1)\cdots \zeta_{i_n}(z_n)\;\in\E_\hbar^{(n)}(W).
\end{align*}
\end{lem}

The following result is a generalization of \cite{K-Quantum-aff-va}.

\begin{prop}\label{prop:normal-ordering-rel-general}
Let $V$ be an $\hbar$-adic nonlocal VA and let $\set{\zeta_i}{i\in I}\subset V$ such that
$\set{Y(\zeta_i,z)}{i\in I}$ satisfies the relations \eqref{eq:q-local}.
Then we have that
\begin{align}
  &Y\(\(\zeta_i\)_0^n\zeta_j,z\)
  =\:Y(\zeta_i,z+r_i((n-1)a_{ii}+a_{ij})\hbar)
    Y(\zeta_i,z+r_i((n-2)a_{ii}+ a_{ij})\hbar)\\
    &\qquad\cdots Y(\zeta_i,z+r_ia_{ij}\hbar)Y(\zeta_j,z)\;,\nonumber\\
  &\Sing_{z_1,\dots,z_n}Y(\zeta_i,z_1)\cdots Y(\zeta_i,z_n)\zeta_j
    =\prod_{a=1}^n\frac{1}{z_a-r_i((n-a)a_{ii}+a_{ij})\hbar}\(\zeta_i\)_0^n\zeta_j,\\
  &Y\(\(\zeta_k\)_{-1}^n\vac,z\)=\prod_{a=1}^{n-1}f_0(2ar_k\hbar)\\
  &\quad\times \:Y(\zeta_k,z+2(n-1)r_k\hbar)
    Y(\zeta_k,z+2(n-2)r_k\hbar)\cdots Y(\zeta_k,z)\;,\nonumber\\
  &\Rat_{z_1\inv,\dots,z_n\inv}Y(\zeta_k,z_1)\cdots Y(\zeta_k,z_n)\vac
    =\prod_{a=1}^n\frac{z_a}{z_a-2(n-a)r_k\hbar}\(\zeta_k\)_{-1}^n\vac,
\end{align}
where $i,j,k\in I$ with $a_{ij}\le 0$ and $n\in\Z_+$.
\end{prop}

\begin{coro}\label{coro:normal-ordering-rel-general}
Let $V$ and $\set{\zeta_i}{i\in I}\subset V$ be as in Proposition \ref{prop:normal-ordering-rel-general}.
Let $i,j\in I$ with $a_{ij}<0$. Then $\(\zeta_i\)_0^{m_{ij}}\zeta_j=0$ if and only if
\begin{align*}
  \:Y(\zeta_i,z-r_ia_{ij}\hbar)
    Y(\zeta_i,z-r_i(a_{ij}-2)\hbar)\cdots Y(\zeta_i,z+r_ia_{ij}\hbar)Y(\zeta_j,z)\;\vac=0.
\end{align*}
Suppose that $\ell\in\Z_+$. Then $\(\zeta_i\)_{-1}^{r\ell/r_i}\zeta_i=0$ if
and only if
\begin{align*}
  \:Y(\zeta_i,z+2r\ell\hbar)
    Y(\zeta_i,z+2(r\ell-r_i)\hbar)\cdots Y(\zeta_i,z)\;\vac=0.
\end{align*}
\end{coro}

\begin{prop}\label{prop:ideal-def-alt}
Define
\begin{align*}
  A_{i}(z)=&Y_{\wh\ell}(x_i^+,z)^-x_i^-
  -\frac{1}{q^{r_i}-q^{-r_i}}\(\frac{\vac}{z}-\frac{E_\ell(h_i)}{z+2r\ell\hbar}\)&&\te{for }i\in I,\\
  Q_{ij}^\pm(z_1,&\dots,z_{m_{ij}})=%\Sing_{z_1,z_2,\dots,z_{m_{ij}}}
    Y_{\wh\ell}(x_i^\pm,z_1)^-
  \cdots Y_{\wh\ell}(x_i^\pm,z_{m_{ij}})^-x_j^\pm&&\te{for }i,j\in I\,\,\te{with}\,\,a_{ij}<0.
\end{align*}
Then $R_{1}^\ell$ is the minimal closed ideal of $F_{\hat\g,\hbar}^{\ell}$ such that $[R_{1}^\ell]=R_{1}^\ell$,
and contains all coefficients of $A_{i}(z)$ for $i\in I$ and all coefficients of $Q_{ij}^\pm(z_1,\dots,z_{m_{ij}})$ for $(i,j)\in\mathbb I$.
Moreover, suppose $\ell\in\Z_+$.
For $i\in I$, set
\begin{align*}
  &M_{i}^\pm(z_1,\dots,z_{r\ell/r_i})=\Sing_{z_1,\dots,z_{r\ell/r_i}}z_1\inv\cdots z_{r\ell/r_i}\inv
   Y_{\wh\ell}(x_i^\pm,z_1)\cdots Y_{\wh\ell}(x_i^\pm,z_{r\ell/r_i})x_i^\pm.
\end{align*}
Then $R_{2}^\ell$ is the minimal closed ideal of $V_{\hat\g,\hbar}^{\ell}$
such that $[R_{2}^\ell]=R_{2}^\ell$ and contains all coefficients of $M_{i}^\pm(z_1,\dots,z_{r\ell/r_i})$ for $i\in I$.
\end{prop}

%We need the following generalization of Proposition \ref{prop:normal-ordering-rel}.

\begin{thm}\label{thm:quotient-algs}
Let $\ell\in\C$. Then $V_{\hat\g,\hbar}^{\ell}$ is an $\hbar$-adic quantum VA.
Moreover, if $\ell\in \Z_+$, then $L_{\hat\g,\hbar}^{\ell}$ is also an $\hbar$-adic quantum VA.
Furthermore, the quantum Yang-Baxter operators $S_{\ell,\ell}(z)$ of both $V_{\hat\g,\hbar}^{\ell}$ and $L_{\hat\g,\hbar}^{\ell}$ satisfy the following relations
\begin{align*}
  &S_{\ell,\ell}(z)(h_j\ot h_i)=h_j\ot h_i
  +\vac\ot\vac\ot
  \pdiff{z}{2}\log f(z)^{ [r_ia_{ij}]_{q}[r\ell]_{q}[r\ell]_q(q-q\inv) },\\
  %[a_{ij}]_{q^{r_i\pd{z}}}[r\ell/r_j]_{q^{r_j\pd{z}}}^2\(q^{\pd{z}}-q^{-\pd{z}}\)\pdiff{z}{2}\log f(z),\\
  &S_{\ell,\ell}(z)(x_j^\pm\ot h_i)=x_j^\pm\ot h_i
  \pm x_j^\pm\ot \vac \ot
  \pd{z}\log f(z)^{ [r_ia_{ij}]_{q}[r\ell]_q(q-q\inv) }  ,\\
  %[a_{ij}]_{q^{r_i\pd{z}}}[r\ell]_{q^{\pd{z}}}\(q^{\pd{z}}-q^{-\pd{z}}\)\pd{z}\log f(z),\\
  &S_{\ell,\ell}(z)(h_j\ot x_i^\pm)=h_j\ot x_i^\pm
  \mp\vac\ot x_i^\pm
  \ot \pd{z}\log f(z)^{ [r_ja_{ji}]_{q}[r\ell]_q(q-q\inv) },\\
  %[a_{ji}]_{q^{r_j\pd{z}}}[r\ell]_{q^{\pd{z}}}\(q^{\pd{z}}-q^{-\pd{z}}\)\pd{z}\log f(z),\\
  &S_{\ell,\ell}(z)(x_j^{\epsilon_1}\ot x_i^{\epsilon_2})=x_j^{\epsilon_1}\ot x_i^{\epsilon_2}\ot
  f(z)^{q^{-\epsilon_1\epsilon_2r_ia_{ij}}
    -q^{\epsilon_1\epsilon_2r_ia_{ij}}}.
\end{align*}
\end{thm}

\section{Twisted tensor product of quantum affine vertex algebras}\label{sec:construct-n-qvas-qaffva}

Let $\ell'$ be another complex number. We have $\hbar$-adic quantum VAs $F_{\hat\g,\hbar}^{\ell'}$, $V_{\hat\g,\hbar}^{\ell'}$ and $L_{\hat\g,\hbar}^{\ell'}$.
%We denote the elements $h_{i,\hbar}$, $x_{i,\hbar}^\pm$ and $E_\ell(h_{i,\hbar})$ in $F_{\wh{\ell'}}(A,\bar\ell)$ by
%$h_{i,\hbar}'$, $x_{i,\hbar}^{\pm\prime}$ and $E_\ell(h_{i,\hbar})'$, respectively,
%to distinguish from the elements in $F_{\wh{\ell}}(A,\ell)$.
%
Define the element $\wh{\ell,\ell'}\in\mathfrak T$ as follows:
\begin{align}
  &\wh{\ell,\ell'}_{ij}(z)=
  \pdiff{z}{2}\log f(z)^{ [r_ia_{ij}]_{q}[r\ell]_{q}[r\ell']_q(q\inv-q) }  ,\\
  %[a_{ij}]_{q^{r_i\pd{z}}}[r\ell/r_j]_{q^{r_j\pd{z}}}[r\ell']_{q^{\pd{z}}}
  %  \(q^{-\pd{z}}-q^{\pd{z}}\)\pdiff{z}{2}\log f(z),\\
  &\wh{\ell,\ell'}_{ij}^{1,\pm}(z)=
  \pd{z}\log f(z)^{ [r_ia_{ij}]_{q}[r\ell']_q(q-q\inv) },\\
  %[a_{ij}]_{q^{r_i\pd{z}}}[r\ell']_{q^{\pd{z}}}
  %  \(q^{\pd{z}}-q^{-\pd{z}}\)\pd{z}\log f(z),\\
  &\wh{\ell,\ell'}_{ij}^{2,\pm}(z)=
  \pd{z}\log f(z)^{ [r_ja_{ji}]_{q}[r\ell]_q(q-q\inv) },\\
  %[a_{ji}]_{q^{r_j\pd{z}}}[r\ell]_{q^{\pd{z}}}
  %  \(q^{\pd{z}}-q^{-\pd{z}}\)\pd{z}\log f(z),\\
  &\wh{\ell,\ell'}_{ij}^{\epsilon_1,\epsilon_2}(z)=
  f(z)^{q^{-\epsilon_1\epsilon_2r_ia_{ij}}
    -q^{\epsilon_1\epsilon_2r_ia_{ij}}}.
\end{align}

It is straightforward to check the following result.
\begin{lem}\label{lem:S-twisted}
Define %$S_{\ell,\ell'}(z):F_{\hat\g,\hbar}^{\ell}\wh\ot F_{\hat\g,\hbar}^{\ell'}\to
%    F_{\hat\g,\hbar}^{\ell}\wh\ot F_{\hat\g,\hbar}^{\ell'}\wh\ot\C((z))[[\hbar]]$ by
\begin{align}\label{eq:def-S-twisted}
  S_{\ell,\ell'}(z):F_{\hat\g,\hbar}^{\ell}\wh\ot F_{\hat\g,\hbar}^{\ell'}&\to
    F_{\hat\g,\hbar}^{\ell}\wh\ot F_{\hat\g,\hbar}^{\ell'}\wh\ot\C((z))[[\hbar]]\nonumber\\
  v\ot u&\mapsto \sum \wh{\ell,\ell'}(u_{(2)},-z)v\ot u_{(1)}.
\end{align}
%for $v\in F_{\hat\g,\hbar}^{\ell}$, $u\in F_{\hat\g,\hbar}^{\ell'}$.
Then for $i,j\in I$ we have that
\begin{align}
  &S_{\ell,\ell'}(z)(h_j\ot h_i)=h_j\ot h_i+\vac\ot\vac\label{eq:S-twisted-1}
  \ot \pdiff{z}{2}\log f(z)^{[r_ia_{ij}]_{q}[r\ell]_{q} [r\ell']_q(q-q\inv) }
  ,\\
  %[a_{ij}]_{q^{r_i\pd{z}}}[r\ell/r_j]_{q^{r_j\pd{z}}}[r\ell']_{q^{\pd{z}}}
  %\(q^{\pd{z}}-q^{-\pd{z}}\)\pdiff{z}{2}\log f(z),\nonumber\\
  &S_{\ell,\ell'}(z)(x_j^\pm\ot h_i)=x_j^\pm\ot h_i\pm x_j^\pm\ot \vac \label{eq:S-twisted-2}
  \ot
  \pd{z}\log f(z)^{ [r_ia_{ij}]_{q}[r\ell']_q(q-q\inv) },\\
  %[a_{ij}]_{q^{r_i\pd{z}}}[r\ell']_{q^{\pd{z}}}
  %\(q^{\pd{z}}-q^{-\pd{z}}\)\pd{z}\log f(z),\nonumber\\
  &S_{\ell,\ell'}(z)(h_j\ot x_i^\pm)=h_j\ot x_i^\pm\mp\vac\ot x_i^\pm\label{eq:S-twisted-3}
  \ot
  \pd{z}\log f(z)^{ [r_ja_{ji}]_{q}[r\ell]_q(q-q\inv) },\\
  %[a_{ji}]_{q^{r_j\pd{z}}}[r\ell]_{q^{\pd{z}}}\(q^{\pd{z}}-q^{-\pd{z}}\)\pd{z}\log f(z),\nonumber\\
  &S_{\ell,\ell'}(z)(x_j^{\epsilon_1}\ot x_i^{\epsilon_2})=x_j^{\epsilon_1}\ot x_i^{\epsilon_2}\ot f(z)^{q^{-\epsilon_1\epsilon_2r_ia_{ij}}
    -q^{\epsilon_1\epsilon_2r_ia_{ij}}}.
    \label{eq:S-twisted-4}
\end{align}
\end{lem}

\begin{rem}\label{rem:ell=ell-prime-case}
{\em
If $\ell=\ell'$, then the map $S_{\ell,\ell'}(z)$ defined in Lemma \ref{lem:S-twisted} is equivalent to the map $S_{\ell,\ell}(z)$ given in Theorem \ref{thm:quotient-algs}.
}
\end{rem}

\begin{lem}\label{lem:S-twisted-inverse}
$\sigma S_{\ell,\ell'}(-z)\inv \sigma=S_{\ell',\ell}(z)$.
\end{lem}

\begin{proof}
It is straightforward to check that
\begin{align*}
  \sigma S_{\ell,\ell'}(-z)\inv \sigma(v\ot u)=S_{\ell',\ell}(z)(v\ot u)\quad\te{for }u,v\in\set{h_i,x_i^\pm}{i\in I}.
\end{align*}
%
%we have that
%\begin{align*}
%  &\sigma S_{\ell,\ell'}(-z)\inv \sigma(h_j\ot h_i)
%  =\sigma S_{\ell,\ell'}(-z)\inv (h_i\ot h_j)=h_j\ot h_i\\
%  &+\vac\ot\vac\ot
%  \pdiff{z}2\log f(z)^{ [a_{ij}]_{q^{r_i}}[r\ell/r_j]_{q^{r_j}}[r\ell']_q(q-q\inv) }\\
%  %[a_{ij}]_{q^{r_i\pd{z}}}[r\ell/r_j]_{q^{r_j\pd{z}}}[r\ell']_{q^{\pd{z}}}
%  %\(q^{\pd{z}}-q^{-\pd{z}}\)\pdiff{z}2\log f(z)\\
%  =&S_{\ell',\ell}(z)(h_j\ot h_i),\\
%  &\sigma S_{\ell,\ell'}(-z)\inv \sigma(x_j^\pm\ot h_i)
%  =\sigma S_{\ell,\ell'}(-z)\inv (h_i\ot x_j^\pm)\\
%  =&x_j^\pm\ot h_i\pm x_j^\pm\ot\vac\ot
%  \pd{z}\log f(z)^{ [a_{ij}]_{q^{r_i}}[r\ell]_q(q-q\inv) }\\
%  %[a_{ij}]_{q^{r_i\pd{z}}}[r\ell]_{q^{\pd{z}}}
%  %\(q^{\pd{z}}-q^{-\pd{z}}\)\pd{z}\log f(z)\\
%  =&S_{\ell',\ell}(z)(x_j^\pm\ot h_i),\\
%  &\sigma S_{\ell,\ell'}(-z)\inv \sigma(h_j\ot x_i^\pm)
%  =\sigma S_{\ell,\ell'}(-z)\inv (x_i^\pm\ot h_j)\\
%  =&h_j\ot x_i^\pm\mp \vac\ot x_i^\pm\ot
%  \pd{z}\log f(z)^{ [a_{ji}]_{q^{r_j}}[r\ell']_q(q-q\inv) }  \\
%  %[a_{ji}]_{q^{r_j\pd{z}}}[r\ell']_{q^{\pd{z}}}
%  %\(q^{\pd{z}}-q^{-\pd{z}}\)\pd{z}\log f(z)\\
%  =&S_{\ell',\ell}(z)(h_j\ot x_i^\pm),\\
%  &\sigma S_{\ell,\ell'}(-z)\inv \sigma(x_j^{\epsilon_1}\ot x_i^{\epsilon_2})
%  =\sigma S_{\ell,\ell'}(-z)\inv (x_i^{\epsilon_2}\ot x_j^{\epsilon_1})\\
%  =&x_j^{\epsilon_1}\ot x_i^{\epsilon_2}\ot
%    f(z)^{q^{-\epsilon_1\epsilon_2r_ia_{ij}}
%    -q^{\epsilon_1\epsilon_2r_ia_{ij}}}
%  =S_{\ell',\ell}(z)(x_j^{\epsilon_1}\ot x_i^{\epsilon_2}).
%\end{align*}
Notice that both $\sigma S_{\ell,\ell'}(-z)\inv \sigma$ and $S_{\ell',\ell}(z)$ are quantum Yang-Baxter operators for the ordered pair
$(F_{\hat\g,\hbar}^{\ell'},F_{\hat\g,\hbar}^{\ell})$.
Since both $F_{\hat\g,\hbar}^{\ell}$ and $F_{\hat\g,\hbar}^{\ell'}$ are generated by $\{h_i,\,x_i^\pm\,\mid\,i\in I\}$,
one has $\sigma S_{\ell,\ell'}(-z)\inv \sigma=S_{\ell',\ell}(z)$.
\end{proof}

\begin{lem}\label{lem:n-qva-F}
For $\ell_1,\ell_2,\dots,\ell_n\in\C$, we have that $$((F_{\hat\g,\hbar}^{\ell_i})_{i=1}^n,(S_{\ell_i,\ell_j}(z))_{i,j=1}^n)$$ is an $\hbar$-adic $n$-quantum VA.
\end{lem}

\begin{proof}
For $1\le i,j\le n$, we define $\lambda_{ij}\in\mathfrak T$ as follows
\begin{align*}
  \lambda_{ij}=\begin{cases}
                 \wh\ell_i, & \mbox{if } i=j, \\
                 \wh{\ell_i,\ell_j}, & \mbox{if } i<j \\
                 \varepsilon, & \mbox{if }i>j.
               \end{cases}
\end{align*}
Proposition \ref{prop:deform-datum} provides deforming triples $(H,\rho,\lambda_{ij})$ of $F_\varepsilon(A,\ell_i)$ for any $1\le i,j\le n$.
From \eqref{eq:F-deform-modVA-struct-com}, we see that $\lambda_{ij}$ and $\lambda_{ik}$ are commute for any $1\le i,j,k\le n$.
Then Proposition \ref{prop:qva-twisted-tensor-deform} provides an $\hbar$-adic $n$-quantum VA
\begin{align*}
  &\(\(\mathfrak D_{\lambda_{ii}}^\rho(F_\varepsilon(\g,\ell_i))\)_{i\in I},\(S_{ij}(z)\)_{i,j\in I}\).
\end{align*}
Comparing \eqref{eq:S-twisted-def} with Lemmas \ref{lem:S-twisted}, \ref{lem:S-twisted-inverse} and Remark \ref{rem:ell=ell-prime-case}, we have that
\begin{align*}
  S_{ij}(z)=S_{\ell_i,\ell_j}(z)\quad\te{for }1\le i,j\le n.
\end{align*}
It follows from Proposition \ref{prop:classical-limit} that
\begin{align*}
  \mathfrak D_{\lambda_{ii}}^\rho(F_\varepsilon(\g,\ell_i))
  =\mathfrak D_{\wh\ell_i}^\rho(F_\varepsilon(\g,\ell_i))
  =F_{\wh\ell_i}(\g,\ell_i)=F_{\wh\g,\hbar}^{\ell_i}.
\end{align*}
Therefore, we complete the proof of lemma.
\end{proof}

\begin{lem}\label{lem:special-tau-tech1}
For any $i,j\in I$, we have that
\begin{align}
  &2\hbar f_0\(2\hbar\pd{z}\)\wh{\ell,\ell'}_{ij}(z)
  =\( q^{-r\ell\pd{z}}-q^{r\ell\pd{z}} \)\wh{\ell,\ell'}_{ij}^{1,+}(z),\label{eq:special-tau-tech1-3}\\
  &2\hbar f_0\(2\hbar\pd{z}\)\wh{\ell,\ell'}_{ij}(z)
  =\( q^{-r\ell'\pd{z}}-q^{r\ell'\pd{z}} \)\wh{\ell,\ell'}_{ij}^{2,+}(z),\\
  &2\hbar f_0\(2\hbar\pd{z}\)\wh{\ell,\ell'}_{ij}^{1,\pm}(z)\label{eq:special-tau-tech1-4}
  =%\(q^{-r\ell'\pd{z}}-q^{r\ell'\pd{z}}\)\(q^{-r_ia_{ij}\pd{z}}-q^{r_ia_{ij}\pd{z}}\)
  \log f(z)^{ (q^{r_ia_{ij}}-q^{-r_ia_{ij}})(q^{r\ell'}-q^{-r\ell'}) },\\
  &2\hbar f_0\(2\hbar\pd{z}\)\wh{\ell,\ell'}_{ij}^{2,\pm}(z)
  =%&\(q^{-r\ell\pd{z}}-q^{r\ell\pd{z}}\)\(q^{-r_ia_{ij}\pd{z}}-q^{r_ia_{ij}\pd{z}}\)
  \log f(z)^{ (q^{r_ia_{ij}}-q^{-r_ia_{ij}})(q^{r\ell}-q^{-r\ell}) }.
\end{align}
\end{lem}

Similar to the proof of \cite[Lemma 6.19, Lemma 6.20, Lemma 6.21]{K-Quantum-aff-va}, by using Lemma \ref{lem:special-tau-tech1}, we get the following result.

\begin{lem}\label{lem:S-ell-ell-prime}
For $i,j\in I$, we have that
\begin{align*}
  &S_{\ell,\ell'}(z_1)(A_i(z_2)\ot h_j)=A_i(z_2)\ot h_j
  +\Sing_{z_2}\(A_i(z_2)\ot \vac\ot\(e^{z_2\pd{z_1}}-1\){\wh{\ell,\ell'}}_{ij}^{1,+}(-z_1) \),\\
  &S_{\ell,\ell'}(z_1)(h_j\ot A_i(z_2))=h_j\ot A_i(z_2)
  +\Sing_{z_2}\(\vac\ot A_i(z_2)\ot\(1-e^{-z_2\pd{z_1}}\){\wh{\ell,\ell'}}_{ji}^{2,+}(-z_1)\),\\
  &S_{\ell,\ell'}(z_1)(A_i(z_2)\ot x_j^\pm)
  =\Sing_{z_2}\Big(A_i(z_2)\ot x_j^\pm
  \ot
  \exp\(\(1-e^{z_2\pd{z_1}}\) \log{\wh{\ell,\ell'}}_{ji}^{\pm,+}(-z_1) \)\Big),\\
  &S_{\ell,\ell'}(z_1)(x_j^\pm\ot A_i(z_2))
  =\Sing_{z_2}\Big(x_j^\pm\ot A_i(z_2)
  \ot \exp\(\(1-e^{-z_2\pd{z_1}}\) \log{\wh{\ell,\ell'}}_{ij}^{+,\pm}(-z_1) \)\Big).
\end{align*}
For $i,j,k\in I$ such that $a_{ij}\le 0$, we have that
\begin{align*}
  &S_{\ell,\ell'}(z)\(Q_{ij}^\pm(z_1,\dots,z_{m_{ij}})\ot h_k\)
  =Q_{ij}^\pm(z_1,\dots,z_{m_{ij}})\ot h_k\\
  &\quad\pm \Sing_{z_1,\dots,z_{m_{ij}}}\Bigg(Q_{ij}^\pm(z_1,\dots,z_{m_{ij}})\ot\vac
  \ot
    \({\wh{\ell,\ell'}}_{kj}^{1,\pm}(-z)+\sum_{a=1}^{m_{ij}}{\wh{\ell,\ell'}}_{ki}^{1,\pm}(-z-z_a)
  \)\Bigg),\\
  &S_{\ell,\ell'}(z)\(h_k\ot Q_{ij}^\pm(z_1,\dots,z_{m_{ij}})\)
  =h_k\ot Q_{ij}^\pm(z_1,\dots,z_{m_{ij}})\\
  &\quad\mp \Sing_{z_1,\dots,z_{m_{ij}}}\Bigg(\vac\ot Q_{ij}^\pm(z_1,\dots,z_{m_{ij}})
  \ot \({\wh{\ell,\ell'}}_{jk}^{2,\pm}(-z)
  +\sum_{a=1}^{m_{ij}}{\wh{\ell,\ell'}}_{ik}^{2,\pm}(-z+z_a)\)\Bigg),\\
  &S_{\ell,\ell'}(z)\(Q_{ij}^\pm(z_1,\dots,z_{m_{ij}})\ot x_k^\epsilon\)
  =\Sing_{z_1,\dots,z_{m_{ij}}}\\
  &\quad\(Q_{ij}^\pm(z_1,\dots,z_{m_{ij}})\ot x_k^\epsilon
  \ot {\wh{\ell,\ell'}}_{kj}^{\epsilon,\pm}(-z)\inv
  \prod_{a=1}^{m_{ij}}{\wh{\ell,\ell'}}_{ki}^{\epsilon,\pm}(-z-z_a)\inv\),\\
  &S_{\ell,\ell'}(z)\(x_k^\epsilon\ot Q_{ij}^\pm(z_1,\dots,z_{m_{ij}})\)
  =\Sing_{z_1,\dots,z_{m_{ij}}}\\
  &\quad\(x_k^\epsilon\ot Q_{ij}^\pm(z_1,\dots,z_{m_{ij}})\ot
  {\wh{\ell,\ell'}}_{jk}^{\epsilon,\pm}(-z)\inv
    \prod_{a=1}^{m_{ij}}{\wh{\ell,\ell'}}_{ik}^{\epsilon,\pm}(-z+z_a)\inv\).
\end{align*}
For any $\ell\in \Z_+$ and $i,j\in I$, we have that
\begin{align*}
  &S_{\ell,\ell'}(z)\(M_i^\pm(z_1,\dots,z_{r\ell/r_i})\ot h_j\)
  =M_i^\pm(z_1,\dots,z_{r\ell/r_i})\ot h_j\\
  &\quad\pm\Sing_{z_1,\dots,z_{r\ell/r_i}}z_1\inv\cdots z_{r\ell/r_i}\inv
  \(
  M_i^\pm(z_1,\dots,z_{r\ell/r_i})\ot \vac\ot
    \sum_{a=1}^{r\ell/r_i+1}{\wh{\ell,\ell'}}_{ji}^{1,\pm}(-z-z_a)
  \),\\
  &S_{\ell,\ell'}(z)\(h_j\ot M_i^\pm(z_1,\dots,z_{r\ell/r_i})\)
  =h_j\ot M_i^\pm(z_1,\dots,z_{r\ell/r_i})\\
  &\quad\mp\Sing_{z_1,\dots,z_{r\ell/r_i}}z_1\inv \cdots z_{r\ell/r_i}\inv
    \(\vac\ot M_i^\pm(z_1,\dots,z_{r\ell/r_i})\ot
    \sum_{a=1}^{r\ell/r_i+1}{\wh{\ell,\ell'}}_{ij}^{1,\pm}(-z+z_a)
  \),\\
  &S_{\ell,\ell'}(z)\(M_i^\pm(z_1,\dots,z_{r\ell/r_i})\ot x_j^\epsilon\)
  =\Sing_{z_1,\dots,z_{r\ell/r_i}}z_1\inv\cdots z_{r\ell/r_i}\inv\\
  &\quad\(
    M_i^\pm(z_1,\dots,z_{r\ell/r_i})\ot x_j^\epsilon\ot
    \prod_{a=1}^{r\ell/r_i+1}{\wh{\ell,\ell'}}_{ji}^{\epsilon,\pm}(-z-z_a)\inv
  \),\\
  &S_{\ell,\ell'}(z)\(x_j^\epsilon\ot M_i^\pm(z_1,\dots,z_{r\ell/r_i})\)
  =\Sing_{z_1,\dots,z_{r\ell/r_i}}z_1\inv\cdots z_{r\ell/r_i}\inv\\
  &\quad\(
    x_j^\epsilon\ot M_i^\pm(z_1,\dots,z_{r\ell/r_i})\ot
    \prod_{a=1}^{r\ell/r_i+1}{\wh{\ell,\ell'}}_{ij}^{\epsilon,\pm}(-z+z_a)\inv
  \),
\end{align*}
where $z_{r\ell/r_i+1}=0$.
\end{lem}

From Proposition \ref{prop:ideal-def-alt} and Lemma \ref{lem:S-ell-ell-prime}, we have the following result.
\begin{prop}
$S_{\ell,\ell'}(z)$ induces a $\C[[\hbar]]$-module map on $$V_{\hat\g,\hbar}^{\ell}\wh\ot V_{\hat\g,\hbar}^{\ell'}\wh\ot\C((z))[[\hbar]].$$
If $\ell,\ell'\in\Z_+$, it further induces a $\C[[\hbar]]$-module map on $$L_{\hat\g,\hbar}^{\ell}\wh\ot L_{\hat\g,\hbar}^{\ell'}\wh\ot\C((z))[[\hbar]].$$
\end{prop}

Combining this with Lemma \ref{lem:n-qva-F}, we get the main result of this section.

\begin{thm}
Let $\ell_1,\ell_2,\dots,\ell_n\in\C$. Then $$((X_{\hat\g,\hbar}^{\ell_i})_{i=1}^n,(S_{\ell_i,\ell_j}(z))_{i,j=1}^n)$$
is an $\hbar$-adic $n$-quantum VA
for $X=F,V,L$.
\end{thm}

\begin{rem}
{\em
Let $\ell_1,\dots,\ell_n\in\C$. In this paper, we simply write the twisted tensor product
\begin{align*}
  X_{\hat\g,\hbar}^{\ell_1}\wh\ot_{S_{\ell_1,\ell_2}}\cdots \wh\ot_{S_{\ell_{n-1},\ell_n}}X_{\hat\g,\hbar}^{\ell_n}
\end{align*}
as
\begin{align*}
  X_{\hat\g,\hbar}^{\ell_1}\wh\ot\cdots \wh\ot X_{\hat\g,\hbar}^{\ell_n}
\end{align*}
for $X=F,V,L$.
}
\end{rem}

\section{Corpoduct of quantum affine vertex algebras}\label{sec:coprod}

The purpose of this section is to prove the following result.

\begin{thm}\label{thm:coproduct}
Let $X=F,V,L$. For any $\ell,\ell'\in\C$, there is an $\hbar$-adic quantum VA homomorphism
$\Delta:X_{\hat\g,\hbar}^{\ell+\ell'}\to X_{\hat\g,\hbar}^{\ell}\wh\ot X_{\hat\g,\hbar}^{\ell'}$
defined by
\begin{align}
  &\Delta(h_i)=q^{-r\ell'\partial}h_i\ot\vac+\vac\ot q^{r\ell\partial}h_i,\label{eq:def-Delta-h}\\
  &\Delta(x_i^+)=x_i^+\ot\vac+q^{2r\ell\partial}E_\ell(h_i)\ot q^{2r\ell\partial}x_i^+,\label{eq:def-Delta-x+}\\
  &\Delta(x_i^-)=x_i^-\ot\vac+\vac\ot x_i^-.\label{eq:def-Delta-x-}
\end{align}
Moreover, let $n\in\Z_+$ and $\ell_1,\dots,\ell_n\in\C$. For $1\le i<n$, we have
\begin{equation*}
\begin{tikzcd}[column sep=1em, row sep=1.5em]
    \Big(X_{\hat\g,\hbar}^{\ell_1}\ar[d,equal]&\dots&X_{\hat\g,\hbar}^{\ell_{i-1}}\ar[d,equal]
    &X_{\hat\g,\hbar}^{\ell_i+\ell_{i+1}}\ar[d,"\Delta"]
    &X_{\hat\g,\hbar}^{\ell_{i+2}}\ar[d,equal]
    &\dots&X_{\hat\g,\hbar}^{\ell_n}\Big)\ar[d,equal]\\
    \Big(X_{\hat\g,\hbar}^{\ell_1}&\dots&X_{\hat\g,\hbar}^{\ell_{i-1}}&
    X_{\hat\g,\hbar}^{\ell_i}\wh\ot X_{\hat\g,\hbar}^{\ell_{i+1}}&
    X_{\hat\g,\hbar}^{\ell_{i+2}}&\dots&
    X_{\hat\g,\hbar}^{\ell_n}\Big)
\end{tikzcd}
\end{equation*}
is an $\hbar$-adic $(n-1)$-quantum VA homomorphism.
Furthermore, for $\ell,\ell',\ell''\in\C$, we have that
\begin{align*}
  (\Delta\ot 1)\circ\Delta=(1\ot\Delta)\circ\Delta:X_{\hat\g,\hbar}^{\ell+\ell'+\ell''}\to X_{\hat\g,\hbar}^{\ell}\wh\ot X_{\hat\g,\hbar}^{\ell'}\wh\ot X_{\hat\g,\hbar}^{\ell''}
\end{align*}
\end{thm}

We prove this theorem by proving the following results:

\begin{prop}\label{prop:S-Delta}
Let $\ell,\ell',\ell''\in\C$. Set
\begin{align*}
  S_{\{\ell,\ell'\},\ell''}(z)=S_{\ell',\ell''}^{23}(z)S_{\ell,\ell''}^{13}(z),\quad
  S_{\ell,\{\ell',\ell''\}}(z)=S_{\ell,\ell'}^{12}(z)S_{\ell,\ell''}^{13}(z).
\end{align*}
Then
\begin{align*}
  &S_{\{\ell,\ell'\},\ell''}(z)\circ(\Delta\ot 1)=(\Delta\ot 1)\circ S_{\ell+\ell',\ell''}(z)
  \quad \te{on }F_{\hat\g,\hbar}^{\ell+\ell'}\wh\ot F_{\hat\g,\hbar}^{\ell''},\\
  &S_{\ell,\{\ell',\ell''\}}(z)\circ(1\ot \Delta)=(1\ot \Delta)\circ S_{\ell,\ell'+\ell''}(z)
  \quad\te{on }F_{\hat\g,\hbar}^{\ell}\wh\ot F_{\hat\g,\hbar}^{\ell'+\ell''}.
\end{align*}
\end{prop}

\begin{prop}\label{prop:Y-Delta-cartan}
Denote by $Y_\Delta$ the vertex operator of $X_{\hat\g,\hbar}^{\ell}\wh\ot X_{\hat\g,\hbar}^{\ell'}$ for $X=F,V,L$. Then we have that ($i,j\in I$)
\begin{align}
  &[Y_\Delta(\Delta(h_i),z_1),Y_\Delta(\Delta(h_j),z_2)]
  =[r_ia_{ij}]_{q^{\pd{z_2}}}[r(\ell+\ell')]_{q^{\pd{z_2}}}\label{eq:prop-Y-Delta-1}\\
  &\quad\times\(\iota_{z_1,z_2}q^{-r(\ell+\ell')\pd{z_2}}
  -\iota_{z_2,z_1}q^{r(\ell+\ell')\pd{z_2}}\)
  \pd{z_1}\pd{z_2}\log f(z_1-z_2),\nonumber\\
  &[Y_\Delta(\Delta(h_i),z_1),Y_\Delta(\Delta(x_j^\pm),z_2)]
  =\pm Y_\Delta(\Delta(x_j^\pm),z_2)[r_ia_{ij}]_{q^{\pd{z_2}}}\label{eq:prop-Y-Delta-2}\\
  &\quad\(\iota_{z_1,z_2}q^{-r(\ell+\ell')\pd{z_2}}
  -\iota_{z_2,z_1}q^{r(\ell+\ell')\pd{z_2}}\)
  \pd{z_1}\log f(z_1-z_2),\nonumber\\
  &f(z_1-z_2)^{\delta_{ij}+\delta_{ij}q^{2r\ell}}
  Y_\Delta(\Delta(x_i^+),z_1)Y_\Delta(\Delta(x_j^-),z_2)\label{eq:prop-Y-Delta-3-pre}\\
  &\quad=\iota_{z_2,z_1}f(z_1-z_2)^{\delta_{ij}+\delta_{ij}q^{2r\ell}+q^{-r_ia_{ij}}-q^{r_ia_{ij}}}
   Y_\Delta(\Delta(x_j^-),z_2)
    Y_\Delta(\Delta(x_i^+),z_1)\nonumber
\end{align}
on $F_{\hat\g,\hbar}^{\ell}\wh\ot F_{\hat\g,\hbar}^{\ell'}$,
and
\begin{align}
  &Y_\Delta(\Delta(x_i^+),z_1)Y_\Delta(\Delta(x_j^-),z_2)\nonumber\\
  &\quad-\iota_{z_2,z_1}f(z_1-z_2)^{q^{-r_ia_{ij}}-q^{r_ia_{ij}}}
  Y_\Delta(\Delta(x_j^-),z_2)
    Y_\Delta(\Delta(x_i^+),z_1)\nonumber\\
  =&\frac{\delta_{ij}}{q^{r_i}-q^{-r_i}}\bigg(z_1\inv\delta\(\frac{z_2}{z_1}\)
  -Y_\Delta\(E_{\ell+\ell'}(\Delta(h_i)),z_2\)
  z_1\inv\delta\(\frac{z_2-2r(\ell+\ell')\hbar}{z_1}\)\bigg)\label{eq:prop-Y-Delta-3}
\end{align}
on $V_{\hat\g,\hbar}^{\ell}\wh\ot V_{\hat\g,\hbar}^{\ell'}$
\end{prop}

\begin{prop}\label{prop:Y-Delta-serre}
For $i,j\in I$, we have that
\begin{align}
  &\iota_{z_1,z_2}f(z_1-z_2)^{-\delta_{ij}+q^{-r_ia_{ij}}}
  Y_\Delta(\Delta(x_i^\pm),z_1)
  Y_\Delta(\Delta(x_j^\pm),z_2)\label{eq:prop-Y-Delta-locality}\\
  =&\iota_{z_2,z_1}f(z_1-z_2)^{-\delta_{ij}+q^{r_ia_{ij}}}
  Y_\Delta(\Delta(x_j^\pm),z_2)Y_\Delta(\Delta(x_i^\pm),z_1)\nonumber
\end{align}
on $F_{\wh\g,\hbar}^\ell\wh\ot F_{\wh\g,\hbar}^{\ell'}$.
In addition, if $a_{ij}\le 0$, then we have that
\begin{align}\label{eq:prop-Y-Delta-serre}
  \(\Delta(x_i^\pm)\)_0^{1-a_{ij}}\Delta(x_j^\pm)=0\quad\te{on }V_{\hat\g,\hbar}^\ell\wh\ot V_{\hat\g,\hbar}^{\ell'}.
\end{align}
%\end{prop}
%
%\begin{prop}\label{prop:Y-Delta-int}
Moreover, if $\ell\in\Z_+$, then we have that
\begin{align}\label{eq:prop-Y-Delta-int}
  \(\Delta(x_i^\pm)\)_0^{r(\ell+\ell')/r_i}\Delta(x_i^\pm)=0\,\,\te{on }L_{\hat\g,\hbar}^\ell\wh\ot L_{\hat\g,\hbar}^{\ell'}\quad \te{for }i\in I.
\end{align}
\end{prop}

\begin{prop}\label{prop:Delta-coasso}
For $\ell,\ell',\ell''\in\C$, we have that
\begin{align*}
  (\Delta\ot 1)\circ \Delta(u)=(1\ot \Delta)\circ\Delta(u)\in F_{\hat\g,\hbar}^\ell
  \wh\ot F_{\hat\g,\hbar}^{\ell'} \wh\ot
  F_{\hat\g,\hbar}^{\ell''}
\end{align*}
for $u\in F_{\hat\g,\hbar}^{\ell+\ell'+\ell''}$.
\end{prop}

Assume now that the above four Propositions hold.

\vspace{2mm}

\noindent\emph{Proof of Theorem \ref{thm:coproduct}:}
From \eqref{eq:prop-Y-Delta-1}, \eqref{eq:prop-Y-Delta-2}, \eqref{eq:prop-Y-Delta-3-pre}
and \eqref{eq:prop-Y-Delta-locality}, we get that
\begin{align*}
  \(F_{\hat\g,\hbar}^\ell\wh\ot F_{\hat\g,\hbar}^{\ell'},
    \left\{Y_\Delta\(\Delta(h_i),z\)\right\}_{i\in I},
    \left\{Y_\Delta\(\Delta\(x_i^\pm\),z\)\right\}_{i\in I}\)
  \in\obj \mathcal M_{\wh{\ell+\ell'}}(\g).
\end{align*}
So Proposition \ref{prop:universal-M-tau} provides an $\hbar$-adic nonlocal VA homomorphism
$\Delta$ from $F_{\hat\g,\hbar}^{\ell+\ell'}$ to $F_{\hat\g,\hbar}^\ell\ot F_{\hat\g,\hbar}^{\ell'}$.
From Propositions \ref{prop:S-Delta} and \ref{prop:Delta-coasso}, we complete the proof for $X=F$.
For the case $X=V$ or $L$, the theorem follows immediate from \eqref{eq:prop-Y-Delta-3} and Proposition \ref{prop:Y-Delta-serre}.

\subsection{Some formulas}\label{subsec:some-formulas}

In this subsection, we collect some formulas that will be used later on.
Denote by $Y(\cdot,z)$ the vertex operator map of the VA $F_{\hat\g}^\ell$ (see Definition \ref{de:affVAs}).
From Remark \ref{rem:F-varepsilon=F}, we have that
\begin{align*}
  \(F_{\hat\g}^\ell[[\hbar]],\{Y(h_i,z)\}_{i\in I},\{Y(x_i^\pm,z)\}_{i\in I}\)\in\obj\mathcal M_\varepsilon(\g).
\end{align*}
In view of Proposition \ref{prop:classical-limit}, we have that
\begin{align*}
  F_{\wh\g,\hbar}^\ell=\mathfrak D_{\wh\ell}^\rho(F_\varepsilon(\g,\ell))=\mathfrak D_{\wh\ell}^\rho(F_{\hat\g}^\ell).
\end{align*}
It follows that ($i\in I$):
\begin{align*}
  Y_{\wh\ell}(h_i,z)=Y(h_i,z)+\wh\ell(\al_i^\vee,z),\quad Y_{\wh\ell}(x_i^\pm,z)=Y(x_i^\pm,z)\wh\ell(e_i^\pm,z).
\end{align*}

\begin{lem}\label{lem:com-formulas}
For $i\in I$, we define
\begin{align*}
  &h_i^-(z)=Y(h_i,z)^-
  +\wh\ell(\al^\vee_i,z),\quad
   h_i^+(z)=Y(h_i,z)^+.
\end{align*}
Then we have that
\begin{align}
  &[h_i^-(z_1),h_j^+(z_2)]\label{eq:com-formulas-1}
  %[a_{ij}]_{q^{r_i\pd{z_2}}}[r\ell/r_j]_{q^{r_j\pd{z_2}}}q^{-r\ell\pd{z_2}}
    =\pd{z_1}\pd{z_2}\log f(z_1-z_2)^{ [r_ia_{ij}]_{q}[r\ell]_{q}q^{r\ell} },\\
  &[h_i^-(z_1),Y_{\wh\ell}(x_j^\pm,z_2)]=\pm Y_{\wh\ell}(x_j^\pm,z_2)%[a_{ij}]_{q^{r_i\pd{z_2}}}q^{-r\ell\pd{z_2}}
  \pd{z_1}\log f(z_1-z_2)^{[r_ia_{ij}]_{q}q^{r\ell}},\label{eq:com-formulas-2}\\
  &[h_i^+(z_1),Y_{\wh\ell}(x_j^\pm,z_2)]=\mp Y_{\wh\ell}(x_j^\pm,z_2)%[a_{ij}]_{q^{r_i\pd{z_2}}}q^{r\ell\pd{z_2}}
  \pd{z_1}\log f(-z_2+z_1)^{[r_ia_{ij}]_{q}q^{-r\ell}}.\label{eq:com-formulas-3}
\end{align}
\end{lem}

\begin{proof}
From \eqref{eq:fqva-rel1} and \eqref{eq:fqva-rel2}, we have that
\begin{align*}
  &[Y(h_i,z_1),Y(h_j,z_2)]={r_j}{a_{ij}r\ell} \pd{z_2}z_1\inv\delta\(\frac{z_2}{z_1}\),\\
  &[Y(h_i,z_1),Y(x_j^\pm,z_2)]=\pm Y(x_j^\pm,z_2)r_ia_{ij} z\inv\delta\(\frac{z_2}{z_1}\).
\end{align*}
Then we have that
\begin{align}
  &[Y(h_i,z_1)^-,Y(h_j,z_2)^+]={r_j}{a_{ij}r\ell} (z_1-z_2)^{-2},\label{eq:com-formulas-temp1}\\
  &[Y(h_i,z_1)^-,Y(x_j^\pm,z_2)]=\pm Y(x_j^\pm,z_2)r_ia_{ij}(z_1-z_2)\inv,\label{eq:com-formulas-temp2}\\
  &[Y(h_i,z_1)^+,Y(x_j^\pm,z_2)]=\mp Y(x_j^\pm,z_2)r_ia_{ij}(z_2-z_1)\inv.\label{eq:com-formulas-temp3}
\end{align}
From the definition of $\wh\ell$ (see Proposition \ref{prop:deform-datum}), we have that
\begin{align*}
  [\wh\ell(\al^\vee_i,z_1),Y(h_j,z_2)]=&Y(\wh\ell_i(z_1-z_2)h_j,z_2)=\wh\ell_{ij}(z_1-z_2),\\
  [\wh\ell(e_i^\pm,z_1),Y(h_j,z_2)]=&Y(\wh\ell_i^\pm(z_1-z_2)h_j,z_2)\wh\ell_i^\pm(z_1)-Y(h_j,z_2)\wh\ell_i^\pm(z_1)\\
  =&\mp \wh\ell(e_i^\pm,z_1)\wh\ell_{ij}^{2,+}(z_1-z_2).
%  \\=&\mp\wh\ell(e_i^\pm,z_1)[r_ia_{ij}]_{q^{\pd{z_2}}}q^{-r\ell\pd{z_2}}\frac{1+e^{-z_1+z_2}}{2-2e^{-z_1+z_2}}
%  \pm \wh\ell(e_i^\pm,z_1)r_ia_{ij}(z_1-z_2)\inv.
\end{align*}
It follows that
\begin{align*}
  &[\wh\ell(\al^\vee_i,z_1),Y(h_j,z_2)^-]=0,&&
  [\wh\ell(\al^\vee_i,z_1),Y(h_j,z_2)^+]=\wh\ell_{ij}(z_1-z_2),\\
  &[\wh\ell(e_i^\pm,z_1),Y(h_j,z_2)^-]=0,
  &&[\wh\ell(e_i^\pm,z_1),Y(h_j,z_2)^+]=\mp\wh\ell(e_i^\pm,z_1)\wh\ell_{ij}^{2,+}(z_1-z_2).
\end{align*}
Combining these with equations \eqref{eq:com-formulas-temp1}, \eqref{eq:com-formulas-temp2} and \eqref{eq:com-formulas-temp3},
we complete the proof of lemma.
%we get that
%\begin{align*}
%  &[h_i^-(z_1),h_j^+(z_2)]\\
%  =&[Y(h_i,z_1)^-+\wh\ell(\al^\vee_i,z_1),Y(h_j,z_2)^+]
%  \\
%  =&\wh\ell_{ij}(z_1-z_2)+\frac{a_{ij}r\ell}{r_j}(z_1-z_2)^{-2}\\
%  =&
%  %[a_{ij}]_{q^{r_i\pd{z_2}}}[r\ell/r_j]_{q^{r_j\pd{z_2}}}q^{-r\ell\pd{z_2}}
%  \pd{z_1}\pd{z_2}\log f(z_1-z_2)^{[a_{ij}]_{q^{r_i}}[r\ell/r_j]_{q^{r_j}}q^{r\ell}},\\
%  &[h_i^-(z_1),Y_{\wh\ell}(x_j^\pm,z_2)]\\
%  =&[Y(h_i,z_1)^-+\wh\ell(\al^\vee_i,z_1),Y(x_j^\pm,z_2)\wh\ell(e_j^\pm,z_2)]\\
%  =&[Y(h_i,z_1)^-,Y(x_j^\pm,z_2)]\wh\ell(e_j^\pm,z_2)\\
%  &+[\wh\ell(\al_i,z_1),Y(x_j^\pm,z_2)]\wh\ell(e_j^\pm,z_2)\\
%  &-Y(x_j^\pm,z_2)[\wh\ell(e_j^\pm,z_2),Y(h_i,z_1)^-]\\
%  =&\pm Y(x_j^\pm,z_2)\wh\ell(e_j^\pm,z_2)r_ia_{ij}(z_1-z_2)\inv\\
%  & +Y(\wh\ell_i(z_1-z_2)x_j^\pm,z_2)\wh\ell(e_j^\pm,z_2)\\
%  =&\pm Y(x_j^\pm,z_2)\wh\ell(e_j^\pm,z_2)\(r_ia_{ij}(z_1-z_2)\inv
%   +\wh\ell_{ij}^{1,+}(z_1-z_2)\)\\
%  =&\pm Y_{\wh\ell}(x_j^\pm,z_2)\pd{z_1}\log f(z_1-z_2)^{[a_{ij}]_{q^{r_i}}q^{r\ell}},\\
%  &[h_i^+(z_1),Y_{\wh\ell}(x_j^\pm,z_2)]\\
%  =&[Y(h_i,z_1)^+,Y(x_j^\pm,z_2)\wh\ell(e_j^\pm,z_2)]\\
%  =&[Y(h_i,z_1)^+,Y(x_j^\pm,z_2)]\wh\ell(e_j^\pm,z_2)\\
%  &-Y(x_j^\pm,z_2)[\wh\ell(e_j^\pm,z_2),Y(h_i,z_2)^+]\\
%  =&\mp Y(x_j^\pm,z_2)\wh\ell(e_j^\pm,z_2)a_{ij}(z_1-z_2)\inv\\
%  &\pm Y(x_j^\pm,z_2)\wh\ell(e_j^\pm,z_2)\wh\ell_{ji}^{2,+}(z_2-z_1)\\
%  =&\mp Y_{\wh\ell}(x_j^\pm,z_2)\pd{z_1}\log f(-z_2+z_1)^{[a_{ij}]_{q^{r_i}}q^{-r\ell}}.
%\end{align*}
%Therefore, we complete the proof.
\end{proof}

\begin{lem}\label{lem:com-formulas2}
For $i\in I$, we define
\begin{align}
  %&\wt h_i^\pm(w_1,w_2,z)=G_{w_1,w_2}\(\pd{z}\)[r\ell]_{q^{\pd{z}}}\inv h_i^\pm(z),\\
  &\wt h_i^\pm(z)=-q^{-r\ell\pd{z}}2\hbar f_0\(2\hbar\pd{z}\)h_i^\pm(z).
  %&\wt h_i^{\pm,n}(z)=-q^{-nr\ell\pd{z}}[n]_{q^{nr\ell\pd{z}}}F\(\pd{z}\)h_i^\pm(z).
\end{align}
Then we have that
\begin{align}
  &[\wt h_i^-(z_1),h_j^+(z_2)]\label{eq:com-formulas-4}
  =\pd{z_1}\log f(z_1-z_2)^{[r_ja_{ji}]_{q}[r\ell]_q \(q-q\inv\)},\\
  &[h_i^-(z_1),\wt h_j^+(z_2)]\label{eq:com-formulas-5}
  =\pd{z_1}\log f(z_1-z_2)^{[r_ia_{ij}]_{q} \(q^{r\ell}-q^{-r\ell}\)q^{ 2r\ell}},\\
  &[\wt h_i^-(z_1),\wt h_j^+(z_2)]\label{eq:com-formulas-6}
  =\log f(z_1-z_2)^{( q^{2r\ell}-1 )( q^{-r_ia_{ij}}-q^{r_ia_{ij}} )},\\
  &[\wt h_i^-(z_1),Y_{\wh\ell}(x_j^\pm,z_2)]\label{eq:com-formulas-7}
  =\pm Y_{\wh\ell}(x_j^\pm,z_2)%\( q^{r_ia_{ij}\pd{z_2}}-q^{-r_ia_{ij}\pd{z_2}} \)
  \log f(z_1-z_2)^{q^{-r_ia_{ij}}-q^{r_ia_{ij}}},\\
  &[\wt h_i^+(z_1),Y_{\wh\ell}(x_j^\pm,z_2)]\label{eq:com-formulas-8}
  =\pm Y_{\wh\ell}(x_j^\pm,z_2)\log f(z_2-z_1)^{\( q^{-r_ia_{ij}}-q^{r_ia_{ij}} \)q^{2r\ell}}.
\end{align}
\end{lem}

\begin{proof}
From \eqref{eq:com-formulas-1}, we have that
\begin{align*}
  &[\wt h_i^-(z_1),h_j^+(z_2)]\\
  =&-q^{-r\ell\pd{z_1}}2\hbar f_0\(2\hbar\pd{z_1}\)
    \pd{z_1}\pd{z_2}\log f(z_1-z_2)^{[r_ja_{ji}]_{q}[r\ell]_{q}q^{r\ell}}\\
  =&-2\hbar f_0\(2\hbar\pd{z_1}\)
    \pd{z_1}\pd{z_2}\log f(z_1-z_2)^{[r_ja_{ji}]_{q}[r\ell]_{q}}\\
  =&-\(q^{\pd{z_1}}-q^{-\pd{z_1}}\)
  \pd{z_2}\log f(z_1-z_2)^{[r_ja_{ji}]_{q}[r\ell]_{q}}\\
  =&\pd{z_1}\log f(z_1-z_2)^{[r_ja_{ji}]_{q}[r\ell]_{q}(q-q\inv)}.
\end{align*}
This proves \eqref{eq:com-formulas-4}. The proofs of the rest equations are similar.
\end{proof}

The following result is an immediate consequence of Lemma \ref{lem:com-formulas2}.

\begin{coro}
For $i,j\in I$, we have that
\begin{align}
  &\left[\exp\(\wt h_i^-(z_1)\),h_j^+(z_2)\right]
  =\exp\(\wt h_i^-(z_1)\) \label{eq:com-formulas-9}
  \pd{z_1}\log f(z_1-z_2)^{[r_ja_{ji}]_{q}[r\ell]_q(q-q\inv)},\\
  &\left[h_i^-(z_1),\exp\(\wt h_j^+(z_2)\)\right]
  =\exp\(\wt h_j^+(z_2)\) \label{eq:com-formulas-10}
  \pd{z_1}\log f(z_1-z_2)^{[r_ia_{ij}]_{q}[r\ell]_qq^{2r\ell} },\\
  &\exp\(\wt h_i^-(z_1)\)\exp\(\wt h_j^+(z_2)\)
  =\exp\(\wt h_j^+(z_2)\)\exp\(\wt h_i^-(z_1)\)
  \label{eq:com-formulas-11}\\
  &\qquad\times f(z_1-z_2)^{(q^{r_ia_{ij}}-q^{-r_ia_{ij}})(1-q^{2r\ell})},\nonumber
    \\
  &\exp\(\wt h_i^-(z_1)\)Y_{\wh\ell}(x_j^\pm,z_2)\label{eq:com-formulas-12}
  =Y_{\wh\ell}(x_j^\pm,z_2)\exp\(\wt h_i^-(z_1)\)
  f(z_1-z_2)^{q^{\mp r_ia_{ij}}-q^{\pm r_ia_{ij}}},
  \\
  &\exp\(\wt h_i^+(z_1)\)Y_{\wh\ell}(x_j^\pm,z_2)\label{eq:com-formulas-13}
  =Y_{\wh\ell}(x_j^\pm,z_2)\exp\(\wt h_i^+(z_1)\)
  f(z_1-z_2)^{q^{2r\ell}(q^{\mp r_ia_{ij}}-q^{\pm r_ia_{ij}})}.
  %&\exp\(\wt h_i^-(z_1)\)Y_{\wh\ell}\(\(x_j^\pm\)_{-1}^k\vac,z_2\)
%  =Y_{\wh\ell}\(\(x_j^\pm\)_{-1}^k\vac,z_2\)\label{eq:com-formulas-14}\\
%  &\times\exp\(\wt h_i^-(z_1)\)
%    \prod_{a=1}^kg_{ij,\hbar}(z_1-z_2-2(k-a)r_j\hbar)^{\pm 1},\nonumber\\
\end{align}
\end{coro}

\begin{lem}\label{lem:exp-cal}
Let $W$ be a topologically free $\C[[\hbar]]$-module and let
\begin{align*}
  &\al(z)\in \Hom(W,W\wh\ot\C[z,z\inv][[\hbar]]),\quad \beta(z)\in\Hom(W,W[[z]]),\quad\xi(z)\in \E_\hbar(W)
\end{align*}
satisfying the condition that
\begin{align*}
  &[\al(z_1),\al(z_2)]=0=[\beta(z_1),\beta(z_2)],\quad
  [\al(z_1),\beta(z_2)]=\iota_{z_1,z_2}\gamma(z_2-z_1),\\
  &[\al(z_1),\xi(z_2)]=\xi(z_2) \iota_{z_1,z_2}\gamma_1(z_1-z_2),
\end{align*}
for some $\gamma(z)\in\C(z)[[\hbar]]$ and $\gamma_1(z)\in \C((z))[[\hbar]]$.
Assume that $$\al(z),\beta(z)\in\hbar(\End W)[[z,z\inv]].$$ Then
\begin{align*}
  &\exp\((\al(z)+\beta(z))_{-1}\)\xi(z)\\
  =&\exp\(\beta(z)\)\xi(z)\exp\(\al(z)\)
   \exp\(\half \Res_zz\inv\gamma(-z)+z\inv\gamma_1(z)\).
\end{align*}
\end{lem}

\begin{proof}
Set $U=\{\al(z),\beta(z),\xi(z)\}$. Then $U$ generates an $\hbar$-adic nonlocal VA $\<U\>$.
From Baker-Campbell-Hausdorff formula we obtain that
\begin{align*}
  &\exp\((\al(z)+\beta(z))_{-1}\)
  =\exp\(\beta(z)_{-1}\)\exp\(\al(z)_{-1}\)\exp\(\half \Res_zz\inv\gamma(-z)\).
\end{align*}
Note that
\begin{align*}
  \al(z_1)\xi(z_2)\al(z_2)^m=\xi(z_2)\al(z_1)\al(z_2)^m+\xi(z_2)\al(z_2)^m\iota_{z_1,z_2}\gamma_1(z_1-z_2).
\end{align*}
Then
\begin{align*}
  \al(z)_{-1}\xi(z)\al(z)^m=\xi(z)\al(z)^{m+1}+\xi(z)\al(z)^m\Res_zz\inv\gamma_1(z).
\end{align*}
So
\begin{align*}
  \exp\(\al(z)_{-1}\)\xi(z)=\xi(z)\exp\(\al(z)\)\exp\(\Res_zz\inv\gamma_1(z)\).
\end{align*}
Therefore,
\begin{align*}
  &\exp\((\al(z)+\beta(z))_{-1}\)\xi(z)\\
  =&\exp\(\beta(z)\)\xi(z)\exp\(\al(z)\)
  \exp\(\half \Res_zz\inv\gamma(-z)+z\inv\gamma_1(z)\).
\end{align*}
We complete the proof.
\end{proof}

\begin{prop}\label{prop:Y-E}
For each $i\in I$, we have that
\begin{align*}
  Y_{\wh\ell}(E_\ell(h_i),z)=\exp(\wt h_i^+(z))\exp(\wt h_i^-(z)).
\end{align*}
\end{prop}

\begin{proof}
From \eqref{eq:com-formulas-6}, we have that
\begin{align*}
  [\wt h_i^-(z_1),\wt h_i^+(z_2)]
  =\gamma(z_2-z_1),\quad
  \te{where }
  \gamma(z)=\log f(-z)^{ (q^{-2r_i}-q^{2r_i})(q^{2r\ell}-1) }.
\end{align*}
Notice that
\begin{align*}
  &\Res_zz\inv\gamma(-z)
  =\Res_zz\inv\log f(z)^{ (q^{2r_i}-q^{-2r_i})(q^{-2r\ell}-1) }\\
  =&\Res_zz\inv\log f_0(z)^{ (q^{2r_i}-q^{-2r_i})(q^{-2r\ell}-1) }
  =\log\frac{f_0(2r_i\hbar-2r\ell\hbar)}{f_0(2r_i\hbar+2r\ell\hbar)}.
\end{align*}
Lemma \ref{lem:exp-cal} provides that
\begin{align*}
  &Y_{\wh\ell}(E_\ell(h_i),z)
  =\(\frac{f_0(2r_i\hbar+2r\ell\hbar)}{f_0(2r_i\hbar-2r\ell\hbar)}\)^\half Y_{\wh\ell}\(\exp\(\(-q^{-r\ell\partial}2\hbar f_0(2\partial\hbar)  h_i\)_{-1}\)\vac,z\)\\
  =&\(\frac{f_0(2r_i\hbar+2r\ell\hbar)}{f_0(2r_i\hbar-2r\ell\hbar)}\)^\half  \exp(\wt h_i^+(z))\exp(\wt h_i^-(z))
  \(\frac{f_0(2r_i\hbar-2r\ell\hbar)}{f_0(2r_i\hbar+2r\ell\hbar)}\)^\half\\
  =&\exp(\wt h_i^+(z))\exp(\wt h_i^-(z)),
\end{align*}
as desired.
\end{proof}

\begin{lem}
For $i,j\in I$, we have that
\begin{align}
  &Y_{\wh\ell}(h_i,z)^-E_\ell(h_j)\label{eq:Y-action-1}
  =E_\ell(h_j)\ot[r_ia_{ij}]_{q^{\pd{z}}}
  \(q^{r\ell\pd{z}}-q^{-r\ell\pd{z}}\)q^{2r\ell\pd{z}}  z\inv,\\
  &Y_{\wh\ell}(E_\ell(h_i),z)E_\ell(h_j)\label{eq:Y-action-2}
  =
  \exp\(\wt h_i^+(z)\)E_\ell(h_j)\ot f(z)^{ (q^{r_ia_{ij}}-q^{-r_ia_{ij}})(1-q^{2r\ell}) }
  %\frac{ f(z+(2r\ell-r_ia_{ij})\hbar)f(z+r_ia_{ij}\hbar) }
  %  { f(z+(2r\ell+r_ia_{ij})\hbar)f(z-r_ia_{ij}\hbar) }
    ,\\
  &Y_{\wh\ell}(E_\ell(h_i),z)x_j^\pm\label{eq:Y-action-3}
  =\exp\(\wt h_i^+(z)\)x_j^\pm\ot f(z)^{\mp (q^{r_ia_{ij}}-q^{-r_ia_{ij}}) }
  %\frac{f(z\mp r_ia_{ij}\hbar)}{f(z\pm r_ia_{ij}\hbar)}
  ,\\
  &Y_{\wh\ell}(x_i^\pm,z)E_\ell(h_j)\label{eq:Y-action-4}
  =\exp\(\wt h_j(0)\)e^{z\partial} x_i^\pm\ot
  f(z)^{\pm (q^{r_ia_{ij}}-q^{-r_ia_{ij}})q^{2r\ell} }
  %\frac{f(z+(2r\ell\pm r_ia_{ij})\hbar)}{f(z+(2r\ell\mp r_ia_{ij})\hbar)}
  .
\end{align}
\end{lem}

\begin{proof}
From \eqref{eq:com-formulas-9}, we have that
\begin{align*}
  &h_i^-(z_1)Y_{\wh\ell}(E_\ell(h_j),z_2)\vac\\
  =&[h_i^-(z_1),Y_{\wh\ell}(E_\ell(h_j),z_2)]\vac
  =\left[h_i^-(z_1),\exp\(\wt h_i^+(z_2)\)\exp\(\wt h_i^-(z_2)\)\right]\vac\\
  =&\exp\(\wt h_i^+(z_2)\)\exp\(\wt h_i^-(z_2)\)\vac [r_ia_{ij}]_{q^{\pd{z_2}}}
  \(q^{-r\ell\pd{z_2}}-q^{r\ell\pd{z_2}}\)q^{- 2r\ell\pd{z_2}}  \pd{z_1}\log f(z_1-z_2)\\
  =&\exp\(\wt h_i^+(z_2)\)\vac [r_ia_{ij}]_{q^{\pd{z_1}}}
  \(q^{r\ell\pd{z_1}}-q^{-r\ell\pd{z_1}}\)q^{2r\ell\pd{z_1}}  \pd{z_1}\log f(z_1-z_2).
\end{align*}
Taking $z_2\to 0$, we have that
\begin{align*}
  &h_i^-(z)E_\ell(h_j)
  =E_\ell(h_j)[r_ia_{ij}]_{q^{\pd{z}}}
  \(q^{r\ell\pd{z}}-q^{-r\ell\pd{z}}\)q^{2r\ell\pd{z}}  \pd{z}\log f(z).
\end{align*}
Notice that
\begin{align*}
  Y_{\wh\ell}(h_i,z)^-=h_i^-(z)^-,\quad \Sing_z\pd{z}\log f(z)=z\inv.
\end{align*}
Then we have that
\begin{align*}
  Y_{\wh\ell}(h_i,z)^-E_\ell(h_j)=E_\ell(h_j)[r_ia_{ij}]_{q^{\pd{z}}}
  \(q^{r\ell\pd{z}}-q^{-r\ell\pd{z}}\)q^{2r\ell\pd{z}}  z\inv,
\end{align*}
which proves \eqref{eq:Y-action-1}.
The proof of the rest equations are similar.
%
%\begin{align*}
%  &Y_{\wh\ell}(E_\ell(h_i),z_1)Y_{\wh\ell}(E_\ell(h_j),z_2)\vac\\
%  =&\exp\(\wt h_i^+(z_1)\)\exp\(\wt h_i^-(z_1)\)
%  \exp\(\wt h_j^+(z_2)\)\exp\(\wt h_j^-(z_2)\)\vac\\
%  =&\exp\(\wt h_i^+(z_1)\)\exp\(\wt h_j^+(z_2)\)\vac
%  g_{ij,\hbar}(z_1-z_2+ 2r\ell\hbar)g_{ij,\hbar}(z_1-z_2)\inv.
%\end{align*}
%Taking $z_2\to 0$, we get that
%\begin{align*}
%  &Y_{\wh\ell}(E_\ell(h_i),z)E_\ell(h_j)=
%  \exp\(\wt h_i^+(z)\)E_\ell(h_j)
%  g_{ij,\hbar}(z+ 2r\ell\hbar)g_{ij,\hbar}(z)\inv.
%\end{align*}
%
%\begin{align*}
%  &Y_{\wh\ell}(E_\ell(h_i),z_1)Y_{\wh\ell}(x_j^\pm,z_2)\vac\\
%  =&\exp\(\wt h_i^+(z_1)\)\exp\(\wt h_i^-(z_1)\)Y_{\wh\ell}(x_j^\pm,z_2)\vac\\
%  =&\exp\(\wt h_i^+(z_1)\)Y_{\wh\ell}(x_j^\pm,z_2)\exp\(\wt h_i^-(z_1)\)\vac g_{ij,\hbar}(z_1-z_2)^{\pm 1}\\
%  =&\exp\(\wt h_i^+(z_1)\)Y_{\wh\ell}(x_j^\pm,z_2)\vac g_{ij,\hbar}(z_1-z_2)^{\pm 1}.
%\end{align*}
%Taking $z_2\to 0$, we get that
%\begin{align*}
%  &Y_{\wh\ell}(E_\ell(h_i),z)x_j^\pm=\exp\(\wt h_i^+(z)\)x_j^\pm g_{ij,\hbar}(z)^{\pm 1}.
%\end{align*}
%
%\begin{align*}
%  &Y_{\wh\ell}(x_i^\pm,z_1)Y_{\wh\ell}(E_\ell(h_j),z_2)\vac\\
%  =&Y_{\wh\ell}(x_i^\pm,z_1)\exp\(\wt h_j^+(z_2)\)\exp\(\wt h_j^-(z_2)\)\vac\\
%  =&\exp\(\wt h_j^+(z_2)\)Y_{\wh\ell}(x_i^\pm,z_1)\exp\(\wt h_j^-(z_2)\)\vac g_{ji,\hbar}(z_1-z_2+2r\ell\hbar)^{\mp 1}\\
%  =&\exp\(\wt h_j^+(z_2)\)Y_{\wh\ell}(x_i^\pm,z_1)\vac g_{ji,\hbar}(z_1-z_2+2r\ell\hbar)^{\mp 1}.
%\end{align*}
%Taking $z_2\to 0$, we get that
%\begin{align*}
%  &Y_{\wh\ell}(x_i^\pm,z)E_\ell(h_j)=\exp\(\wt h_j(0)\)e^{z\partial} x_i^\pm g_{ji,\hbar}(z+2r\ell\hbar)^{\mp 1}.
%\end{align*}
\end{proof}

\begin{lem}\label{lem:S-E}
For $i,j\in I$, we have that
\begin{align}
  &S_{\ell,\ell'}(z)\(E_\ell(h_j)\ot h_i\)\label{eq:S-E-1}
  =E_\ell(h_j)\ot h_i+E_\ell(h_j)\ot \vac\\
  &\quad\ot \pd{z}\log f(z)^{-[r_ia_{ij}]_{q}[r\ell']_{q}[r\ell]_{q}
  \(q-q\inv\)^2q^{-r\ell}},\nonumber\\
  &S_{\ell,\ell'}(z)\(h_j\ot E_{\ell'}(h_i)\)
  =h_j\ot E_{\ell'}(h_i)+\vac \ot E_{\ell'}(h_i)\label{eq:S-E-2}\\
  &\quad\ot \pd{z}\log f(z)^{-[r_ja_{ji}]_{q}[r\ell']_{q}[r\ell]_{q}
  \(q-q\inv\)^2q^{r\ell'}},\nonumber\\
  &S_{\ell,\ell'}(z)\(E_\ell(h_j)\ot E_{\ell'}(h_i)\)
  =E_\ell(h_j)\ot E_{\ell'}(h_i)\label{eq:S-E-3}\\
  &\quad\ot f(z)^{ (q^{r_ia_{ij}}-q^{-r_ia_{ij}})(q^{2r\ell'}+q^{-2r\ell}-1-q^{2r(\ell'-\ell)}) }
  %\frac{f(z-r_ia_{ij}\hbar)f(z+(2r(\ell'-\ell)-r_ia_{ij})\hbar)}
  %  {f(z+r_ia_{ij}\hbar)f(z+(2r(\ell'-\ell)+r_ia_{ij})\hbar)}\nonumber\\
  %&\quad\times \frac{f(z+(2r\ell'+r_ia_{ij})\hbar)f(z-(2r\ell-r_ia_{ij})\hbar)}
  %  {f(z+(2r\ell'-r_ia_{ij})\hbar)f(z-(2r\ell+r_ia_{ij})\hbar)}
  %\frac{g_{ij,\hbar}(z)g_{ij,\hbar}(z+2r(\ell'-\ell)\hbar)}{g_{ij,\hbar}(z+2r\ell'\hbar) g_{ij,\hbar}(z-2r\ell\hbar)}
  ,\nonumber\\
  &S_{\ell,\ell'}(z)\(x_j^\pm\ot E_{\ell'}(h_i)\)
  =x_j^\pm\ot E_{\ell'}(h_i)\label{eq:S-E-4}
  \ot f(z)^{\pm(q^{r_ia_{ij}} -q^{-r_ia_{ij}})(1-q^{2r\ell'})  }
  %\frac{f(z\pm r_ia_{ij}\hbar)f(z+(2r\ell'\mp r_ia_{ij})\hbar)}
  %{f(z\mp r_ia_{ij}\hbar) f(z+(2r\ell'\pm r_ia_{ij})\hbar) }
  %g_{ij,\hbar}(z+2r\ell'\hbar)^{\pm 1}g_{ij,\hbar}(z)^{\mp 1}
  ,\\
  &S_{\ell,\ell'}(z)\(E_\ell(h_j)\ot x_i^\pm\)
  =E_\ell(h_j)\ot x_i^\pm\label{eq:S-E-5}
  \ot f(z)^{\pm(q^{r_ia_{ij}} -q^{-r_ia_{ij}})(1-q^{-2r\ell})  }
  %\frac{f(z\pm r_ia_{ij}\hbar)f(z-(2r\ell\pm r_ia_{ij})\hbar)}
  %{f(z\mp r_ia_{ij}\hbar) f(z-(2r\ell\mp r_ia_{ij})\hbar) }
  %g_{ij,\hbar}(z-2r\ell\hbar)^{\pm 1}g_{ij,\hbar}(z)^{\mp 1}
  .
\end{align}
\end{lem}

\begin{proof}
From \eqref{eq:S-twisted-1} and \eqref{eq:multqyb-der-shift}, we have that
\begin{align}
  &S_{\ell,\ell'}(z)(q^{-r\ell\partial}2\hbar f_0(2\hbar\partial)h_j\ot h_i)\nonumber\\
  =&q^{-r\ell\partial\ot 1-r\ell\pd{z}}2\hbar f_0\(2\hbar\(\partial\ot 1+1\ot\pd{z}\)\)
    S_{\ell,\ell'}(z)(h_j\ot h_i)\nonumber\\
  =&q^{-r\ell\partial}2\hbar f_0(2\hbar\partial)h_j\ot h_i
   +\vac\ot\vac\ot q^{-r\ell\pd{z}}2\hbar f_0\(2\hbar\pd{z}\){\wh{\ell,\ell'}}_{ij}(-z)\nonumber\\
  =&q^{-r\ell\partial}2\hbar f_0(2\hbar\partial)h_j\ot h_i
  +\vac\ot\vac\ot \(1-q^{-2r\ell\pd{z}}\){\wh{\ell,\ell'}}_{ij}^{1,+}(-z)\label{eq:S-E-temp-1}\\
  %=&q^{-r\ell\partial}2\hbar f_0(2\hbar\partial)h_j\ot h_i
%  +\vac\ot\vac\ot \(1-q^{-2r\ell\pd{z}}\)
%  [r_ia_{ij}]_{q^{\pd{z}}}[r\ell']_{q^{\pd{z}}}
%  \(q^{\pd{z}}-q^{-\pd{z}}\)\pd{z}f(z)\nonumber\\
  =&q^{-r\ell\partial}2\hbar f_0(2\hbar\partial)h_j\ot h_i
  +\vac\ot\vac
  \ot\pd{z}f(z)^{[r_ia_{ij}]_{q}[r\ell']_{q}
  [r\ell]_{q}(q-q\inv)^2q^{-r\ell} },\nonumber
\end{align}
where the equation \eqref{eq:S-E-temp-1} follows from \eqref{eq:special-tau-tech1-3}.
Then we complete the proof of \eqref{eq:S-E-1} by using Lemma \ref{lem:S-special-tech-gen2}.
The proof of the rest equations are similar.
\end{proof}

\subsection{Proof of Proposition \ref{prop:S-Delta}}

\begin{lem}\label{lem:S3-1}
For $i,j\in I$, we have that
\begin{align}
  &S_{\{\ell,\ell'\},\ell''}(z)\(\Delta(h_j)\ot h_i\)=
  \Delta(h_j)\ot h_i+\Delta(\vac)\ot\vac\label{eq:S3-Delta-h-h}\\
  &\quad\ot \pdiff{z}{2}\log f(z)^{ [r_ia_{ij}]_{q}[r\ell'']_{q}
  [r(\ell+\ell')]_{q}
  (q-q\inv) },\nonumber\\
  &S_{\{\ell,\ell'\},\ell''}(z)\(\Delta(h_j)\ot x_i^\pm\)
  =\Delta(h_j)\ot x_i^\pm\mp \Delta(\vac)\ot x_i^\pm\label{eq:S3-Delta-h-x}\\
  &\quad\ot \pd{z}\log f(z)^{[r_ja_{ji}]_{q}[r(\ell+\ell')]_{q}(q-q\inv) },\nonumber\\
  &S_{\{\ell,\ell'\},\ell''}(z)\(\Delta(x_j^\pm)\ot h_i\)
  =\Delta(x_j^\pm)\ot h_i\pm\Delta(x_j^\pm)\ot \vac\label{eq:S3-Delta-x-h}\\
  &\quad\ot \pd{z}\log f(z)^{ [r_ia_{ij}]_{q}[r\ell'']_{q}(q-q\inv) }.\nonumber
\end{align}
\end{lem}

\begin{proof}
Recall from \eqref{eq:def-Delta-h} that $\Delta(h_i)=q^{-r\ell'\partial}h_i\ot \vac+\vac\ot q^{r\ell\partial}h_i$. Then
\begin{align*}
  &S_{\{\ell,\ell'\},\ell''}(z)\(\Delta(h_j)\ot h_i\)\\
  =&S_{\ell',\ell''}^{23}(z)S_{\ell,\ell''}^{13}(z)\(q^{-r\ell'\partial}h_j\ot \vac\ot h_i
  +\vac\ot q^{r\ell\partial}h_j\ot h_i\)\\
  =&(q^{-r\ell'\hbar}\ot 1\ot 1) S_{\ell,\ell''}^{13}(z-r\ell'\hbar)\(h_j\ot \vac\ot h_i\)\\
  &+(1\ot q^{r\ell\hbar}\ot 1) S_{\ell',\ell''}(z+r\ell\hbar)\(\vac\ot h_j\ot h_i\)\\
  =&q^{-r\ell'\partial}h_j\ot \vac\ot h_i
  +\vac\ot q^{r\ell\partial}h_j\ot h_i\\
  &+\vac\ot\vac\ot\vac\ot
   \pdiff{z}{2}\log f(z)^{ [r_ia_{ij}]_{q}[r\ell'']_{q}
  ([r\ell]_{q}q^{-r\ell'}
  +[r\ell']_{q}q^{r\ell})(q-q\inv) }\\
  =&q^{-r\ell'\partial}h_j\ot \vac\ot h_i
  +\vac\ot q^{r\ell\partial}h_j\ot h_i\\
  &+\vac\ot\vac\ot\vac\ot
  \pdiff{z}{2}\log f(z)^{ [r_ia_{ij}]_{q}[r\ell'']_{q}
  [r(\ell+\ell')]_{q}
  \(q-q\inv\) }\\
  =&\Delta(h_j)\ot h_i+\Delta(\vac)\ot\vac\ot
  \pdiff{z}{2}\log f(z)^{ [r_ia_{ij}]_{q}[r\ell'']_{q}
  [r(\ell+\ell')]_{q}
  \(q-q\inv\) },
  %\\
%  =&\(\Delta\ot 1\)S_{\ell+\ell',\ell''}(z)\(h_j\ot h_i\),
\end{align*}
where the second equation follows from Lemma \ref{lem:multqyb-shift-total} and
the third equation follows from \eqref{eq:S-twisted-1}.
We also have that
\begin{align*}
  &S_{\{\ell,\ell'\},\ell''}(z)\(\Delta(h_j)\ot x_i^\pm\)\\
  =&S_{\ell',\ell''}^{23}(z)S_{\ell,\ell''}^{13}(z)\(q^{-r\ell'\partial}h_j\ot \vac\ot x_i^\pm
  +\vac\ot q^{r\ell\partial}h_j\ot x_i^\pm\)\\
  =&\(q^{-r\ell'\partial}\ot 1\ot 1\)S_{\ell,\ell''}^{13}(z-r\ell'\hbar)\(h_j\ot \vac\ot x_i^\pm\)\\
  &+\(1\ot q^{r\ell\partial}\ot 1\)S_{\ell',\ell''}^{23}(z+r\ell\hbar)\(\vac\ot q^{r\ell\partial}h_j\ot x_i^\pm\)\\
  =&q^{-r\ell'\partial}h_j\ot \vac\ot x_i^\pm
  +\vac\ot q^{r\ell\partial}h_j\ot x_i^\pm\mp \vac\ot \vac\ot x_i^\pm\\
  &\ot \pd{z}\log f(z)^{ [r_ja_{ji}]_{q}([r\ell]_{q}q^{-r\ell'}
  +[r\ell']_{q}q^{r\ell})
  (q-q\inv) }\\
  =&\Delta(h_j)\ot x_i^\pm\mp \Delta(\vac)\ot x_i^\pm
  \ot \pd{z}\log f(z)^{ [r_ja_{ji}]_{q}[r(\ell+\ell')]_{q}(q-q\inv) },
\end{align*}
where the second equation follows from Lemma \ref{lem:multqyb-shift-total} and
the third equation follows from \eqref{eq:S-twisted-3}.
The ``$-$'' case of \eqref{eq:S3-Delta-x-h} follows from the following two relations:
\begin{align}
  &S_{\{\ell,\ell'\},\ell''}(z)\(x_j^\pm\ot \vac\ot h_i\)\nonumber\\
  =&S_{\ell',\ell''}^{23}(z)S_{\ell,\ell''}^{13}(z)\(x_j^\pm\ot \vac\ot h_i\)
  =S_{\ell,\ell''}^{13}(z)\(x_j^\pm\ot \vac\ot h_i\)\nonumber\\
  =&x_j^\pm\ot \vac\ot h_i\pm x_j^\pm\ot\vac\ot\vac
  \ot \pd{z}\log f(z)^{ [r_ia_{ij}]_{q}[r\ell'']_q(q-q\inv) },\label{eq:S3-x-vac-h}
\end{align}
where the last equation follows from \eqref{eq:S-twisted-2};
\begin{align}
  &S_{\{\ell,\ell'\},\ell''}(z)\(\vac\ot x_j^\pm\ot h_i\)\nonumber\\
  =&S_{\ell',\ell''}^{23}(z)S_{\ell,\ell''}^{13}(z)\(\vac\ot x_j^\pm\ot h_i\)
  =S_{\ell',\ell''}^{23}(z)\(\vac\ot x_j^\pm\ot h_i\)\nonumber\\
  =&\vac\ot x_j^\pm\ot h_i\pm \vac\ot x_j^\pm\ot\vac
  \ot
  \pd{z}\log f(z)^{[r_ia_{ij}]_{q}[r\ell'']_q
  (q-q\inv) },\label{eq:S3-vac-x-h}
\end{align}
where the last equation follows from \eqref{eq:S-twisted-2}.
In order to prove the ``$+$'' case of \eqref{eq:S3-Delta-x-h}, we need the following relation:
\begin{align}
  &S_{\{\ell,\ell'\},\ell''}(z)\( E_\ell(h_j)\ot x_j^+ \ot h_i \)
  =S_{\ell',\ell''}^{23}(z)S_{\ell,\ell''}^{13}(z)\( E_\ell(h_j)\ot x_j^+ \ot h_i \)\nonumber\\
  =&S_{\ell',\ell''}^{23}(z)\(E_\ell(h_j)\ot x_j^+ \ot h_i\)
  \nonumber\\
  &-E_\ell(h_j)\ot x_j^+\ot \vac\ot \pd{z}\log f(z)^{ [r_ia_{ij}]_{q}[r\ell]_q
  [r\ell'']_q(q-q\inv)^2q^{-r\ell} } \nonumber\\
  =&E_\ell(h_j)\ot x_j^+ \ot h_i
  +E_\ell(h_j)\ot x_j^+\ot \vac
  \ot
  \pd{z}\log f(z)^{ [r_ia_{ij}]_{q}
  [r\ell'']_q\(1-1+q^{-2r\ell}\)
  \(q-q\inv\) }\nonumber\\
  =&E_\ell(h_j)\ot x_j^+ \ot h_i
  +E_\ell(h_j)\ot x_j^+\ot \vac
  \ot
  \pd{z}\log f(z)^{[r_ia_{ij}]_{q}
  [r\ell'']_qq^{-2r\ell}(q-q\inv) },\label{eq:S3-E-x-h}
\end{align}
where the second equation follows from \eqref{eq:S-E-1} and the third equation follows from \eqref{eq:S-twisted-2}.
Then we have that
\begin{align*}
  &S_{\{\ell,\ell'\},\ell''}(z)\(\Delta(x_j^+)\ot h_i\)\\
  =&S_{\{\ell,\ell'\},\ell''}(z)(x_j^+\ot\vac \ot h_i)
  +S_{\{\ell,\ell'\},\ell''}(z)(
  q^{2r\ell\partial}E_\ell(h_j)\ot q^{2r\ell\partial} x_j^+\ot h_i)\\
  =&x_j^+\ot\vac \ot h_i+x_j^+\ot\vac \ot\vac
  \ot \pd{z}\log f(z)^{ [r_ia_{ij}]_{q}[r\ell'']_q(q-q\inv) }\\
  &+S_{\{\ell,\ell'\},\ell''}(z)(
  q^{2r\ell\partial}E_\ell(h_j)\ot q^{2r\ell\partial} x_j^+\ot h_i)\\
  =&x_j^+\ot\vac \ot h_i+x_j^+\ot\vac \ot\vac
  \ot \pd{z}\log f(z)^{ [r_ia_{ij}]_{q}[r\ell'']_q(q-q\inv) }\\
  &+q^{2r\ell\partial}E_\ell(h_j)\ot q^{2r\ell\partial} x_j^+\ot h_i
  +q^{2r\ell\partial}E_\ell(h_j)\ot q^{2r\ell\partial} x_j^+\ot\vac
  \ot \pd{z}\log f(z)^{ [r_ia_{ij}]_{q}[r\ell'']_q(q-q\inv) }\\
  =&\Delta(x_j^+)\ot h_i+\Delta(x_j^+)\ot \vac
  \ot \pd{z}\log f(z)^{ [r_ia_{ij}]_{q}[r\ell'']_q(q-q\inv) },
\end{align*}
where the second equation follows from \eqref{eq:S3-x-vac-h} and the third equation follows from Lemma \ref{lem:multqyb-shift-total} and \eqref{eq:S3-E-x-h}.
\end{proof}

\begin{lem}\label{lem:S3-2}
For $i,j\in I$ and $\epsilon\in\{\pm\}$, we have that
\begin{align}\label{eq:S3-Delta-x-x}
  &S_{\{\ell,\ell'\},\ell''}(z)\(\Delta(x_j^\pm)\ot x_i^\epsilon\)
  =\Delta(x_j^\pm)\ot x_i^\epsilon\ot f(z)^{ q^{\mp\epsilon r_ia_{ij}}-q^{\pm \epsilon r_ia_{ij}} }.
\end{align}
\end{lem}

\begin{proof}
Using \eqref{eq:S-twisted-4}, we have that
\begin{align*}
  &S_{\{\ell,\ell'\},\ell''}(z)\(\Delta(x_j^-)\ot x_i^\epsilon\)\\
  =&S_{\ell',\ell''}^{23}(z)S_{\ell,\ell''}^{13}(z)\(x_j^-\ot\vac\ot x_i^\epsilon
  +\vac\ot x_j^-\ot x_i^\epsilon\)\\
  =&x_j^-\ot\vac\ot x_i^\epsilon\ot f(z)^{ q^{\epsilon r_ia_{ij}}-q^{- \epsilon r_ia_{ij}} }
  +\vac\ot x_j^-\ot x_i^\epsilon\ot f(z)^{ q^{\epsilon r_ia_{ij}}-q^{- \epsilon r_ia_{ij}} }\\
  =&\Delta(x_j^-)\ot x_i^\epsilon\ot f(z)^{ q^{\epsilon r_ia_{ij}}-q^{- \epsilon r_ia_{ij}} }.
\end{align*}
Note that
\begin{align*}
  &S_{\{\ell,\ell'\},\ell''}(z)\(q^{2r\ell\partial}E_\ell(h_j)\ot q^{2r\ell\partial}x_j^+\ot x_i^\epsilon
  \)\\
  =&S_{\ell',\ell''}^{23}(z)S_{\ell,\ell''}^{13}(z)\(q^{2r\ell\partial}E_\ell(h_j)\ot q^{2r\ell\partial}x_j^+\ot x_i^\epsilon
  \)\\
  =&\(q^{2r\ell\partial}\ot q^{2r\ell\partial}\ot 1\)S_{\ell',\ell''}^{23}(z+2r\ell\hbar)S_{\ell,\ell''}^{13}(z+2r\ell\hbar)
  \(E_\ell(h_j)\ot x_j^+\ot x_i^\epsilon\)\\
  =&q^{2r\ell\partial}E_\ell(h_j)\ot q^{2r\ell\partial}x_j^+\ot x_i^\epsilon
  \ot f(z)^{ (q^{-\epsilon r_ia_{ij}}-q^{\epsilon r_ia_{ij}})(1-q^{2r\ell}+q^{2r\ell} }\\
  =&q^{2r\ell\partial}E_\ell(h_j)\ot q^{2r\ell\partial}x_j^+\ot x_i^\epsilon
  \ot f(z)^{ q^{-\epsilon r_ia_{ij}}-q^{\epsilon r_ia_{ij}} },
\end{align*}
where the second equation follows from Lemma \ref{lem:multqyb-shift-total} and the third equation follows from \eqref{eq:S-twisted-4} and \eqref{eq:S-E-5}.
Combining this with \eqref{eq:S-twisted-4}, we get that
\begin{align*}
  &S_{\{\ell,\ell'\},\ell''}(z)\(\Delta(x_j^+)\ot x_i^\epsilon\)\\
  =&S_{\ell',\ell''}^{23}(z)S_{\ell,\ell''}^{13}(z)\(
    x_j^+\ot\vac\ot x_i^\epsilon
    +q^{2r\ell\partial}E_\ell(h_j)\ot q^{2r\ell\partial}x_j^+\ot x_i^\epsilon
  \)\\
  =&x_j^+\ot\vac\ot x_i^\epsilon\ot f(z)^{ q^{-\epsilon r_ia_{ij}}-q^{\epsilon r_ia_{ij}} }
  +q^{2r\ell\partial}E_\ell(h_j)\ot q^{2r\ell\partial}x_j^+\ot x_i^\epsilon
  \ot f(z)^{ q^{-\epsilon r_ia_{ij}}-q^{\epsilon r_ia_{ij}} }\\
  =&\Delta(x_j^+)\ot x_i^\epsilon\ot f(z)^{ q^{-\epsilon r_ia_{ij}}-q^{\epsilon r_ia_{ij}} }.
\end{align*}
Therefore, we complete the proof of lemma.
\end{proof}

Note that $F_{\hat\g,\hbar}^{\ell}$ is generated by $\set{h_i,x_i^\pm}{i\in I}$.
Combining this with Lemmas \ref{lem:S-twisted}, \ref{lem:S3-1} and \ref{lem:S3-2}, we get the following result.
\begin{prop}\label{prop:S-Delta-1}
As operators on $F_{\hat\g,\hbar}^{\ell+\ell'}\wh\ot F_{\hat\g,\hbar}^{\ell''}$, one has
\begin{align*}
  S_{\{\ell,\ell'\},\ell''}(z) (\Delta\ot 1)=(\Delta\ot 1) S_{\ell+\ell',\ell''}(z).
\end{align*}
\end{prop}

Similarly, we have:

\begin{prop}\label{prop:S-Delta-2}
As operators on $F_{\hat\g,\hbar}^{\ell}\wh\ot F_{\hat\g,\hbar}^{\ell'+\ell''}$, one has
\begin{align*}
  S_{\ell,\{\ell',\ell''\}}(z)(1\ot\Delta)=(1\ot\Delta)S_{\ell,\ell'+\ell''}(z).
\end{align*}
\end{prop}

Proposition \ref{prop:S-Delta} is immediate from Propositions \ref{prop:S-Delta-1} and \ref{prop:S-Delta-2}.

\subsection{Proof of Propositions \ref{prop:Y-Delta-cartan} and \ref{prop:Delta-coasso}}

\begin{lem}\label{lem:S-Delta-3}
Denote by $S_\Delta(z)$ the quantum Yang-Baxter operator of $$F_{\hat\g,\hbar}^{\ell}\wh\ot F_{\hat\g,\hbar}^{\ell'}.$$
For $i,j\in I$, we have that
\begin{align}
  &S_\Delta(z)\(\Delta(h_j)\ot \Delta(h_i)\)
  =\Delta(h_j)\ot \Delta(h_i)+\vac\ot\vac\ot\vac\ot\vac\label{eq:S-Delta-1-1}\\
  &\quad\ot
  \pdiff{z}{2}\log f(z)^{[r_ia_{ij}]_{q}
  [r(\ell+\ell')]_{q}
  (q-q\inv) },\nonumber\\
  &S_\Delta(z)\(\Delta(h_j)\ot \Delta(x_i^\pm)\)=\Delta(h_j)\ot \Delta(x_i^\pm)\label{eq:S-Delta-3-1}\\
  &\quad\mp\vac\ot\vac\ot \Delta(x_i^\pm)\ot
  \pd{z}\log f(z)^{ [r_ja_{ji}]_{q}(q^{r(\ell+\ell')}-q^{-r(\ell+\ell')}) },\nonumber\\
  &S_\Delta(z)\(\Delta(x_j^\pm)\ot \Delta(h_i)\)=\Delta(x_j^\pm)\ot \Delta(h_i)\label{eq:S-Delta-3-2}\\
  &\quad\pm\Delta(x_j^\pm)\ot \vac\ot\vac\ot
  \pd{z}\log f(z)^{ [r_ia_{ij}]_{q}(q^{r(\ell+\ell')}-q^{-r(\ell+\ell')}) },\nonumber\\
  &S_\Delta(z)\(\Delta(x_j^{\epsilon_1})\ot \Delta(x_i^{\epsilon_2})\)
  =\Delta(x_j^{\epsilon_1})\ot \Delta(x_i^{\epsilon_2})\label{eq:S-Delta-3-3}
  \ot \frac{f(z)^{ q^{-\epsilon_1\epsilon_2r_ia_{ij}} }}{f(z)^{q^{\epsilon_1\epsilon_2r_ia_{ij}}} }.
\end{align}
\end{lem}

\begin{proof}
Notice that
\begin{align}\label{eq:def-S-Delta}
  S_\Delta(z)=&S_{\ell',\ell}^{23}(z)S_{\ell,\ell}^{13}(z)S_{\ell',\ell'}^{24}(z)S_{\ell,\ell'}^{14}(z)
  =S_{\{\ell,\ell'\},\ell}^{\{12\},3}(z)S_{\{\ell,\ell'\},\ell'}^{\{12\},4}(z).
\end{align}
Then as operators on $F_{\hat\g,\hbar}^{\ell+\ell'}\wh\ot F_{\hat\g,\hbar}^{\ell+\ell'}$, we have that
\begin{align*}
  &S_\Delta(z)(\Delta\ot\Delta)
  =S_{\{\ell,\ell'\},\ell}^{\{12\},3}(z)S_{\{\ell,\ell'\},\ell'}^{\{12\},4}(z)
  (\Delta\ot 1\ot 1)(1\ot \Delta)\\
  =&(\Delta\ot 1\ot 1)S_{\ell+\ell',\ell}^{12}(z)S_{\ell+\ell',\ell'}^{13}(z)(1\ot \Delta)
  =(\Delta\ot 1\ot 1)S_{\ell+\ell',\{\ell,\ell'\}}(z)(1\ot\Delta)\\
  =&(\Delta\ot 1\ot 1)(1\ot\Delta)S_{\ell+\ell',\ell+\ell'}(z)
  =(\Delta\ot\Delta)S_{\ell+\ell',\ell+\ell'}(z),
\end{align*}
where the second equation follows from Proposition \ref{prop:S-Delta-1} and the forth equation follows from Proposition \ref{prop:S-Delta-2}.
Combining this with Lemma \ref{lem:S-twisted}, we complete the proof.
\end{proof}

%\begin{proof}
%Recall from \eqref{eq:def-Delta-x+} and \eqref{eq:def-Delta-x-} that
%\begin{align*}
%  &\Delta(x_i^+)=x_i^+\ot\vac+q^{2r\ell\partial}E_\ell(h_i)\ot q^{2r\ell\partial}x_i^+,\\
%  &\Delta(x_i^-)=x_i^-\ot\vac+\vac\ot x_i^-.
%\end{align*}
%Then the equation \eqref{eq:S-Delta-3-1} follows from \eqref{eq:S-Delta-1-2}, \eqref{eq:S-Delta-1-3}, \eqref{eq:S-Delta-1-4} and Lemma \ref{lem:multqyb-shift-total}.
%The equation \eqref{eq:S-Delta-3-2} follows from \eqref{eq:S-Delta-1-5}, \eqref{eq:S-Delta-1-6},
%\eqref{eq:S-Delta-1-7} and Lemma \ref{lem:multqyb-shift-total}.
%And the equation \eqref{eq:S-Delta-3-3} follows from Lemmas \ref{lem:S-Delta-2} and \ref{lem:multqyb-shift-total}.
%\end{proof}

\begin{lem}\label{lem:Sing-Y}
For $i,j\in I$, we have that
\begin{align}
  &Y_{\wh\ell}(h_i,z)^-h_j=\vac [r_ia_{ij}]_{q^{\pd{z}}}[r\ell]_{q^{\pd{z}}}q^{r\ell\pd{z}}z^{-2},\label{eq:Sing-Y-1}\\
  &Y_{\wh\ell}(h_i,z)^-x_j^\pm
  =\pm x_j^\pm [r_ia_{ij}]_{q^{\pd{z}}}q^{r\ell\pd{z}}z\inv.\label{eq:Sing-Y-2}
\end{align}
\end{lem}

\begin{proof}
From Proposition \ref{prop:deform-datum} and \eqref{eq:tau-1}, we get that
\begin{align*}
  &Y_{\wh\ell}(h_i,z)^-h_j=Y(h_i,z)^-h_j+
  \vac\wh\ell_{ij}(z)^-\\
  =&\vac r_ia_{ij}r\ell z^{-2}+\vac [r_ia_{ij}]_{q^{\pd{z}}}[r\ell]_{q^{\pd{z}}}q^{r\ell\pd{z}}z^{-2}
  -\vac r_ia_{ij}r\ell z^{-2}\\
  =&\vac [r_ia_{ij}]_{q^{\pd{z}}}[r\ell]_{q^{\pd{z}}}q^{r\ell\pd{z}}z^{-2}.
\end{align*}
Similarly, \eqref{eq:Sing-Y-2} follows from Proposition \ref{prop:deform-datum} and \eqref{eq:tau-1}.
\end{proof}

\begin{lem}\label{lem:Y-Delta-1}
We denote by $Y_\Delta$ the vertex operator map of the twisted tensor product $\hbar$-adic quantum VA
$F_{\hat\g,\hbar}^{\ell}\wh\ot F_{\hat\g,\hbar}^{\ell'}$.
For $i,j\in I$, we have that
\begin{align}
  &Y_\Delta\(\Delta(h_i),z\)^-\Delta(h_j)\label{eq:Y-Delta-1-1}
  =\vac\ot\vac\ot\vac\ot\vac
  \ot[r_ia_{ij}]_{q^{\pd{z}}}[r(\ell+\ell')]_{q^{\pd{z}}}q^{r(\ell+\ell')\pd{z}}z^{-2},\\
  &Y_\Delta\(\Delta(h_i),z\)^-(x_j^\pm\ot\vac)\label{eq:Y-Delta-1-2}
  =\pm x_j^\pm\ot\vac\ot
  [r_ia_{ij}]_{q^{\pd{z}}}q^{r(\ell+\ell')\pd{z}}z\inv,\\
  &Y_\Delta\(\Delta(h_i),z\)^-(\vac\ot x_j^\pm)\label{eq:Y-Delta-1-3}
  =\pm \vac\ot x_j^\pm\ot
  [r_ia_{ij}]_{q^{\pd{z}}}q^{r(\ell+\ell')\pd{z}}z\inv,\\
  &Y_\Delta\(\Delta(h_i),z\)^-(E_\ell(h_j)\ot x_j^+)\label{eq:Y-Delta-1-4}
  =E_\ell(h_j)\ot x_j^+
  \ot [r_ia_{ij}]_{q^{\pd{z}}}q^{r(\ell+\ell')\pd{z}}q^{2r\ell\pd{z}}z\inv.
\end{align}
\end{lem}

\begin{proof}
From Lemma \ref{lem:vertex-op-K}, we have that
\begin{align}\label{eq:Y-Delta}
  Y_\Delta(z)=Y_{\wh\ell}^{12}(z)Y_{\wh{\ell'}}^{34}(z)S_{\ell,\ell'}^{23}(-z)\sigma.
\end{align}
Recall from \eqref{eq:def-Delta-h} that $\Delta(h_i)=q^{-r\ell'\partial}h_i\ot\vac+\vac\ot q^{r\ell\partial}h_i$.
Then we have that
\begin{align*}
  &Y_\Delta\(\Delta(h_i),z\)^-\Delta(h_j)\\
  =&Y_\Delta\(q^{-r\ell'\partial}h_i\ot\vac+\vac\ot q^{r\ell\partial}h_i,z\)^-
    \(q^{-r\ell'\partial}h_j\ot\vac+\vac\ot q^{r\ell\partial}h_j\)\\
  =&Y_{\wh\ell}(q^{-r\ell'\partial}h_i,z)^-q^{-r\ell'\partial}h_j\ot\vac
  +\vac\ot Y_{\wh{\ell'}}(q^{r\ell\partial}h_i,z)^-q^{r\ell\partial}h_j\\
  &+\Sing_z Y_{\wh\ell}^{12}(z)Y_{\wh{\ell'}}^{34}(z)\vac
  \ot S_{\ell,\ell'}(-z)\(q^{-r\ell'\partial}h_j\ot q^{r\ell\partial}h_i\)\ot\vac\\
  =&Y_{\wh\ell}(q^{-r\ell'\partial}h_i,z)^-q^{-r\ell'\partial}h_j\ot\vac
  +\vac\ot Y_{\wh{\ell'}}(q^{r\ell\partial}h_i,z)^-q^{r\ell\partial}h_j\\
  &+\Sing_z Y_{\wh\ell}^{12}(z)Y_{\wh{\ell'}}^{34}(z)\vac
  \ot \(q^{-r\ell'\partial+r\ell'\pd{z}}\ot q^{r\ell\partial+r\ell\pd{z}}\) S_{\ell,\ell'}(-z)\(h_j\ot h_i\)\ot\vac\\
  =&Y_{\wh\ell}(q^{-r\ell'\partial}h_i,z)^-q^{-r\ell'\partial}h_j\ot\vac
  +\vac\ot Y_{\wh{\ell'}}(q^{r\ell\partial}h_i,z)^-q^{r\ell\partial}h_j\\
  &+\Sing_z Y_{\wh\ell}(\vac,z)q^{r\ell\partial}h_j\ot Y_{\wh{\ell'}}(q^{-r\ell'\partial}h_i,z)\vac+\Sing_z \vac\ot\vac\\
  &\quad\ot q^{r(\ell+\ell')\pd{z}}[r_ia_{ij}]_{q^{\pd{z}}}[r\ell]_{q^{\pd{z}}}[r\ell']_{q^{\pd{z}}}
  \(q^{-\pd{z}}-q^{\pd{z}}\)\pdiff{z}{2}\log f(z)\\
  =&\vac\ot\vac
  \ot[r_ia_{ij}]_{q^{\pd{z}}}q^{r(\ell+\ell')\pd{z}}\bigg([r\ell]_{q^{\pd{z}}}q^{-r\ell'\pd{z}}\\
  &\quad+[r\ell']_{q^{\pd{z}}}q^{-r\ell\pd{z}}
   +[r\ell]_{q^{\pd{z}}}\(q^{r\ell'\pd{z}}-q^{-r\ell'\pd{z}}\)
  \bigg)z^{-2}\\
  =&\vac\ot\vac\ot[r_ia_{ij}]_{q^{\pd{z}}}[r(\ell+\ell')]_{q^{\pd{z}}}q^{r(\ell+\ell')\pd{z}}z^{-2},
\end{align*}
where the third equation follows from Lemma \ref{lem:multqyb-shift-total},
the fourth equation follows from \eqref{eq:S-twisted-1}, and the fifth equation follows from \eqref{eq:Sing-Y-1}.
%Hence, we complete the proof of \eqref{eq:Y-Delta-1-1}.
For the equation \eqref{eq:Y-Delta-1-2}, we have that
\begin{align*}
  &Y_\Delta\(\Delta(h_i),z\)^-(x_j^\pm\ot\vac)\\
  =&Y_\Delta\(q^{-r\ell'\partial}h_i\ot\vac+\vac\ot q^{r\ell\partial}h_i,z\)^-
  (x_j^\pm\ot\vac)\\
  =&\Sing_z Y_{\wh\ell}^{12}(z)Y_{\wh{\ell'}}^{34}(z)S_{\ell,\ell'}^{23}(-z)
  \(
  q^{-r\ell'\partial}h_i\ot x_j^\pm\ot\vac\ot\vac
  +\vac\ot x_j^\pm\ot q^{r\ell\partial}h_i\ot\vac\)\\
  =&\Sing_z Y_{\wh\ell}^{12}(z)Y_{\wh{\ell'}}^{34}(z)\bigg(
    q^{-r\ell'\partial}h_i\ot x_j^\pm\ot\vac\ot\vac
  +\vac\ot x_j^\pm\ot q^{r\ell\partial}h_i\ot\vac\\
  &\quad\pm\vac\ot x_j^\pm \ot \vac\ot \vac\ot q^{r\ell\pd{z}}[r_ia_{ij}]_{q^{\pd{z}}}\(q^{r\ell'\pd{z}}-q^{-r\ell'\pd{z}}\)
  \pd{z}\log f(z)
  \bigg)\\
  =&\pm x_j^\pm\ot\vac\ot [r_ia_{ij}]_{q^{\pd{z}}}\(q^{r(\ell-\ell')\pd{z}}+q^{r(\ell+\ell')\pd{z}}-q^{r(\ell-\ell')\pd{z}}\)z\inv\\
  =&\pm x_j^\pm\ot \vac\ot [r_ia_{ij}]_{q^{\pd{z}}}q^{r(\ell+\ell')\pd{z}}z\inv,
\end{align*}
where the third equation follows from Lemma \ref{lem:multqyb-shift-total} and \eqref{eq:S-twisted-2},
and the forth equation follows from \eqref{eq:Sing-Y-2}.
From the equation \eqref{eq:Sing-Y-2} we get that
\begin{align*}
  &Y_\Delta\(\Delta(h_i),z\)^-(\vac\ot x_j^\pm)\\
  =&Y_\Delta\(q^{-r\ell'\partial}h_i\ot\vac+\vac\ot q^{r\ell\partial}h_i,z\)(\vac\ot x_j^\pm)^-\\
  =&\Sing_z Y_{\wh\ell}^{12}(z)Y_{\wh{\ell'}}^{34}(z)S_{\ell,\ell'}^{23}(-z)
  \(
  q^{-r\ell'\partial}h_i\ot\vac\ot\vac\ot x_j^\pm
  +\vac\ot \vac \ot q^{r\ell\partial}h_i\ot x_j^\pm\)\\
  =&\vac\ot Y_{\wh{\ell'}}(q^{r\ell\partial}h_i,z)^-x_j^\pm
  =\pm \vac\ot x_j^\pm\ot [r_ia_{ij}]_{q^{\pd{z}}}q^{r(\ell+\ell')\pd{z}}z\inv,
\end{align*}
Finally, we have that
\begin{align*}
  &Y_\Delta\(\Delta(h_i),z\)^-\(E_\ell(h_j)\ot x_j^+\)\\
  =&Y_\Delta\(q^{-r\ell'\partial}h_i\ot\vac+\vac\ot q^{r\ell\partial}h_i,z\)^-
    \(E_\ell(h_j)\ot x_j^+\)\\
  =&\Sing_z Y_{\wh\ell}^{12}(z)Y_{\wh{\ell'}}^{34}(z)S_{\ell,\ell'}^{23}(-z)
  \(
    q^{-r\ell'\partial}h_i\ot E_\ell(h_j)\ot \vac\ot x_j^+
    +\vac\ot E_\ell(h_j)\ot q^{r\ell\partial}h_i\ot x_j^+
  \)\\
  =&Y_{\wh\ell}^{12}(z)Y_{\wh{\ell'}}^{34}(z)\bigg(
    q^{-r\ell'\partial}h_i\ot E_\ell(h_j)\ot \vac\ot x_j^+
    +\vac\ot E_\ell(h_j)\ot q^{r\ell\partial}h_i\ot x_j^+\\
  &
  +\vac\ot E_\ell(h_j)\ot \vac\ot x_j^+
  \ot [r_ia_{ij}]_{q^{\pd{z}}}
  \(q^{r\ell\pd{z}}-q^{-r\ell\pd{z}}\)
  \(q^{r\ell'\pd{z}}-q^{-r\ell'\pd{z}}\)q^{2r\ell\pd{z}}\pd{z}\log f(z)
  \bigg)\\
  =&E_\ell(h_j)\ot x_j^+\\
  &\ot [r_ia_{ij}]_{q^{\pd{z}}}q^{r(\ell+\ell')\pd{z}}
  \(q^{2r(\ell-\ell')\pd{z}}-q^{-2r\ell'\pd{z}} +1+\(q^{2r\ell\pd{z}}-1\)\(1-q^{-2r\ell'\pd{z}}\) \)z\inv\\
  =&E_\ell(h_j)\ot x_j^+\ot [r_ia_{ij}]_{q^{\pd{z}}}q^{r(\ell+\ell')\pd{z}}q^{2r\ell\pd{z}}z\inv,
\end{align*}
where third equation follows from Lemma \ref{lem:multqyb-shift-total} and \eqref{eq:S-E-1},
and the forth equation follows from \eqref{eq:Sing-Y-2} and \eqref{eq:Y-action-1}.
\end{proof}

From the definition \eqref{eq:def-Delta-x+}, \eqref{eq:def-Delta-x-} and the equations \eqref{eq:Y-Delta-1-2}, \eqref{eq:Y-Delta-1-3} and \eqref{eq:Y-Delta-1-4}, we immediately get the following result.
\begin{lem}\label{lem:Y-Delta-h-x}
For $i,j\in I$, we have that
\begin{align*}
  Y_\Delta\(\Delta(h_i),z\)^-\Delta(x_j^\pm)=\pm \Delta(x_j^\pm)\ot [r_ia_{ij}]_{q^{\pd{z}}}q^{r(\ell+\ell')\pd{z}}z\inv.
\end{align*}
\end{lem}

\begin{lem}\label{lem:Y-Delta-x+x-}
In $V_{\hat\g,\hbar}^\ell\wh\ot V_{\hat\g,\hbar}^{\ell'}$, we have that ($i\in I$):
\begin{align*}
  &Y_\Delta\(\Delta(x_i^+),z\)^-\Delta(x_i^-)\\
  =&\frac{1}{q^{r_i}-q^{-r_i}}\(\vac\ot\vac z\inv-q^{-2r\ell'\partial}E_\ell(h_i)\ot E_{\ell'}(h_i)(z+2r(\ell+\ell')\hbar)\inv\).
\end{align*}
\end{lem}

\begin{proof}
We have that
\begin{align*}
  &Y_\Delta\(x_i^+\ot\vac+q^{2r\ell\partial}E_\ell(h_i)\ot q^{2r\ell\partial}x_i^+,z\)^-\(x_i^-\ot\vac+\vac\ot x_i^-\)\\
  =&\Sing_z Y_{\wh\ell}^{12}(z)Y_{\wh{\ell'}}^{34}(z)S_{\ell,\ell'}^{23}(-z)
  \bigg(
    x_i^+\ot x_i^-\ot\vac\ot\vac +q^{2r\ell\partial}E_\ell(h_i)\ot x_i^-\ot q^{2r\ell\partial}x_i^+\ot \vac\\
  &\quad +x_i^+\ot\vac\ot\vac\ot x_i^-
    +q^{2r\ell\partial}E_\ell(h_i)\ot \vac\ot q^{2r\ell\partial}x_i^+\ot x_i^-
  \bigg)\\
  =&\Sing_z Y_{\wh\ell}^{12}(z)Y_{\wh{\ell'}}^{34}(z)
  \bigg(
    x_i^+\ot x_i^-\ot\vac\ot\vac\\
 &\quad   +q^{2r\ell\partial}E_\ell(h_i)\ot x_i^-\ot q^{2r\ell\partial}x_i^+\ot \vac
    \ot f(z)^{ (q^{-2r_i}-q^{2r_i})q^{2r\ell} }\\
 &\quad   +x_i^+\ot\vac\ot\vac\ot x_i^-
    +q^{2r\ell\partial}E_\ell(h_i)\ot \vac\ot q^{2r\ell\partial}x_i^+\ot x_i^-
  \bigg)\\
  =&\frac{1}{q^{r_i}-q^{-r_i}}\(\vac\ot\vac z\inv- E_\ell(h_i)\ot\vac(z+2r\ell\hbar)\inv\)\\
  &+\Sing_z\exp\(\wt h_i(z+2r\ell\hbar)\)x_i^-\ot e^{(z+2r\ell\hbar)\partial}x_i^+\\
  &+\frac{1}{q^{r_i}-q^{-r_i}}\Sing_z\Big(e^{(z+2r\ell\hbar)\partial}E_\ell(h_i)\ot \vac(z+2r\ell\hbar)\inv\\
  &\quad  -e^{(z+2r\ell\hbar)\partial}E_\ell(h_i)\ot E_{\ell'}(h_i)(z+2r(\ell+\ell')\hbar)\inv
  \Big)\\
  =&\frac{1}{q^{r_i}-q^{-r_i}}\Big(\vac\ot\vac z\inv
  -E_\ell(h_i)\ot\vac(z+2r\ell\hbar)\inv+E_\ell(h_i)\ot\vac(z+2r\ell\hbar)\inv\\
  &\quad-q^{-2r\ell'\partial}E_\ell(h_i)\ot E_{\ell'}(h_i)(z+2r(\ell+\ell')\hbar)\inv\Big)\\
  =&\frac{1}{q^{r_i}-q^{-r_i}}\(\vac\ot\vac z\inv
  -q^{-2r\ell'\partial}E_\ell(h_i)\ot E_{\ell'}(h_i)(z+2r(\ell+\ell')\hbar)\inv\).
\end{align*}
where the first equation follows from \eqref{eq:Y-Delta} and the definition of $\Delta(x_i^\pm)$ (see \eqref{eq:def-Delta-x+}, \eqref{eq:def-Delta-x-}),
the second equation follows from Lemma \ref{lem:multqyb-shift-total} and \eqref{eq:S-E-4},
and the third equation follows from Proposition \ref{prop:ideal-def-alt} and \eqref{eq:Y-action-3}.
\end{proof}

\begin{lem}\label{lem:E-ell+}
For any $i\in I$, we have that
\begin{align*}
  E_{\ell+\ell'}\(\Delta(h_i)\)
  =q^{-2r\ell'\partial} E_\ell(h_i)\ot E_{\ell'}(h_i).
\end{align*}
\end{lem}

\begin{proof}
Set
\begin{align*}
    &a=-q^{-r(\ell+2\ell')\partial}2\hbar f_0(2\hbar\partial) h_i\ot \vac,\quad
    b=-\vac\ot q^{-r\ell'\partial}2\hbar f_0(2\hbar\partial) h_i
\end{align*}
and $\wt h_i=a+b$.
Recall from \eqref{eq:def-S-Delta} that $S_\Delta(z)=S_{\ell',\ell}^{23}(z)S_{\ell,\ell}^{13}(z)S_{\ell',\ell'}^{24}(z)S_{\ell,\ell'}^{14}(z)$.
Then
\begin{align*}
  &S_\Delta(z)(b\ot a)\\
  =&S_{\ell',\ell}^{23}(z)S_{\ell,\ell}^{13}(z)S_{\ell',\ell'}^{24}(z)S_{\ell,\ell'}^{14}(z)
  \(\vac\ot q^{-r\ell'\partial}2\hbar f_0(2\hbar\partial)h_i\ot q^{-r(\ell+2\ell')\partial}2\hbar f_0(2\hbar\partial)h_i\ot \vac\)\\
  =&\bigg(1\ot q^{-r\ell'\partial-r\ell'\pd{z}}2\hbar f_0\(2\hbar\(\partial+\pd{z}\)\)\ot q^{-r(\ell+2\ell')\partial+r(\ell+2\ell')\pd{z}}\\
  &\times 2\hbar f_0\(2\hbar\(\partial-\pd{z}\)\)\ot 1\bigg)
  S_{\ell',\ell}^{23}(z)(\vac\ot h_i\ot h_i\ot\vac)\\
  =&b\ot a+\vac\ot\vac\ot\vac\ot\vac\ot q^{r(\ell+\ell')\pd{z}}4\hbar^2f_0\(2\hbar\pd{z}\)^2
   \pdiff{z}{2}\log f(z)^{ [2r_i]_{q} [r\ell]_{q}
  [r\ell']_{q}(q-q\inv) }\\
  =&b\ot a+\vac\ot\vac\ot\vac\ot\vac\ot
  \log f(z)^{ (q^{2r_i}-q^{-2r_i})(q^{2r\ell}-1)(q^{2r\ell'}-1) },
\end{align*}
where the second equation follows from Lemma \ref{lem:multqyb-shift-total} and
the third equation follows from \eqref{eq:S-twisted-1}.
Then
\begin{align*}
  &[Y_\Delta(a,z_1),Y_\Delta(b,z_2)]w\\
  =&Y_\Delta(a,z_1)Y_\Delta(b,z_2)w-Y_\Delta(z_2)Y_\Delta^{23}(z_1)S_\Delta^{12}(z_2-z_1)(b\ot a\ot w)\\
  &+Y_\Delta(z_2)Y_\Delta^{23}(z_1)S_\Delta^{12}(z_2-z_1)(b\ot a\ot w)-Y_\Delta(b,z_2)Y_\Delta(a,z_1)w\\
  =&Y_\Delta(Y_\Delta(a,z_1-z_2)b-Y_\Delta(a,-z_2+z_1)b,z_2)\\
  &+Y_\Delta(z_2)Y_\Delta^{23}(z_1)S_\Delta^{12}(z_2-z_1)(b\ot a\ot w)-Y_\Delta(b,z_2)Y_\Delta(a,z_1)w\\
  =&Y_\Delta(z_2)Y_\Delta^{23}(z_1)\(S_\Delta^{12}(z_2-z_1)(b\ot a\ot w)-b\ot a\ot w\)\\
  =&w\ot \iota_{z_2,z_1}\log f(z_2-z_1)^{ (q^{2r_i}-q^{-2r_i})(q^{2r\ell}-1)
  (q^{2r\ell'}-1) },
\end{align*}
where the second equation follows from \cite[(2.25)]{Li-h-adic}.
Applying $\Res_{z_1,z_2}z_1\inv z_2\inv$ on both hand sides, we get that
\begin{align*}
  &[a_{-1},b_{-1}]\\
  =&\Res_{z_1,z_2}z_1\inv z_2\inv
  \log f(z_2-z_1)^{ (q^{2r_i}-q^{-2r_i})(q^{2r\ell}-1)
  (q^{2r\ell'}-1) }\\
  =&\Res_{z_2}z_2\inv \log f(z_2)^{ (q^{2r_i}-q^{-2r_i})(q^{2r\ell}-1)
  (q^{2r\ell'}-1) }\\
  =&\Res_{z_2}z_2\inv \log f_0(z_2)^{ (q^{2r_i}-q^{-2r_i})(q^{2r\ell}-1)
  (q^{2r\ell'}-1) }\\
  =&\log \frac{f_0(2(r_i+r\ell+r\ell')\hbar)f_0(2r_i\hbar) }{f_0(2(r_i+r\ell)\hbar)f_0(2(r_i+r\ell')\hbar) }
  \frac{f_0(2(-r_i+r\ell)\hbar)f_0(2(-r_i+r\ell')\hbar) }{f_0(2(-r_i+r\ell+r\ell')\hbar)f_0(-2r_i\hbar) }\\
  =&\log \frac{f_0(2(r_i+r\ell+r\ell')\hbar)}{f_0(2(r_i-r\ell-r\ell')\hbar) }
  \frac{f_0(2(r_i-r\ell)\hbar) }{f_0(2(r_i+r\ell)\hbar)}
  \frac{f_0(2(r_i-r\ell')\hbar)}{f_0(2(r_i+r\ell')\hbar) }.
\end{align*}
From Baker-Campbell-Hausdorff formula, we have that
\begin{align*}
  &E_{\ell+\ell'}\(\Delta(h_i)\)
  =\(\frac{f_0(2(r_i+r\ell+r\ell')\hbar)}{f_0(2(r_i-r\ell-r\ell')\hbar) }\)^\half\exp(a_{-1}+b_{-1})\vac\ot\vac\\
  =&\(\frac{f_0(2(r_i+r\ell+r\ell')\hbar)}{f_0(2(r_i-r\ell-r\ell')\hbar) }\)^\half
  \exp(a_{-1})\vac\ot \exp(b_{-1})\vac\ot \exp(-[a_{-1},b_{-1}]/2)
  \\
  =&\exp(a_{-1})\vac\ot \exp(b_{-1})\vac
  \(\frac{f_0(2(r_i+r\ell)\hbar)}{f_0(2(r_i-r\ell)\hbar) }\)^\half
  \(\frac{f_0(2(r_i+r\ell')\hbar) }{f_0(2(r_i-r\ell')\hbar)}\)^\half\\
  =&
  q^{-2r\ell'\partial} E_\ell(h_i)\ot E_{\ell'}(h_i),
\end{align*}
which complete the proof.
\end{proof}

\noindent\emph{Proof of Proposition \ref{prop:Y-Delta-cartan}:}
In view of Remark \ref{rem:Jacobi-S},
the relation \eqref{eq:prop-Y-Delta-1} follows from \eqref{eq:S-Delta-3-1} and \eqref{eq:Y-Delta-1-1},
the relation \eqref{eq:prop-Y-Delta-2} follows from \eqref{eq:S-Delta-3-2} and Lemma \ref{lem:Y-Delta-h-x},
and the relation \ref{eq:prop-Y-Delta-3} follows from \eqref{eq:S-Delta-3-3} and Lemmas \ref{lem:Y-Delta-x+x-}, \ref{lem:E-ell+}.

\vspace{2mm}
\noindent\emph{Proof of Proposition \ref{prop:Delta-coasso}:}
It is easy to check that
\begin{align*}
  &(\Delta\ot 1)\circ\Delta(h_i)=(1\ot \Delta)\circ\Delta(h_i),\quad
  (\Delta\ot 1)\circ\Delta(x_i^-)=(1\ot \Delta)\circ\Delta(x_i^-)\quad\te{for }i\in I.
\end{align*}
By utilizing Lemma \ref{lem:E-ell+}, one can straightforwardly verify that
\begin{align*}
  &(\Delta\ot 1)\circ\Delta(x_i^+)=(1\ot \Delta)\circ\Delta(x_i^+)\quad\te{for }i\in I.
\end{align*}
Since $F_{\hat\g,\hbar}^\ell$ is generated by
$\set{h_i,x_i^\pm}{i\in I}$,
we complete the proof of Proposition \ref{prop:Delta-coasso}.

\subsection{Proof of Proposition \ref{prop:Y-Delta-serre}}

\begin{lem}\label{lem:Delta-x-pm-q-local}
For $i,j\in I$, we have that
\begin{align*}
  \Sing_z z^{-\delta_{ij}}(z-r_ia_{ij}\hbar)Y_\Delta(\Delta(x_i^\pm),z)\Delta(x_j^\pm)=0.
\end{align*}
\end{lem}

\begin{proof}
We only need to consider the ``$+$'' case. From \eqref{eq:Y-Delta}, we have that
\begin{align*}
  &Y_\Delta(E_\ell(h_i)\ot x_i^+,z)(E_\ell(h_j)\ot x_j^+)\\
  =&Y_{\wh\ell}^{12}(z)Y_{\wh{\ell'}}^{34}(z)S_{\ell,\ell'}^{23}(-z)
  \(E_\ell(h_i)\ot E_\ell(h_j)\ot x_i^+\ot x_j^+\)\\
  =&Y_{\wh\ell}(E_\ell(h_i),z)E_\ell(h_j)\ot Y_{\wh{\ell'}}(x_i^+,z)x_j^+ \ot
  f(z)^{(q^{-r_ia_{ij}}-q^{r_ia_{ij}})(1-q^{2r\ell})}\\
  =&\exp\(\wt h_i^+(z)\)E_\ell(h_j)\ot Y_{\wh{\ell'}}(x_i^+,z)x_j^+,
\end{align*}
where the second equation follows from \eqref{eq:S-E-5} and the last equation follows from \eqref{eq:Y-action-2}.
Then we have
\begin{align*}
  \Sing_z z^{-\delta_{ij}}(z-r_ia_{ij}\hbar)Y_\Delta(E_\ell(h_i)\ot x_i^+,z)(E_\ell(h_j)\ot x_j^+)=0.
\end{align*}
Similarly, by using \eqref{eq:Y-action-4} we have that
\begin{align*}
  &Y_\Delta(x_i^+\ot\vac,z)(E_\ell(h_j)\ot x_j^+)
  =Y_{\wh\ell}^{12}(z)Y_{\wh{\ell'}}^{34}(z)S_{\ell,\ell'}(-z)\(x_i^+\ot E_\ell(h_j)\ot \vac\ot x_j^+\)\\
  =&Y_{\wh\ell}(x_i^+,z)E_\ell(h_j)\ot x_j^+
  =\exp\(\wt h_j^+(0)\)x_i^+\ot x_j^+\ot f(z)^{(q^{r_ia_{ij}}-q^{-r_ia_{ij}})q^{2r\ell}},
\end{align*}
and by using \eqref{eq:S-twisted-4} and \eqref{eq:Y-action-3}, we have that
\begin{align*}
  &Y_\Delta(E_\ell(h_i)\ot x_i^+,z)(x_j^+\ot\vac)
  =Y_{\wh\ell}^{12}(z)Y_{\wh{\ell'}}^{34}(z)S_{\ell,\ell'}(-z)\( E_\ell(h_i)\ot x_j^+\ot x_i^+\ot\vac \)\\
  =&Y_{\wh\ell}^{12}(z)Y_{\wh{\ell'}}^{34}(z)\( E_\ell(h_i)\ot x_j^+\ot x_i^+\ot\vac  \)\ot f(z)^{q^{r_ia_{ij}}-q^{-r_ia_{ij}}}\\
  =&Y_{\wh\ell}(E_\ell(h_i),z)x_j^+\ot e^{z\partial}x_i\ot f(z)^{q^{r_ia_{ij}}-q^{-r_ia_{ij}}}
  =\exp\(\wt h_i^+(z)\)x_j^+\ot e^{z\partial}x_i^+.
\end{align*}
Then we have that
\begin{align*}
  &\Sing_zz^{-\delta_{ij}}(z-r_ia_{ij}\hbar)Y_\Delta(\Delta(x_i^+),z)\Delta(x_j^+)\\
  =&\Sing_zz^{-\delta_{ij}}(z-r_ia_{ij}\hbar)Y_\Delta(x_i^+\ot \vac+q^{2r\ell\partial} E_\ell(h_i)\ot q^{2r\ell\partial} x_i^+,z)\\
   &\quad(x_j^+\ot \vac+q^{2r\ell\partial}E_\ell(h_j)\ot q^{2r\ell\partial}x_j^+)\\
  =&\Sing_zz^{-\delta_{ij}}(z-r_ia_{ij}\hbar)Y_{\wh\ell}(x_i^+,z)x_j^+\ot\vac\\
    &+q^{2r\ell(\partial\ot1+1\ot\partial)}\Sing_zz^{-\delta_{ij}}(z-r_ia_{ij}\hbar)
    Y_\Delta(x_i^+\ot \vac,z-2r\ell\hbar)(E_\ell(h_j)\ot x_j^+)\\
    &+\Sing_zz^{-\delta_{ij}}(z-r_ia_{ij}\hbar)Y_\Delta(E_\ell(h_i)\ot x_i^+,z+2r\ell\hbar)(x_j^+\ot \vac)\\
    &+q^{2r\ell(\partial\ot1+1\ot\partial)}\Sing_zz^{-\delta_{ij}}(z-r_ia_{ij}\hbar)
     Y_\Delta(E_\ell(h_i)\ot x_i^+,z)(E_\ell(h_j)\ot x_j^+)\\
  =&\exp\(\wt h_j^+(2r\ell\hbar)\)q^{2r\ell\partial}x_i^+\ot q^{2r\ell\partial}x_j^+\\
  &\quad\ot \Sing_zz^{-\delta_{ij}}(z+r_ia_{ij}\hbar)f_0(z+r_ia_{ij}\hbar)f_0(z-r_ia_{ij}\hbar)\inv\\
  &+\Sing_zz^{-\delta_{ij}}(z-r_ia_{ij}\hbar)\exp\(\wt h_i^+(z+2r\ell\hbar)\)x_j^+\ot e^{z\partial}q^{2r\ell\partial}x_i^+\\
  =&\delta_{ij}\exp\(\wt h_j^+(2r\ell\hbar)\)q^{2r\ell\partial}x_i^+\ot q^{2r\ell\partial}x_i^+
  \ot z\inv \((r_ia_{ii}\hbar)f_0(r_ia_{ii}\hbar)f_0(-r_ia_{ii}\hbar)\inv-r_ia_{ii}\hbar\)\\
  =&0,
\end{align*}
which completes the proof of lemma.
\end{proof}

Combining Remark \ref{rem:Jacobi-S}, Lemma \ref{lem:Delta-x-pm-q-local} and \eqref{eq:S-Delta-3-3}, we get the following result.

\begin{lem}\label{lem:normal-ordering-Delta}
For $i,j\in I$, we have that
\begin{align*}
  &\iota_{z_1,z_2} f(z_1-z_2)^{-\delta_{ij}+q^{-r_ia_{ij}}} Y_\Delta(\Delta(x_i^\pm),z_1)
  Y_\Delta(\Delta(x_j^\pm),z_2)\\
  =&\iota_{z_2,z_1}f(-z_2+z_1)^{-\delta_{ij}+q^{r_ia_{ij}}}Y_\Delta(\Delta(x_j^\pm),z_2)
  Y_\Delta(\Delta(x_i^\pm),z_1).
\end{align*}
\end{lem}

%\begin{lem}
%\begin{align*}
%  \Sing_z Y_{\wh\ell}(x_i^\pm,z)\(\(x_i^\pm\)_0\)^kx_j^\pm
%  =\frac{c_{k+1}}{z-r_i(2k+a_{ij})\hbar}\(\(x_i^\pm\)_0\)^{k+1}x_j^\pm
%  \quad \te{for some }c_{k+1}\in \C[[\hbar]]^\times.
%\end{align*}
%\end{lem}

\begin{lem}\label{lem:Y-Delta-type2}
For $n\in\Z_+$, we have that
\begin{align}\label{eq:Y-Delta-n}
  &Y_\Delta^{12,34}(z_1)Y_\Delta^{34,56}(z_2)\cdots Y_\Delta^{2n-3,2n-2,2n-1,2n}(z_{n-1})\nonumber\\
  =&Y_{\wh\ell}^{12}(z_1)Y_{\wh\ell}^{23}(z_2)\cdots Y_{\wh\ell}^{n-1,n}(z_{n-1})
   Y_{\wh{\ell'}}^{n+1,n+2}(z_1)Y_{\wh{\ell'}}^{n+2,n+3}(z_2)\cdots Y_{\wh{\ell'}}^{2n-1,2n}(z_{n-1})\\
  &\times\prod_{a-b=n-1}S_{\ell,\ell'}^{a,n+b}(-z_b+z_a)
    \prod_{a-b=n-2}S_{\ell,\ell'}^{a,n+b}(-z_b+z_a)\cdots \prod_{a-b=1}S_{\ell,\ell'}^{a,n+b}(-z_b+z_a)\nonumber\\
  & \times
  \sigma^{n,n+1,n+2,\dots,2n-1}\sigma^{n-1,n,n+1,\dots,2n-3} \cdots\sigma^{23},\nonumber
\end{align}
where $\sigma^{a,a+1,\dots,b}=\sigma^{a,a+1}\sigma^{a+1,a+2}\cdots\sigma^{b-1,b}$ and $z_n=0$.
\end{lem}

\begin{proof}
By utilizing \eqref{eq:multqyb-hex1} and \eqref{eq:multqyb-hex2}, one can straightforwardly prove the lemma using induction on $n$.
\end{proof}

\begin{prop}\label{prop:normal-ordering-J+}
Let $n\in\Z_+$, and $i_1,\dots,i_n\in I$. For a subset $J$ of $\{1,2,\dots,n\}$,
we set
\begin{align*}
  u_{i_a}^{J}=\begin{cases}
    q^{2r\ell\partial}E_\ell(h_{i_a}),&\mbox{if }a\in J,\\
    x_{i_a}^+,&\mbox{if }a\not\in J,
  \end{cases}
  \quad\te{and}\quad
  v_{i_a}^{J}=\begin{cases}
              q^{2r\ell\partial}x_{i_a}^+,&\mbox{if }a\in J,\\
              \vac,&\mbox{if }a\not\in J.
            \end{cases}
\end{align*}
Assume $J=\{a_1<a_2<\cdots<a_k\}$ and the complementary set $\{b_1<b_2<\cdots<b_{n-k}\}$. We set
\begin{align*}
  &Y_J^+(z_1,\dots,z_n)
  =\:Y_{\wh\ell}(x_{i_{b_1}}^+,z_{b_1})Y_{\wh\ell}(x_{i_{b_2}}^+,z_{b_2})\cdots Y_{\wh\ell}(x_{i_{b_{n-k}}}^+,z_{b_{n-k}})\;\\
  &{Y_J'}^+(z_1,\dots,z_n)
  =\:Y_{\wh\ell'}(x_{i_{a_1}}^+,z_{a_1}+2r\ell\hbar)
  Y_{\wh\ell'}(x_{i_{a_2}}^+,z_{a_2}+2r\ell\hbar)\cdots Y_{\wh\ell'}(x_{i_{a_k}}^+,z_{a_k}+2r\ell\hbar)\;,\\
  &Y_{i_1,\dots,i_n,J,\Delta}^+(z_1,\dots,z_n)
  =\iota_{z_1,\dots,z_n}\prod_{1\le a<b\le n}
  f(z_a-z_b)^{-\delta_{i_a,i_b}+q^{-r_{i_a}a_{i_a,i_b}}} \\
  %\frac{ f(z_a-z_b-r_{i_a}a_{i_a,i_b}\hbar) }{f(z_a-z_b)^{\delta_{i_a,i_b}}}\\
  &\qquad\times Y_\Delta(u_{i_1}^{J}\ot v_{i_1}^{J},z_1)\cdots Y_\Delta(u_{i_n}^{J}\ot v_{i_n}^{J},z_n).
\end{align*}
Then
\begin{align*}
  &Y_{i_1,\dots,i_n,J,\Delta}^+(z_1,\dots,z_n)\vac\ot\vac
  =\iota_{z_1,\dots,z_n}f_{i_1,\dots,i_n,J,\hbar}^+(z_1,\dots,z_n)
  X_{i_1,\dots,i_n,J,\Delta}^+(z_1,\dots,z_n),
\end{align*}
where $X_{i_1,\dots,i_n,J,\Delta}^+(z_1,\dots,z_n)\in \(F_{\hat\g}^\ell\wh\ot F_{\hat\g}^{\ell'}\)((z_1,\dots,z_n))[[\hbar]]$ is defined by
\begin{align*}
  &\prod_{\substack{a<b\\ a\not\in J,b\in J}}(-1)^{\delta_{i_a,i_b}+1}
  \prod_{a\in J}\exp\(\wt h_{i_a}^-(z_a+2r\ell\hbar)\)Y_J^+(z_1,\dots,z_n)\vac \ot {Y_J'}^+(z_1,\dots,z_n)\vac
\end{align*}
and
\begin{align}
  f_{i_1,\dots,i_n,J,\hbar}^+(z_1,\dots,z_n)=\prod_{a\in J,b\not\in J}
  f(z_a-z_b)^{ -\delta_{i_a,i_b}+q^{-r_{i_a}a_{i_a,i_b} } }.
\end{align}
\end{prop}

\begin{proof}
From Lemma \ref{lem:multqyb-shift-total}, \eqref{eq:S-twisted-4} and \eqref{eq:S-E-5}, we have that
\begin{align*}
  &S_{\ell,\ell'}(z)(u_{i_a}^{J}\ot v_{i_b}^{J})=u_{i_a}^{J}\ot v_{i_b}^{J}\ot g(u_{i_a}^{J},v_{i_b}^{J},z),
\end{align*}
where
%\begin{align*}
%  g(u_{i_a}^{J},v_{i_b}^{J},z)
%  =\begin{cases}
%    g_{i_b,i_a,\hbar}(z-2r\ell\hbar)g_{i_b,i_a,\hbar}(z)\inv,& \mbox{if }a,b\in J,\\
%    g_{i_b,i_a,\hbar}(z-2r\ell\hbar),&\mbox{if }a\not\in J,\,b\in J,\\
%    1,&\mbox{otherwise}.
%  \end{cases}
%\end{align*}
\begin{align*}
  g(u_{i_a}^{J},v_{i_b}^{J},z)
  =\begin{cases}
    %\displaystyle\frac{f(z-(2r\ell+r_{i_a}a_{i_a,i_b})\hbar)f(z+r_{i_a}a_{i_a,i_b}\hbar)}
    %    {f(z-(2r\ell-r_{i_a}a_{i_a,i_b})\hbar)f(z-r_{i_a}a_{i_a,i_b}\hbar)}
        f(z)^{(q^{r_{i_a} a_{i_a,i_b} } -q^{-r_{i_a}a_{i_a,i_b}} )(1-q^{-2r\ell})}
        ,& \mbox{if }a,b\in J,\\
        f(z)^{ q^{-2r\ell}(q^{-r_{i_a}a_{i_a,i_b}}-q^{r_{i_a}a_{i_a,i_b}}) }
    %\displaystyle\frac{f(z-(2r\ell+r_{i_a}a_{i_a,i_b})\hbar)}
    %    {f(z-(2r\ell-r_{i_a}a_{i_a,i_b})\hbar)}
        ,&\mbox{if }a\not\in J,\,b\in J,\\
    1,&\mbox{otherwise}.
  \end{cases}
\end{align*}
From Lemma \ref{lem:Y-Delta-type2}, it follows that
\begin{align*}
  &Y_\Delta^{12,34}(z_1)\cdots Y_\Delta^{2n-1,2n,2n+1,2n+2}(z_n)
  \(u_{i_1}^{J}\ot v_{i_1}^{J}\ot \cdots \ot u_{i_n}^{J}\ot v_{i_n}^{J}\ot \vac\ot\vac\)\nonumber\\
  =&Y_{\wh\ell}^{12}(z_1)Y_{\wh\ell}^{23}(z_2)\cdots Y_{\wh\ell}^{n,n+1}(z_n)
  Y_{\wh{\ell'}}^{n+2,n+3}(z_1)
  Y_{\wh{\ell'}}^{n+3,n+4}(z_2)\cdots Y_{\wh{\ell'}}^{2n+1,2n+2}(z_n)\nonumber\\
  &\times
  \prod_{a-b=n-1}S_{\ell,\ell'}^{a,n+b}(-z_b+z_a)
  \prod_{a-b=n-2}S_{\ell,\ell'}^{a,n+b}(-z_b+z_a) \cdots \prod_{a-b=1}S_{\ell,\ell'}^{a,n+b}(-z_b+z_a)\nonumber\\
  &\quad\(u_{i_1}^{J}\ot \cdots\ot u_{i_n}^{J}\ot v_{i_1}^{J}\ot \cdots v_{i_n}^{J}\ot\vac\ot\vac\)\nonumber\\
  =&Y_{\wh\ell}^{12}(z_1)Y_{\wh\ell}^{23}(z_2)\cdots Y_{\wh\ell}^{n,n+1}(z_n)
  Y_{\wh{\ell'}}^{n+2,n+3}(z_1)
  Y_{\wh{\ell'}}^{n+3,n+4}(z_2)\cdots Y_{\wh{\ell'}}^{2n+1,2n+2}(z_n)\nonumber\\
  &\quad
  \(u_{i_1}^{J}\ot \cdots\ot u_{i_n}^{J}\ot\vac\ot v_{i_1}^{J}\ot \cdots v_{i_n}^{J}\ot\vac\)
  \prod_{b=1}^{n-1}\prod_{a=b+1}^n g(u_{i_a}^{J},v_{i_b}^{J},-z_b+z_a)\nonumber\\
  =&Y_{\wh\ell}(u_{i_1}^{J},z_1)\cdots Y_{\wh\ell}(u_{i_n}^{J},z_n)\vac
  \ot Y_{\wh{\ell'}}(v_{i_1}^{J},z_1)\cdots Y_{\wh{\ell'}}(v_{i_n}^{J},z_n)\vac\nonumber\\
  &\times \prod_{\substack{b<a\\ a,b\in J}}
  f(-z_b+z_a)^{(q^{r_{i_a} a_{i_a,i_b} } -q^{-r_{i_a}a_{i_a,i_b}} )(1-q^{-2r\ell})}\nonumber\\
  &\times\prod_{\substack{b<a\\ b\in J,a\not\in J}}
  f(-z_b+z_a)^{ q^{-2r\ell}(q^{-r_{i_a}a_{i_a,i_b}}-q^{r_{i_a}a_{i_a,i_b}}) }\nonumber\\
  %\frac{f(-z_b+z_a-(2r\ell+r_{i_a}a_{i_a,i_b})\hbar)}
  %      {f(-z_b+z_a-(2r\ell-r_{i_a}a_{i_a,i_b})\hbar)}\nonumber\\
  =&Y_{\wh\ell}(u_{i_1}^{J},z_1)\cdots Y_{\wh\ell}(u_{i_n}^{J},z_n)\vac
  \ot Y_{\wh{\ell'}}(v_{i_1}^{J},z_1)\cdots Y_{\wh{\ell'}}(v_{i_n}^{J},z_n)\vac\nonumber\\
  &\times \prod_{\substack{a<b\\ a,b\in J}}
  f(z_a-z_b)^{(q^{-r_{i_a} a_{i_a,i_b} } -q^{r_{i_a}a_{i_a,i_b}} )(1-q^{2r\ell})}
  \prod_{\substack{a<b\\ a\in J,b\not\in J}}
  f(z_a-z_b)^{(q^{r_{i_a} a_{i_a,i_b} } -q^{-r_{i_a}a_{i_a,i_b}} )q^{2r\ell}}
  %\frac{f(z_a-z_b+(2r\ell+r_{i_a}a_{i_a,i_b})\hbar)}
  %      {f(z_a-z_b+(2r\ell-r_{i_a}a_{i_a,i_b})\hbar)}
        \nonumber.
\end{align*}
The definition of normal ordering produces (see \eqref{eq:def-normal-ordering}) provides that
\begin{align*}
  &Y_{\wh{\ell'}}(v_{i_1}^{J},z_1)\cdots Y_{\wh{\ell'}}(v_{i_n}^{J},z_n)\vac
  = \iota_{z_1,\dots,z_n}\prod_{\substack{a<b\\ a,b\in J}}
  f(z_a-z_b)^{\delta_{i_a,i_b}-q^{-r_{i_a}a_{i_a,i_b}}}
  %\frac{f(z_a-z_b)^{\delta_{i_a,i_b}}}{f(z_a-z_b-r_{i_a}a_{i_ai_b}\hbar)}
   {Y_J'}^+(z_1,\dots,z_n)\vac.
\end{align*}
Using \eqref{eq:com-formulas-11}, \eqref{eq:com-formulas-12}, \eqref{eq:com-formulas-13}
and Proposition \ref{prop:Y-E}, we get that
\begin{align*}
  &Y_{\wh\ell}(u_{i_1}^{J},z_1)\cdots Y_{\wh\ell}(u_{i_n}^{J},z_n)\vac\\
  =&\prod_{a\in J}\exp\(\wt h_{i_a}^-(z_a+2r\ell\hbar)\)
  Y_J^+(z_1,\dots,z_n)
  \prod_{a\in J}\exp\(\wt h_{i_a}^+(z_a+2r\ell\hbar)\)\vac\\
  &\times
  \prod_{\substack{a<b\\ a,b\in J}}
  f(z_a-z_b)^{ (q^{r_{i_a}a_{i_a,i_b}}-q^{-r_{i_a}a_{i_a,i_b}})(1-q^{2r\ell})  }
  \prod_{\substack{a<b\\ a\in J,b\not\in J}}
  f(z_a-z_b)^{ (q^{-r_{i_a}a_{i_a,i_b}}-q^{r_{i_a}a_{i_a,i_b}})q^{2r\ell} }
  %\frac{f(z_a-z_b+(2r\ell-r_{i_a}a_{i_a,i_b})\hbar)}
  %      {f(z_a-z_b+(2r\ell+r_{i_a}a_{i_a,i_b})\hbar)}
  \\
  &\times \prod_{\substack{a<b\\ a,b\not\in J}}
  f(z_a-z_b)^{ \delta_{i_a,i_b}-q^{-r_{i_a}a_{i_a,i_b}} }
  \prod_{\substack{a<b\\ a\not\in J,b\in J}}
  f(z_a-z_b)^{q^{r_{i_a}a_{i_a,i_b}}-q^{-r_{i_a}a_{i_a,i_b}}}
  %\frac{f(z_a-z_b+r_{i_a}a_{i_a,i_b}\hbar)}
  %      {f(z_a-z_b-r_{i_a}a_{i_a,i_b}\hbar)}
        \\
  =&\prod_{a\in J}\exp\(\wt h_{i_a}^-(z_a+2r\ell\hbar)\)
  Y_J^+(z_1,\dots,z_n)\vac\\
  &\times
  \prod_{\substack{a<b\\ a,b\in J}}
  f(z_a-z_b)^{ (q^{r_{i_a}a_{i_a,i_b}}-q^{-r_{i_a}a_{i_a,i_b}})(1-q^{2r\ell})  }
  \prod_{\substack{a<b\\ a\in J,b\not\in J}}
  f(z_a-z_b)^{ (q^{-r_{i_a}a_{i_a,i_b}}-q^{r_{i_a}a_{i_a,i_b}})q^{2r\ell} }\\
  %\frac{f(z_a-z_b+(2r\ell-r_{i_a}a_{i_a,i_b})\hbar)}
  %      {f(z_a-z_b+(2r\ell+r_{i_a}a_{i_a,i_b})\hbar)}
  &\times
  \prod_{\substack{a<b\\ a,b\not\in J}}
  f(z_a-z_b)^{ \delta_{i_a,i_b}-q^{-r_{i_a}a_{i_a,i_b}} }
  %\frac{f(z_a-z_b)^{\delta_{i_a,i_b}}}{f(z_a-z_b-r_{i_a}a_{i_ai_b}\hbar)}
  \prod_{\substack{a<b\\ a\not\in J,b\in J}}
  f(z_a-z_b)^{q^{r_{i_a}a_{i_a,i_b}}-q^{-r_{i_a}a_{i_a,i_b}}}.
  %\frac{f(z_a-z_b+r_{i_a}a_{i_a,i_b}\hbar)}
  %      {f(z_a-z_b-r_{i_a}a_{i_a,i_b}\hbar)}.
\end{align*}
Combining these equations, we finally get that
\begin{align*}
  &Y_\Delta(u_{i_1}^{J}\ot v_{i_1}^{J},z_1)\cdots Y_\Delta(u_{i_n}^{J}\ot v_{i_n}^{J},z_n)\vac\ot\vac\\
  =&\prod_{a\in J}\exp\(\wt h_{i_a}^-(z_a+2r\ell\hbar)\)
  Y_J^+(z_1,\dots,z_n)\vac \ot {Y_J'}^+(z_1,\dots,z_n)\vac\\
  %&\times \prod_{\substack{a<b\\ a,b\in J}}\frac{f(z_a-z_b+(2r\ell+r_{i_a}a_{i_a,i_b})\hbar)f(z_a-z_b-r_{i_a}a_{i_a,i_b}\hbar)}
%        {f(z_a-z_b+(2r\ell-r_{i_a}a_{i_a,i_b})\hbar)f(z_a-z_b+r_{i_a}a_{i_a,i_b}\hbar)}\\
%  &\times\prod_{\substack{a<b\\ a\in J,b\not\in J}}\frac{f(z_a-z_b+(2r\ell+r_{i_a}a_{i_a,i_b})\hbar)}
%        {f(z_a-z_b+(2r\ell-r_{i_a}a_{i_a,i_b})\hbar)}\nonumber\\
%  &\times
%  \prod_{\substack{a<b\\ a,b\in J}}\frac{f(z_a-z_b+(2r\ell-r_{i_a}a_{i_a,i_b})\hbar)f(z_a-z_b+r_{i_a}a_{i_a,i_b}\hbar)}
%        {f(z_a-z_b+(2r\ell+r_{i_a}a_{i_a,i_b})\hbar)f(z_a-z_b-r_{i_a}a_{i_a,i_b}\hbar)}\\
  &\times
  \prod_{\substack{a<b\\ a,b\not\in J}}
  f(z_a-z_b)^{ \delta_{i_a,i_b}-q^{-r_{i_a}a_{i_a,i_b}} }
  %\frac{f(z_a-z_b)^{\delta_{i_a,i_b}}}{f(z_a-z_b-r_{i_a}a_{i_ai_b}\hbar)}
  %\prod_{\substack{a<b\\ a\in J,b\not\in J}}\frac{f(z_a-z_b+(2r\ell-r_{i_a}a_{i_a,i_b})\hbar)}
%        {f(z_a-z_b+(2r\ell+r_{i_a}a_{i_a,i_b})\hbar)}
  \prod_{\substack{a<b\\ a\not\in J,b\in J}}
  f(z_a-z_b)^{q^{r_{i_a}a_{i_a,i_b}}-q^{-r_{i_a}a_{i_a,i_b}}}
  %\frac{f(z_a-z_b+r_{i_a}a_{i_a,i_b}\hbar)}
  %      {f(z_a-z_b-r_{i_a}a_{i_a,i_b}\hbar)}
        \\
  &\times\prod_{\substack{a<b\\ a,b\in J}}
  f(z_a-z_b)^{ \delta_{i_a,i_b}-q^{-r_{i_a}a_{i_a,i_b}} }
  %\frac{f(z_a-z_b)^{\delta_{i_a,i_b}}}{f(z_a-z_b-r_{i_a}a_{i_ai_b}\hbar)}
  \\
  =&\prod_{a\in J}\exp\(\wt h_{i_a}^-(z_a+2r\ell\hbar)\)
  Y_J^+(z_1,\dots,z_n)\vac \ot {Y_J'}^+(z_1,\dots,z_n)\vac\\
  &\times
  \prod_{\substack{a<b\\ a\in J,b\not\in J}}
  f(z_a-z_b)^{ -\delta_{i_a,i_b}+q^{-r_{i_a}a_{i_a,i_b}} }
  %\frac{f(z_a-z_b-r_{i_a}a_{i_a,i_b}\hbar)}
  %      {f(z_a-z_b)^{\delta_{i_a,i_b}}}
  \prod_{\substack{a<b\\ a\not\in J,b\in J}}
  f(z_a-z_b)^{ -\delta_{i_a,i_b}+q^{r_{i_a}a_{i_a,i_b}} }
  %\frac{f(z_a-z_b+r_{i_a}a_{i_a,i_b}\hbar)}
  %      {f(z_a-z_b)^{\delta_{i_a,i_b}}}
        \\
  &\times\prod_{a<b}
  f(z_a-z_b)^{ \delta_{i_a,i_b}-q^{-r_{i_a}a_{i_a,i_b}} }
  %\frac{f(z_a-z_b)^{\delta_{i_a,i_b}}}{f(z_a-z_b-r_{i_a}a_{i_ai_b}\hbar)}
  \\
%%%%%%%%%%%%%%%%%%%%%%%
%  =&\prod_{a\in J}\exp\(\wt h_{i_a,\hbar}^-(z_a+2r\ell\hbar)\)
%  Y_J(z_1,\dots,z_n)\vac \ot Y_J'(z_1,\dots,z_n)\vac\\
%  &\times
%  \prod_{\substack{a<b\\ a\in J,b\not\in J}}\frac{f(z_a-z_b-r_{i_a}a_{i_a,i_b}\hbar)}
%        {f(z_a-z_b)^{\delta_{i_a,i_b}}}
%  \prod_{\substack{a<b\\ a\not\in J,b\in J}}(-1)^{\delta_{i_a,i_b}+1}\frac{f(z_b-z_a-r_{i_a}a_{i_a,i_b}\hbar)}
%        {f(z_b-z_a)^{\delta_{i_a,i_b}}}\\
%  &\times\prod_{a<b}
%  \frac{f(z_a-z_b)^{\delta_{i_a,i_b}}}{f(z_a-z_b-r_{i_a}a_{i_ai_b}\hbar)}\\
%%%%%%%%%%%%%%%%%%%%%
  =&\prod_{a\in J}\exp\(\wt h_{i_a}^-(z_a+2r\ell\hbar)\)
  Y_J^+(z_1,\dots,z_n)\vac \ot {Y_J'}^+(z_1,\dots,z_n)\vac\\
  &\times
  \prod_{\substack{a<b\\ a\not\in J,b\in J}}(-1)^{\delta_{i_a,i_b}+1}
  \prod_{a<b}
  f(z_a-z_b)^{ \delta_{i_a,i_b}-q^{-r_{i_a}a_{i_a,i_b}} }
  %\frac{f(z_a-z_b)^{\delta_{i_a,i_b}}}{f(z_a-z_b-r_{i_a}a_{i_ai_b}\hbar)}
  \prod_{a\in J,b\not\in J}
  f(z_a-z_b)^{ -\delta_{i_a,i_b}+q^{-r_{i_a}a_{i_a,i_b}} }.
  %\frac{f(z_a-z_b-r_{i_a}a_{i_a,i_b}\hbar)}
  %      {f(z_a-z_b)^{\delta_{i_a,i_b}}}.
\end{align*}
The proof of the proposition is now complete.
\end{proof}

\begin{prop}\label{prop:normal-ordering-J-}
Let $n\in\Z_+$, and $i_1,\dots,i_n\in I$. For a subset $J$ of $\{1,2,\dots,n\}$, we set
\begin{align*}
  u_{i_a}^J=\begin{cases}
              \vac, & \mbox{if }a\in J,\\
              x_{i_a}^-,&\mbox{if }a\not\in J,
            \end{cases}
  \quad\te{and}\quad
  v_{i_a}^J=\begin{cases}
              x_{i_a}^-, & \mbox{if }a\in J,\\
              \vac,&\mbox{if }a\not\in J.
            \end{cases}
\end{align*}
Assume $J=\{a_1<a_2<\cdots<a_k\}$ and the complementary set $\{b_1<b_2<\cdots<b_{n-k}\}$. We set
\begin{align*}
  &Y_J^-(z_1,\dots,z_n)
  =\:Y_{\wh\ell}(x_{i_{b_1}}^-,z_{b_1})Y_{\wh\ell}(x_{i_{b_2}}^-,z_{b_2})\cdots Y_{\wh\ell}(x_{i_{b_{n-k}}}^-,z_{b_{n-k}})\;\\
  &{Y_J'}^-(z_1,\dots,z_n)
  =\:Y_{\wh\ell'}(x_{i_{a_1}}^-,z_{a_1})
  Y_{\wh\ell'}(x_{i_{a_2}}^-,z_{a_2})\cdots Y_{\wh\ell'}(x_{i_{a_k}}^-,z_{a_k})\;,\\
  &Y_{i_1,\dots,i_n,J,\Delta}^-(z_1,\dots,z_n)
  =\iota_{z_1,\dots,z_n}\prod_{1\le a<b\le n}
  f(z_a-z_b)^{ -\delta_{i_a,i_b}+q^{-r_{i_a}a_{i_a,i_b}} }
  %\frac{ f(z_a-z_b-r_{i_a}a_{i_a,i_b}\hbar) }{f(z_a-z_b)^{\delta_{i_a,i_b}}}
  \\
  &\qquad\times Y_\Delta(u_{i_1}^{J}\ot v_{i_1}^{J},z_1)\cdots Y_\Delta(u_{i_n}^{J}\ot v_{i_n}^{J},z_n).
\end{align*}
Then
\begin{align*}
  &Y_{i_1,\dots,i_n,J,\Delta}^-(z_1,\dots,z_n)\vac\ot\vac
  =\iota_{z_1,\dots,z_n}f_{i_1,\dots,i_n,J,\hbar}^-(z_1,\dots,z_n)
  X_{i_1,\dots,i_n,J,\Delta}^-(z_1,\dots,z_n),
\end{align*}
where $X_{i_1,\dots,i_n,J,\Delta}^-(z_1,\dots,z_n)\in \(F_{\hat\g}^\ell\wh\ot F_{\hat\g}^{\ell'}\)((z_1,\dots,z_n))[[\hbar]]$ is defined by
\begin{align*}
  &\prod_{\substack{a<b\\ a\in J,b\not\in J}}(-1)^{\delta_{i_a,i_b}+1}
  Y_J^-(z_1,\dots,z_n)\vac \ot {Y_J'}^-(z_1,\dots,z_n)\vac
\end{align*}
and
\begin{align}
  f_{i_1,\dots,i_n,J,\hbar}^-(z_1,\dots,z_n)=\prod_{a\not\in J,b\in J}
  f(z_a-z_b)^{ -\delta_{i_a,i_b}+q^{-r_{i_a}a_{i_a,i_b}} }.
  %\frac{ f(z_a-z_b-r_{i_a}a_{i_a,i_b}\hbar) }{f(z_a-z_b)^{\delta_{i_a,i_b}}}.
\end{align}
\end{prop}

\begin{proof}
The proof is similar to, but simpler than the proof of Proposition \ref{prop:normal-ordering-J+}.
\end{proof}

\begin{prop}\label{prop:delta-serre+}
For $i,j\in I$ with $a_{ij}\le 0$ and $k\ge 0$, we have that
\begin{align*}
  &\(\Delta(x_i^+)\)_0^k\Delta(x_j^+)=(x_i^+)_0^kx_j^+\ot\vac\\
  +&\sum_{t=0}^{k}
    r_i^t[t]_{q^{r_i}}!(q^{r_i}-q^{-r_i})^t\binom{k}{t}_{q^{r_i}}\binom{k-1+a_{ij}}{t}_{q^{r_i}}\\
  &\times  \prod_{a=t+1}^k\exp\(\wt h_i^+((2r_i(k-a)+r_ia_{ij}+2r\ell)\hbar)\)
    \exp\(\wt h_j^+(2r\ell\hbar)\)\\
  &\times q^{(2k+a_{ij}-2t)r_i\partial}(x_i^+)_{-1}^t\vac\ot q^{2r\ell\partial}(x_i^+)_0^{k-t}x_j^+.
\end{align*}
\end{prop}

\begin{proof}
From Proposition \ref{prop:normal-ordering-rel-general}, we have that
\begin{align*}
  &Y_\Delta\(\(\Delta\(x_i^+\)\)_0^k \Delta\(x_j^+\),z\)
  =\:Y_\Delta\(\Delta\(x_i^+\),z+r_i(2(k-1)+a_{ij})\hbar\)\\
  \times& Y_\Delta\(\Delta\(x_i^+\),z+r_i(2(k-2)+a_{ij})\hbar\)
  \cdots Y_\Delta\(\Delta\(x_i^+\),z+r_ia_{ij}\hbar\)
  Y_\Delta\(\Delta\(x_j^+\),z\)\;.
\end{align*}
Fix a $(k+1)$-tuple $(i,i,\dots,i,j)$.
For any $J\subset\{1,2,\dots,k+1\}$, we see that
\begin{align}\label{eq:Delta-Serre-+temp1}
  &f_{i,\dots,i,j,J,\hbar}^+(z+r_i(2(k-1)+a_{ij})\hbar,\cdots,z+r_ia_{ij}\hbar,z)\ne 0
\end{align}
if and only if
\begin{align*}
  &\prod_{\substack{a\in J,b\not\in J\\ a,b\le k}}f(2r_i(b-a-1)\hbar)
  \prod_{k+1\in J,b\not\in J}f(-2r_i(k-b+a_{ij})\hbar)
  \prod_{a\in J,k+1\not\in J}f(2r_i(k-a)\hbar)\ne 0.
\end{align*}
So relation \eqref{eq:Delta-Serre-+temp1} holds only if
\begin{align*}
  J=\emptyset \quad\te{or}\quad J=\{t+1,\dots,k+1\}\quad\te{for some }0\le t\le k.
\end{align*}
For the case $J=\{t+1,\dots,k+1\}$, we have that
\begin{align*}
  &f_{i,\dots,i,j,J,\hbar}^+(z+r_i(2(k-1)+a_{ij})\hbar,\cdots,z+r_ia_{ij}\hbar,z)\\
  =&\prod_{b\le t< a\le k}\frac{f(2r_i(b-a-1)\hbar)}{f(2r_i(b-a)\hbar)}
  \prod_{b\le t}f(-2r_i(k-b+a_{ij})\hbar)\\
  =&\prod_{b=1}^t\frac{f(2r_i(b-k-1)\hbar)}{f(2r_i(b-t-1)\hbar)}\prod_{b\le t}f(-2r_i(k-b+a_{ij})\hbar)\\
  =&(-1)^t\prod_{b=1}^t\frac{f(2r_i(k+1-b)\hbar)}{f(2r_i(t-b)\hbar)}\prod_{b=1}^tf(2r_i(k-b+a_{ij})\hbar)\\
  =&(-r_i)^t[t]_{q^{r_i}}!(q^{r_i}-q^{-r_i})^t\binom{k}{t}_{q^{r_i}}\binom{k-1+a_{ij}}{t}_{q^{r_i}}.
\end{align*}
Then we get from Proposition \ref{prop:normal-ordering-J+} that
\begin{align*}
  &\:Y_\Delta\(\Delta\(x_i^+\),z+r_i(2(k-1)+a_{ij})\hbar\)
  \cdots Y_\Delta\(\Delta\(x_i^+\),z+r_ia_{ij}\hbar\)
  Y_\Delta\(\Delta\(x_j^+\),z\)\;\(\vac\ot\vac\)\\
  =&\sum_{J\subset\{1,2,\dots,k+1\}}Y_{J,\Delta}^+\(z+r_i(2(k-1)+a_{ij})\hbar,\cdots,z+r_ia_{ij}\hbar,z\)\(\vac\ot\vac\)\\
  =&\:Y_{\wh\ell}\(x_i^+,z+r_i(2(k-1)+a_{ij})\hbar\)\cdots
  Y_{\wh\ell}\(x_i^+,z+r_ia_{ij}\hbar\)Y_{\wh\ell}\(x_j^+,z\)\;\vac\ot\vac\\
  +&\sum_{t=0}^{k}
    r_i^t[t]_{q^{r_i}}!(q^{r_i}-q^{-r_i})^t\binom{k}{t}_{q^{r_i}}\binom{k-1+a_{ij}}{t}_{q^{r_i}}\\
  &\times q^{2r\ell\pd{z}} \prod_{a=t+1}^k\exp\(\wt h_i^+(z+(2r_i(k-a)+r_ia_{ij})\hbar)\)
    \exp\(\wt h_j^+(z)\)\\
  &\times q^{(2k+a_{ij}-2t)r_i\pd{z}}\:Y_{\wh\ell}\(x_i^+,z+2(t-1)r_i\hbar\)
  \cdots Y_{\wh\ell}\(x_i^+,z\)\;\vac\\
  &\quad\ot q^{2r\ell\pd{z}}\:Y_{\wh{\ell'}}\(x_i^+,z+(2r_i(k-t-1)+r_ia_{ij})\hbar\)
  \cdots Y_{\wh{\ell'}}\(x_i^+,z+r_ia_{ij}\hbar\)
    Y_{\wh{\ell'}}\(x_j^+,z\)\;\vac.
\end{align*}
Taking $z\to 0$, we complete the proof of proposition.
\end{proof}

Similarly, we have the following three results.

\begin{prop}\label{prop:delta-serre-}
For $i,j\in I$ with $a_{ij}\le 0$ and $k\ge 0$, we have that
\begin{align*}
  &\(\Delta(x_i^-)\)_0^k\Delta(x_j^-)=\vac\ot(x_i^-)_0^kx_j^-\\
  +&\sum_{t=0}^kr_i^t[t]_{q^{r_i}}!(q^{r_i}-q^{-r_i})^t\binom{k}{t}_{q^{r_i}}\binom{k-1+a_{ij}}{t}_{q^{r_i}}
   (x_i^-)_0^{k-t}x_j^-\ot q^{r_i(2(k-t)+a_{ij})\partial}(x_i^-)_{-1}^t\vac.
\end{align*}
\end{prop}

\begin{prop}\label{prop:delta-int+}
For $i\in I$ and $k\ge 0$, we have that
\begin{align*}
  &\(\Delta\(x_i^+\)\)_{-1}^k(\vac\ot\vac)
  =\sum_{t=0}^k f_0(2tr_i\hbar)\binom{k-1}{t}_{q^{r_i}}
  \binom{k}{t}_{q^{r_i}}\binom{k-1}{t}\inv\\
%    \prod_{a=1}^{k-1}F(ar_i)\prod_{a=1}^{t-1}F(ar_i)\inv\prod_{a=1}^{k-t-1}F(ar_i)\inv\\
  \times&\prod_{a=0}^{k-t-1}\exp\(\wt h_i^+(2(ar_i+r\ell)\hbar)\)
  q^{2r_i(k-t)\partial}\(x_i^+\)_{-1}^t\vac
  \ot q^{2r\ell\partial}\(x_i^+\)_{-1}^{k-t}\vac.
\end{align*}
\end{prop}

\begin{prop}\label{prop:delta-int-}
For $i\in I$ and $k\ge 0$, we have that
\begin{align*}
  &\(\Delta(x_i^-)\)_{-1}^k(\vac\ot\vac)\\
  =&\sum_{t=0}^k(x_i^-)_{-1}^{k-t}\vac
  \ot q^{2(k-t)r_i\partial}(x_i^-)_{-1}^t\vac
  \ot f_0(2tr_i\hbar)\binom{k-1}{t}_{q^{r_i}}
  \binom{k}{t}_{q^{r_i}}\binom{k-1}{t}\inv.
\end{align*}
\end{prop}

\noindent\emph{Proof of Proposition \ref{prop:Y-Delta-serre}:}
The equation \eqref{eq:prop-Y-Delta-locality} is a consequence of Lemma \ref{lem:normal-ordering-Delta}, while the equation \eqref{eq:prop-Y-Delta-serre} can be deduced from Propositions \ref{prop:delta-serre+} and \ref{prop:delta-serre-}, and the equation \eqref{eq:prop-Y-Delta-int} can be derived from Propositions \ref{prop:delta-int+} and \ref{prop:delta-int-}.

\bibliographystyle{unsrt}
%\bibliographystyle{alpha}

%\bibliography{../reference}

\begin{bibdiv}
\begin{biblist}

\bib{BJK-qva-BCD}{article}{
      author={Butorac, M.},
      author={Jing, N.},
      author={Ko{\v{z}}i{\'{c}}, S.},
       title={$\hbar$-adic quantum vertex algebras associated with rational
  ${R}$-matrix in types ${B}$, ${C}$ and ${D}$},
        date={2019},
     journal={Lett. Math. Phys.},
      volume={109},
       pages={2439\ndash 2471},
}

\bib{DF-qaff-RTT-Dr}{article}{
      author={Ding, J.},
      author={Frenkel, I.},
       title={Isomorphism of two realizations of quantum affine algebra
  $\mathcal{U}_q(\widehat{\mathfrak{gl}(n)})$},
        date={1993},
     journal={Comm. Math. Phys.},
      volume={156},
       pages={277\ndash 300},
}

\bib{Dr-new}{inproceedings}{
      author={{D}rinfeld, V.},
       title={A new realization of {Y}angians and quantized affine algebras},
        date={1988},
   booktitle={Soviet {M}ath. {D}okl},
      volume={36},
       pages={212\ndash 216},
}

\bib{EK-qva}{article}{
      author={Etingof, P.},
      author={Kazhdan, D.},
       title={Quantization of {L}ie bialgebras, {P}art {V}: {Q}uantum vertex
  operator algebras},
        date={2000},
     journal={Selecta Math.},
      volume={6},
      number={1},
       pages={105},
}

\bib{J-KM}{article}{
      author={Jing, N.},
       title={Quantum {K}ac-{M}oody algebras and vertex representations},
        date={1998},
     journal={Lett. Math. Phys.},
      volume={44},
      number={4},
       pages={261\ndash 271},
}

\bib{JKLT-Defom-va}{article}{
      author={Jing, N.},
      author={Kong, F.},
      author={Li, H.},
      author={Tan, S.},
       title={Deforming vertex algebras by vertex bialgebras},
        date={2024},
     journal={Comm. Cont. Math.},
      volume={26},
       pages={2250067},
}

\bib{JLM-qaff-RTT-Dr-BD}{article}{
      author={Jing, N.},
      author={Liu, M.},
      author={Molev, A.},
       title={Isomorphism between the ${R}$-matrix and {D}rinfeld presentations
  of quantum affine algebra: {T}ype ${B}$ and ${D}$},
        date={2020},
     journal={SIGMA},
      volume={16},
       pages={043},
}

\bib{JLM-qaff-RTT-Dr-C}{article}{
      author={Jing, N.},
      author={Liu, M.},
      author={Molev, A.},
       title={Isomorphism between the ${R}$-matrix and {D}rinfeld presentations
  of quantum affine algebra: {T}ype ${C}$},
        date={2020},
     journal={J. Math. Phys.},
      volume={61},
       pages={031701},
}

\bib{JYL-R-mat-DY}{article}{
      author={Jing, N.},
      author={Yang, F.},
      author={Liu, M.},
       title={Yangian doubles of classical types and their vertex
  representations},
        date={2020},
     journal={J. Math. Phys.},
      volume={61},
       pages={051704},
}

\bib{Kassel-topologically-free}{book}{
      author={Kassel, C.},
       title={Quantum groups, volume 155 of graduate texts in mathematics},
   publisher={Springer-Verlag, New York},
        date={1995},
}

\bib{K-Quantum-aff-va}{article}{
      author={Kong, F.},
       title={Quantum affine vertex algebras associated to untwisted quantum
  affinization algebras},
        date={2023},
     journal={Comm. Math. Phys.},
      volume={402},
       pages={2577\ndash 2625},
}

\bib{Kozic-qva-tri-A}{article}{
      author={Ko{\v{z}}i{\'{c}}, S.},
       title={On the quantum affine vertex algebra associated with
  trigonometric ${R}$-matrix},
        date={2021},
     journal={Selecta Math. (N. S.)},
      volume={27},
       pages={45},
}

\bib{K-qva-phi-mod-BCD}{article}{
      author={Ko{\v{z}}i{\'{c}}, S.},
       title={$\hbar$-adic quantum vertex algebras in types ${B}$, ${C}$, ${D}$
  and their $\phi$-coordinated modules},
        date={2021},
     journal={J. Phys. A: Math. Theor.},
      volume={54},
       pages={485202},
}

\bib{Li-nonlocal}{article}{
      author={Li, H.},
       title={Nonlocal vertex algebras generated by formal vertex operators},
        date={2006},
     journal={Selecta Math.},
      volume={11},
      number={3-4},
       pages={349},
}

\bib{Li-smash}{article}{
      author={Li, H.},
       title={A smash product construction of nonlocal vertex algebras},
        date={2007},
     journal={Comm. Cont. Math.},
      volume={9},
      number={05},
       pages={605\ndash 637},
}

\bib{Li-h-adic}{article}{
      author={Li, H.},
       title={{$\hbar$-adic quantum vertex algebras and their modules}},
        date={2010},
     journal={Comm. Math. Phys.},
      volume={296},
       pages={475\ndash 523},
}

\bib{LS-twisted-tensor}{article}{
      author={Li, H},
      author={Sun, J.},
       title={Twisted tensor products of nonlocal vertex algebras},
        date={2011},
        ISSN={0021-8693},
     journal={Journal of Algebra},
      volume={345},
      number={1},
       pages={266 \ndash  294},
  url={http://www.sciencedirect.com/science/article/pii/S002186931100425X},
}

\bib{Naka-quiver}{article}{
      author={Nakajima, H.},
       title={Quiver varieties and finite dimensional representations of
  quantum affine algebras},
        date={2001},
     journal={J. Amer. Math. Soc.},
      volume={14},
      number={1},
       pages={145\ndash 238},
}

\bib{RS-RTT}{article}{
      author={Reshetikhin, Y.},
      author={Semenov-{T}ian {S}hansky, A.},
       title={Central extensions of quantum current groups},
        date={1990},
     journal={Lett. Math. Phys.},
      volume={19},
       pages={133\ndash 142},
}

\bib{R-free-conformal-free-va}{article}{
      author={Roitman, M.},
       title={On free conformal and vertex algebras},
        date={1999},
     journal={J. Algebra},
      volume={217},
       pages={496\ndash 527},
}

\bib{S-iter-twisted-tensor}{article}{
      author={Sun, J.},
       title={Iterated twisted tensor products of nonlocal vertex algebras},
        date={2013},
        ISSN={0021-8693},
     journal={Journal of Algebra},
      volume={381},
       pages={233 \ndash  259},
  url={http://www.sciencedirect.com/science/article/pii/S0021869313000902},
}

\end{biblist}
\end{bibdiv}


\end{document}